\documentclass[11pt,a4paper,reqno]{amsart} 
\usepackage[marginparwidth=0pt,margin=20truemm]{geometry} 

\usepackage{mypackage} 
\usepackage{mycommand} 

\usetikzlibrary{decorations.markings} 
\usetikzlibrary{matrix,fit}

\usetikzlibrary{knots}
\usetikzlibrary{positioning}

\allowdisplaybreaks[2] 

\renewcommand{\iu}{\mathrm{i}} 

\DeclareMathOperator{\CS}{CS}
\DeclareMathOperator{\coe}{coe}

\makeatletter
\newcommand\@@vertwv[3]{\mathop{\ooalign{%
			$#1#3$\cr
			\hfil$#1\shortmid$\hfil\cr
			\hfil\raisebox{#2}{$#1\shortmid$}\hfil\cr}}}
\newcommand\@vertwv[1]{\mathchoice
	{\@@vertwv{\displaystyle     }{0.38ex}{#1}}%
	{\@@vertwv{\textstyle        }{0.38ex}{#1}}%
	{\@@vertwv{\scriptstyle      }{0.28ex}{#1}}%
	{\@@vertwv{\scriptscriptstyle}{0.21ex}{#1}}}
\newcommand\wedgemidvert{\@vertwv\wedge}
\newcommand\veemidvert{\@vertwv\vee  }
\makeatother

\renewcommand{\thepart}{\Roman{part}} 

\usepackage[pdfencoding=auto,hypertexnames=false]{hyperref} 
\hypersetup{colorlinks=false}

\numberwithin{equation}{section} 
\usepackage{mytheoremeng} 

\begin{document}

\title[Quantum modularity and asymptotic expansion conjecture]{A framework for proving quantum modularity: Application to Witten's asymptotic expansion conjecture}
\author[Y. Murakami]{Yuya Murakami}
\address{Faculty of Mathematics, Kyushu University, 744, Motooka, Nishi-ku, Fukuoka, 819-0395, JAPAN}
\email{murakami.yuya.896@m.kyushu-u.ac.jp}
\thanks{The author is supported by JSPS KAKENHI Grant Number JP 23KJ1675.}

\date{\today}


\begin{abstract}
	We address two linked problems at the interface of quantum topology and number theory: deriving asymptotic expansions of the Witten--Reshetikhin--Turaev invariants for 3-manifolds and establishing quantum modularity of false theta functions. 
	Previous progress covers Seifert homology 3-spheres for the former and rank-one cases for the latter, both of which rely on single-variable integral representations.
	We extend these results to negative definite plumbed 3-manifolds and to general false theta functions, respectively.
	We address this limitation by developing two techniques: a Poisson summation formula with signature and a framework of modular series, both of which enable a precise and explicit analysis of multivariable integral representations.
	As further applications, our method yields a unified approach to proving quantum modularity for false theta functions, indefinite theta functions, and for Eisenstein series of odd weight.
\end{abstract}



\maketitle
\tableofcontents


\section{Introduction} \label{sec:intro}


We investigate two linked problems:
\begin{itemize}
	\item \textbf{Topological side}: derive asymptotic expansions of quantum invariants.
	\item \textbf{Number-theoretic side}: establish quantum modularity of false theta functions.
\end{itemize}
To address these problems, we develop two techniques: A \emph{Poisson summation formula with signature} and \emph{modular series}.


\subsection{Topological side: Asymptotics of quantum invariants} \label{subsec:intro_topology}


A central problem in quantum topology is to study the asymptotics of quantum invariants. 
In this context, a key conjecture is the \emph{asymptotic expansion conjecture} for the Witten--Reshetikhin--Turaev (WRT) invariants of $ 3 $-manifolds, originally formulated by Witten~\cite{Witten}.
It is stated as follows.

\begin{conj}[{The asymptotic expansion conjecture, \cite[Conjecture 1.1]{Andersen-Himpel-Jorgensen-Martens-McLellan}}] \label{conj:asymptotic}
	Let $ M $ be a closed and oriented $ 3 $-manifold.
	For each positive integer $ k $, let $ Z_k(M) \in \bbC $ denote the $ \SU(2) $ WRT invariant for $ M $ at level $ k $.
	Let $ \CS_{\bbC} (M) $ denote the set of Chern--Simons invariants defined as the image of the Chern--Simons action
	from the $ \SL_2(\bbC) $ character variety of $ M $ to $ \bbC / \Z $.
	Then, for each Chern--Simons invariant $ \theta \in \CS_{\bbC} (M) $, there exists a Puiseux series 
	$ Z_\theta (k) \in \bbC[k^{1/p} \mid p \in \Z_{>0}][[k^{-1}]] $
	and we have an asymptotic expansion of WRT invariants:
	\begin{equation} \label{eq:asymp_conj}
		Z_k (M)
		\sim
		\sum_{\theta \in \CS_{\bbC} (M) \cup \{ 0 \}} e^{2\pi\iu k \theta} Z_\theta (k) 
		\quad \text{ as } k \to \infty.
	\end{equation}
\end{conj}

Here, the symbol $ \sim $ denotes the Poincar\'{e}'s notation of asymptotic expansions (\cref{dfn:asymptotic} \cref{item:dfn:asymptotic:Poincare}). 

There has been substantial progress on this conjecture~\cite{Andersen-Hansen,Andersen-Himpel,Andersen-Mistegard,Andersen-Petersen,Beasley-Witten,Charles,Chun,Chung:Seifert,Chung:rational,Chung:resurgent,Chung:SU(N),Charles-Marche:I,Charles-Marche:II,FIMT,Freed-Gompf,GMP,Hikami:Lattice,Hikami:Lattice2,Hansen-Takata,Jeffrey,Lawrence-Zagier,Matsusaka-Terashima,Rozansky:1,Rozansky:2,Wheeler,Wu}.
At present, it has been proved for lens spaces, Seifert homology $3$-spheres, and finite order mapping tori.

Our first main result establishes asymptotic formulas for a broader class of $3$-manifolds, recovering the known formulas for lens spaces and Seifert homology $3$-spheres as special cases.

\begin{thm}
	\label{thm:main_WRT_asymptotic}
	Let $ M $ be a negative definite plumbed manifold.
	Then, there exists a finite set $ \calS \subset \Q/\Z $ and Puiseux series $ Z_\theta (k) \in \bbC((k^{-1/2})) $ for each $ \theta \in \calS $
	such that
	\[
	Z_k (M)
	\sim
	\sum_{\theta \in \calS} e^{2\pi\iu k \theta} Z_\theta (k) 
	\quad \text{ as } k \to \infty.
	\]
	Moreover, $ \calS $ is described explicitly as in \cref{cor:WRT_asymp}.
\end{thm}


Our proof strategy follows the method of modular transformation established by Lawrence--Zagier~\cite{Lawrence-Zagier}.
We summarize the strategy schematically in the following diagram:

\begin{equation}
	\begin{tikzcd}[row sep=4em]
		Z_k(M)
		\arrow[d, "k \to \infty" ', "(4)", dashed]
		&[5em]
		\widehat{Z}(q; M)
		\arrow[l, "\tau \to 1/k" ', "(1)"]
		\arrow[d, "\tau \mapsto -1/\tau", "(2)" '] 
		\\
		\displaystyle \sum_{\theta \in \Q/\Z} e^{2\pi\iu k \theta} Z_\theta (k) 
		&
		\displaystyle \widehat{Z}^*(q) + \int \widehat{Z}^{**} (\tau; \xi) d\xi
		\arrow[l, "{\tau = -k, \, k \to \infty}", "(3)" ']
	\end{tikzcd}				
\end{equation}

We aim to establish the left arrow (4).
We resolve this by traversing the diagram clockwise along the top, right, and bottom edges, following (1), (2), and (3).
Here, $ \widehat{Z}(q; M) $ in the upper-right corner is the $ q $-series known as the \emph{Gukov--Pei--Putrov--Vafa \textup{(}GPPV\textup{)} invariants} of negative definite plumbed manifolds $ M $ defined in \cite{GPPV}, where we set $ q = e^{2\pi\iu \tau} $.
The top arrow (1) expresses the \emph{radial limit conjecture}, which was conjectured in this case by Gukov--Pei--Putrov--Vafa~\cite{GPPV} and proved by Murakami~\cite{M:GPPV} (see \cref{subsec:GPPV_radial_limits}).
The right arrow (2) expresses \emph{modular transformation} for GPPV invariants.

To establish (2) and (3), we need to prove modular transformation formulas and asymptotic expansions for a broader class of false theta functions, including the GPPV invariants.
These purely number-theoretic problems are resolved in \cref{thm:main_false_theta,thm:main_asymptotic_false_theta} below.


\subsection{Number-theoretic side: Quantum modularity of false theta functions} \label{subsec:intro_number_theory}


Quantum modular forms, introduced by Zagier~\cite{Zagier:quantum}, constitute an important class in number theory.
A central example is provided by rank-one false theta functions.
For higher-rank false theta functions, however, quantum modularity has not yet been established, despite significant advances~\cite{BKM,BKM:vector,BKMN:rk2_false_modular,BKMN:False_modular,BMM:high_depth,Males}.
In this paper, we establish quantum modularity for  higher-rank false theta functions in full generality, defined as
\[
\widetilde{\Theta} (\tau)
=
\sum_{m \in \alpha + \Z^r} \sgn(Am) P(m) q^{\transpose{m} Sm /2},
\]
where $ r $ and $ s $ are positive integers, $ S \in \Sym_r^+ (\Q) $ is a positive definite symmetric matrix,
$ A \in \Mat_{s, r} (\Z) $ is a matrix of rank $ d \ge 0 $, 
$ P(x_1, \dots, x_r) \in \Q[x_1, \dots, x_r] $ is a polynomial, $ \alpha \in \Q^r $ is a vector, and
\[
\sgn(x) 
\coloneqq
\begin{cases}
	1 & \text{ if } x \ge 0, \\
	-1 & \text{ if } x < 0,
\end{cases}
\quad
\sgn(x_1, \dots, x_r)
\coloneqq
\sgn(x_1) \cdots \sgn(x_r).
\]
We call such functions \emph{false theta function of rank $ r $ and depth $ \le d $}.

Our main theorem is stated as follows.

\begin{thm} \label{thm:main_false_theta}
	Let $ \widetilde{\Theta}(\tau) $ be a false theta function of rank $ r $ and depth $ \le d $ 
	attached to a positive definite symmetric matrix $ S \in \Sym_r^+ (\Q) $.
	Then, there exist finitely many false theta functions 
	$ \widetilde{\Theta}_1 (\tau), \dots, \widetilde{\Theta}_M (\tau) $ of rank $ \le r $ and depth $ \le d $ and 
	holomorphic functions $ \Omega(\tau), \Omega_1 (\tau), \dots, \Omega_M (\tau) $ on the universal cover of $ \bbC \smallsetminus \{ 0 \} $ such that
	\[
	\sqrt{\det S} {\sqrt{\frac{\iu}{\tau}}}^{\, r} \cdot \widetilde{\Theta} \left( -\frac{1}{\tau} \right)
	=
	\sum_{1 \le i \le M} \widetilde{\Theta}_i (\tau) \Omega_i (\tau) 
	+ \Omega(\tau).
	\]
\end{thm}

We give $ \widetilde{\Theta}_i (\tau) $, $ \Omega_i (\tau) $, and $ \Omega (\tau) $ explicitly in \cref{thm:false_theta_rk2_modular,thm:false_theta}.
As a corollary of \cref{thm:main_false_theta}, we obtain the modular transformation formula for GPPV invariants.

We also establish the following asymptotic formulas for false theta functions as follows.

\begin{thm} \label{thm:main_asymptotic_false_theta}
	Let $ \widetilde{\Theta} (\tau) $ be a false theta function.
	\begin{enumerate}
		\item \label{item:thm:main_asymptotic_false_theta:1}
		For any rational number $ \rho \in \Q $, we have
		\[
		\widetilde{\Theta} (\tau + \rho)
		\sim
		\sum_{j \in \frac{1}{2} \Z}
		c_j (\rho) \tau^j
		\quad \text{ as } \tau \to 0
		\]
		for some $ c_j (\rho) \in \bbC $, which vanishes for any sufficiently small $ j $.
		
		\item \label{item:thm:main_asymptotic_false_theta:2}
		In the case where $ \rho = h/k \neq 0 $ for fixed $ h \in \Z \smallsetminus \{ 0 \} $, for any $ j $ we have
		\[
		c_j (\rho)
		\sim
		\sum_{\theta \in \calS} \bm{e} \left( -\frac{\theta}{\rho} \right) \widetilde{\varphi}_{\theta, h, j} (k)
		\quad \text{ as } k \to \infty
		\]
		for some finite set $ \calS \subset \Q / \Z $ independent of $ \rho $ and $ j $ and 
		some formal power series $ \widetilde{\varphi}_{\theta, h, j} (k) \in \bbC((k^{1/2})). $
	\end{enumerate}
\end{thm}

We write $ \calS $ explicitly later in \cref{thm:false_theta_asymp} \cref{item:thm:false_theta_asymp:3}.
It can be written as a sum of values of quadratic forms.


\subsection{A framework for proving quantum modularity} \label{subsec:intro_modular_series}


To prove the modular transformation formulas for false theta functions in \cref{thm:main_false_theta}, we derive a \emph{Poisson summation formula with signature}. 
This contrasts with the classical proof of the modular transformation of theta functions, which relies on the Poisson summation formula.

\begin{thm}[Poisson summation formula with signature] \label{thm:main_PSF_sgn}
	Let $ r $ be a positive integer and
	$ f \colon \R^r \to \bbC $ be a continuous function with exponential decay, that is, 
	there exist $ a, K > 0 $ such that 
	$ \abs{f(x)} \le Ke^{-a\abs{x}} $ for any $ x \in \R^r $.
	Then, we have
	\[
	\sum_{m \in \Z^r} \sgn(m) f(m)
	=
	\sum_{I \subset \{ 1, \dots, r \}} (-1)^{\abs{I}} 2^{r - \abs{I}}
	\sum_{n_I^{} \in \Z^I} \sgn(n_I^{})
	\int_{C_+^{I^\complement}} \widehat{f}(n_I^{}, \xi_{I^\complement}^{})
	\prod_{j \in I^\complement} \frac{d\xi_j}{1 - e^{2\pi\iu \xi_j}},
	\]
	if the right hand side absolutely converges.
	Here, we use the following notation\textup{:}
	\begin{itemize}
		\item For a subset $ I \subset \{1, \dots, r\} $, we denote $ I^\complement \coloneqq \{ 1, \dots, r \} \smallsetminus I $.
		\item for a variable $ x = (x_1, \dots, x_r) $, we denote 
		$ x_{I} \coloneqq (x_i)_{i \in I} $ and $ x_{I^\complement}^{} \coloneqq (x_j)_{j \in I^\complement} $.
		\item Denote by $ C_+ $ the integration path shown in \cref{fig:main_C_+}.
		\item Denote by $ \widehat{f} (\xi) $ the Fourier transform of $ f(x) $.
	\end{itemize}
\end{thm}

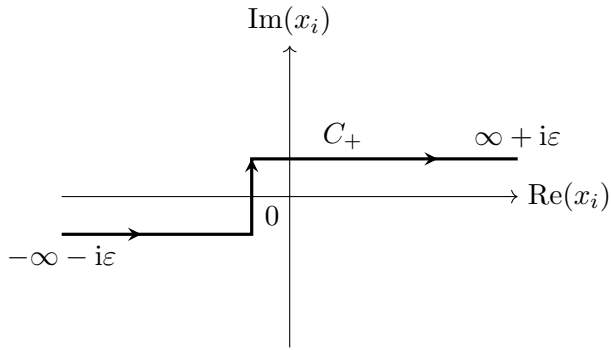
\begin{figure}[htbp]
	\centering
	\begin{tikzpicture}
		\tikzset{midarrow/.style={postaction={decorate},
				decoration={markings,
					mark=at position 0.15 with {\arrow{stealth}},
					mark=at position 0.5 with {\arrow{stealth}},
					mark=at position 0.85 with {\arrow{stealth}}
				}
			}
		}
		\draw[->] (-3, 0) -- (3, 0) node[right]{$ \Re (x_i) $};
		\draw[->] (0, -2) -- (0, 2) node[above]{$ \Im (x_i) $};
		
		\draw (0,0) node[below left]{0};
		
		\draw[very thick, midarrow] (-3,-0.5) node[below]{$ -\infty - \iu \varepsilon $}
		--(-0.5,-0.5)--(-0.5,0.5)
		--(3,0.5) node[above]{$ \infty + \iu \varepsilon $};
		
		\node at (0.7,0.8) {$C_+$};
		
	\end{tikzpicture}
	\caption{The integration path $ C_+ $}
	\label{fig:main_C_+}
\end{figure}

In the case when $ r=1 $, \cref{thm:main_PSF_sgn} implies that
\[
\sum_{m \in \Z} \sgn(m) f(m)
=
-\sum_{n \in \Z} \sgn(n) \widehat{f}(n)
+ 2 \int_{C_+} \widehat{f}(\xi) \frac{d\xi}{1 - e^{2\pi\iu \xi}}.
\]

By applying \cref{thm:main_PSF_sgn} to Gaussians $ f(x) = e^{-\pi \transpose{x} Sx} $, we obtain the modular transformation formula (\cref{thm:main_false_theta}) for false theta functions.
In this case, the integrals along $ C_+ $ have analytic continuation, and thus we obtain quantum modularity.

The GPPV invariants are false theta functions whose summands have complicated polynomial factors.
To derive modular transformations and asymptotic formulas for these series, it is essential to organize these factors uniformly and systematically.
Motivated by this, we introduce a framework of \emph{modular series}.

Modular series are a type of holomorphic functions on the upper half plane $ \bbH $ which are written as
\[
\sum_{m \in \Z^r} g(m) \widehat{\gamma}(\tau; m)
\]
for some map $ g \colon \Z^r \to \bbC $ called a (\emph{partial}/\emph{false}) \emph{quasi-polynomial} and some continuous map $ \gamma \colon \bbH \times \R^r \to \bbC $.
Typical examples of modular series are
false theta functions for a periodic map $ \chi \colon \Z/N\Z \to \bbC $ defined as
\[
\widetilde{\theta} (\tau; \chi) \coloneqq
\sum_{n \in \Z} \sgn(n) \chi(n) q^{n^2/2N}
\]
and Eisenstein series of weight $ k \in 2\Z_{>0} $ defined as
\[
E_k(\tau) \coloneqq
1 - \frac{2k}{B_k} \sum_{d=1}^{\infty} d^{k-1} \frac{q^d}{1 - q^d}.
\]
For modular series, we prove their modular transformation formulas and asymptotic formulas.

A framework of modular series has the following interesting aspects.

\begin{itemize}[nosep]	
	\item \textbf{Refined understanding of modular transformation formulas.}
	
	A framework of modular series enables a detailed analysis of the functions appearing in the modular transformation formulas (\cref{prop:modular_series_modular,thm:false_modular_series_modular,thm:false_theta}). 
	This is essential for obtaining explicit formulas for the finite set $ \calS \subset \Q/\Z $ in \cref{thm:main_asymptotic_false_theta,thm:main_WRT_asymptotic}.
	
	\item \textbf{Providing asymptotic expansions.}
	
	In the study of asymptotic expansions of quantum invariants and quantum modular forms, two types of asymptotic formulas are commonly used: those of Euler--Maclaurin type (\cref{prop:asymp_EM}) and those of stationary phase type (\cref{prop:asymp_stationary_phase}).
	Within a framework of modular series, we generalize both types of formulas.
	
	
	\item \textbf{Understanding GPPV invariants.}
	
	The GPPV invariants are defined via Cauchy principal value integrals that encode topological data in an organized combinatorial form.
	Although this is not apparent from the definition, they are in fact  false theta functions.
	A framework of modular series clarifies and uniformly extends this construction.
	
	\item \textbf{A unified proof of modular transformation formula for various $ q $-series.}
	
	Theta functions and Eisenstein series are two typical examples of modular forms.
	Their modularity is usually proved by different methods.
	Within a framework of modular series, we provide a unified proof of modularity for both of them by the Poisson summation formula.
	In particular, we prove in a unified manner the modular transformation formula for the  Eisenstein series of weight $ 2 $, which is a quasi-modular form.
	Moreover, we give a unified proof of quantum modularity for false theta functions and Eisenstein series of odd weight.
	This method systematically treats deformations such as inserting coefficients into the summands of theta functions, or deforming an Eisenstein series into a Poincar\'{e} series.
	
	\item \textbf{Quantum modularity arises from parity.}
	
	{\renewcommand{\arraystretch}{1.2} 
		\begin{table}[hptb]
			\centering
			\begin{tabular}{cc}
				\specialrule{.1em}{.05em}{.05em} 
				Modular forms & Quantum modular forms \\
				\hline
				Theta functions  & False theta functions \\
				Eisenstein series of even weight \quad & Eisenstein series of odd weight \\
				\specialrule{.1em}{.05em}{.05em} 
			\end{tabular}
			\caption{Modularity v.s.~quantum modularity.}
			\label{tab:modular_vs_quantum_modular}
		\end{table}
	}
	
	\cref{tab:modular_vs_quantum_modular} lists representative examples of modular forms and quantum modular forms.
	This raises the question of the origin of the distinction between modularity and quantum modularity.
	We advance the following principle:
	\begin{quote}
		\centering
		\emph{Quantum modularity arises from parity.}
	\end{quote}
	Moreover, for modular series, we can automatically define their ``false versions'' and we can prove their modular transformation formulas and asymptotic expansion formulas.
	Our framework indicates that Eisenstein series of odd weight ought to be regarded as ``false versions'' of Eisenstein series of even weight.
	
	\item \textbf{The reason why the error function appears in Zwegers' indefinite theta functions}.
	
	Zwegers~\cite{Zwegers:thesis} proved the modular transformation formulas for Ramanujan's mock theta functions by introducing his indefinite theta functions.
	In his construction, the error function appears in a seemingly ad hoc manner.
	Our framework explains this phenomenon: when we express indefinite theta functions as modular series, error functions arise naturally in their modular transformations.
	As a consequence, we also give an alternative proof of Zwegers' modular transformation formulas in \cref{subsec: indefinite_theta_Zwegers_modular_series}.
	
	\item \textbf{Toward a unified framework for various indefinite theta functions.}
	
	Zwegers' treatment of indefinite theta functions assumes that vectors appearing in the sign factors lie in the negative cone of a quadratic form.
	We remove this assumption and, within a framework of modular series, prove that general indefinite theta functions of type $ (r-1,1) $ admit modular transformation formulas and quantum modularity (\cref{thm:indef_theta_mod_trans_further_case}) parallel to those of Zwegers. 
	Unlike the method of modular completion---which requires identifying suitable error function regularizations---our method reduces the argument to routine computation of Fourier transforms.
	
\end{itemize}


\subsection*{Outline of the paper}


This paper will be organized as follows. 

In \cref{part:modular_series}, we establish a framework for proving quantum modularity.
In \cref{sec:PSF_sgn}, we prove a Poisson summation formula with signature (\cref{thm:main_PSF_sgn}).
In \cref{sec:modular_series_framework}, we define modular series and prove its basic properties.
In \cref{sec:modular_series_modular}, we prove modular transformation formula for modular series, which is a variation of Poisson summation formula with signature (\cref{thm:main_PSF_sgn}).

In \cref{part:quantum_modularity_asymptotics}, we apply results in \cref{part:modular_series} for false theta functions, indefinite theta functions in \cref{sec:indefinite_false}, and Eisenstein series in \cref{sec:Eisenstein}.
We prove \cref{thm:main_false_theta,thm:main_asymptotic_false_theta} in \cref{sec:false_theta} 
and modular transformation formulas and quantum modularity for general indefinite theta functions of type $ (r-1,1) $ in \cref{subsec: indefinite_theta_Zwegers_outside}.

Finally, in \cref{part:WRT}, we apply \cref{thm:main_false_theta,thm:main_asymptotic_false_theta} in \cref{sec:false_theta}, for GPPV invariants and prove \cref{thm:main_WRT_asymptotic}.



\subsection*{Notations}


Throughout this article, we use the following notation:

\begin{itemize}
	\item We denote by $ r $ a positive integer, representing the number of variables.
	\item We denote by $ q $ a complex variable with $ \abs{q} < 1 $.
	\item We denote by $ k $ a positive integer.
	\item Let $ \zeta_k \coloneqq e^{2\pi\iu/k} $.
	\item For a complex number $ z $, we denote by $ \bm{e}(z) \coloneqq e^{2\pi\iu z} $.
	\item Let $ \bbH \coloneqq \{ \tau \in \bbC \mid \Im(\tau) > 0 \} $.
	\item We denote by $ \tau $ a variable of $ \bbH $.
	\item $ q \coloneqq \bm{e} (\tau) $.
	\item $ \widetilde{q} \coloneqq \bm{e} (-1/\tau) $.
	\item We denote by $ \calO(\bbH) $ the set of holomorphic functions on $ \bbH $. 
	\item For a real number $ x $, let
	\begin{align}
		\sgn(x) \coloneqq 
		\begin{cases}
			1 & x \ge 0, \\
			-1 & x < 0,
		\end{cases}
		\qquad
		\sgn_0(x) \coloneqq 
		\begin{cases}
			1 & x > 0, \\
			0 & x = 0, \\
			-1 & x < 0.
		\end{cases}
	\end{align}
	\item $ \sgn(x_1, \dots, x_r) \coloneqq \sgn(x_1) \cdots \sgn(x_r) $.
	\item The Cauchy principal values
	\[
	\PV \int_{\abs{z} = 1}
	\coloneqq
	\frac{1}{2} \lim_{\varepsilon \to +0} \left( \int_{\abs{z} = 1+ \varepsilon} + \int_{\abs{z} = 1 - \varepsilon} \right),
	\quad
	\PV \int_{\R}
	\coloneqq
	\frac{1}{2} \lim_{\varepsilon \to +0} \left( \int_{\R + \iu \varepsilon} + \int_{\R - \iu \varepsilon} \right).
	\]
	\item For a set $ A $, we denote by $ \bm{1}_A $ its characteristic function. 
	\item We denote the image of $ g $ under an operator $ \calT $ by $ \calT g $ or $ \calT[g] $. 
	\item For $ e = (e_1, \dots, e_r) \in \{ \pm 1 \}^r $ and $ x = (x_1, \dots, x_r) \in \R^r $, 
	define $ ex \coloneqq (e_1 x_1, \dots, e_r x_r) $.
\end{itemize}


\subsection*{Acknowledgement}


The author would like to express the greatest appreciation to Masanobu Kaneko and Takuya Yamauchi for their valuable advice. 
Toshiki Matsusaka provided the author with valuable advice and introduced me to numerous related studies.
The author would like to thank Koji Hasegawa, Kazuhiro Hikami, William Misteg\aa{}rd, and Yuji Terashima for giving many comments. 
Keito Akiyama provided the author with advice on a proof of \cref{lem:ker_func}.
The author is supported by JSPS KAKENHI Grant Number JP 20J20308.


\part{A framework for ``modular series'' and Poisson summation formula with signature} \label{part:modular_series}



\section{Poisson summation formula with signature} \label{sec:PSF_sgn}


In this section, we prove \cref{thm:main_PSF_sgn}.
Stronger than \cref{thm:main_PSF_sgn}, the following result holds.

\begin{thm}[Poisson summation formula with signature] \label{thm:PSF_sgn}
	Let $ r $ be a positive integer and
	$ f \colon \R^r \to \bbC $ be a continuous function with exponential decay.
	Then, for any $ e \in \{ \pm 1 \}^r $, we have
	\begin{align}
		\sum_{m \in \Z^r} \sgn(m) f(m)
		&=
		2^r \PV \int_{\R^r} \widehat{f}(\xi) \prod_{1 \le i \le r} \frac{d\xi_i}{1 - \bm{e}(\xi_i)}
		\\
		&=
		\sum_{I \subset \{1, \dots, r\}} \left( \prod_{i \in I} (-e_i) \right)
		\sum_{n_I^{} \in \Z^I} \sgn(n_I^{})
		\left( 
		\prod_{j \in I^\complement} 2 \int_{C_{e_j}}
		\frac{d\xi_j}{1 - \bm{e}(\xi_j)}
		\right)
		\widehat{f}(n_I^{}, \xi_{I^\complement}^{}),
	\end{align}
	if the right hand sides absolutely converge.
	Here, we define the integration paths $ C_\pm $ as in \cref{fig:C+,fig:C-}.
\end{thm}

\begin{figure}[htbp]
	\begin{minipage}[b]{0.45\columnwidth}
		\centering
		\begin{tikzpicture}
			\tikzset{midarrow/.style={postaction={decorate},
					decoration={markings,
						mark=at position 0.15 with {\arrow{stealth}},
						mark=at position 0.5 with {\arrow{stealth}},
						mark=at position 0.85 with {\arrow{stealth}}
					}
				}
			}
			\draw[->] (-3, 0) -- (3, 0) node[right]{$ \Re (x_i) $};
			\draw[->] (0, -2) -- (0, 2) node[above]{$ \Im (x_i) $};
			
			\draw (0,0) node[below left]{0};
			
			\draw[very thick, midarrow] (-3,-0.5) node[below]{$ -\infty - \iu \varepsilon $}
			--(-0.5,-0.5)--(-0.5,0.5)
			--(3,0.5) node[above]{$ \infty + \iu \varepsilon $};
			
			\node at (0.7,0.8) {$C_+$};
			
		\end{tikzpicture}
		\caption{The integration path $ C_+ $}
		\label{fig:C+}
	\end{minipage}
	\hspace{0.04\columnwidth} 
	\begin{minipage}[b]{0.45\columnwidth}
		\centering
		\begin{tikzpicture}
			\tikzset{midarrow/.style={postaction={decorate},
					decoration={markings,
						mark=at position 0.15 with {\arrow{stealth}},
						mark=at position 0.5 with {\arrow{stealth}},
						mark=at position 0.85 with {\arrow{stealth}}
					}
				}
			}
			\draw[->] (-3, 0) -- (3, 0) node[right]{$ \Re (x_i) $};
			\draw[->] (0, -2) -- (0, 2) node[above]{$ \Im (x_i) $};
			
			\draw (0,0) node[above left]{0};
			
			\draw[very thick, midarrow] (-3,0.5) node[above]{$ -\infty - \iu \varepsilon $}
			--(-0.5,0.5)--(-0.5,-0.5)
			--(3,-0.5) node[below]{$ \infty + \iu \varepsilon $};
			
			\node at (0.7,-0.8) {$C_-$};
			
		\end{tikzpicture}
		\caption{The integration path $ C_- $}
		\label{fig:C-}
	\end{minipage}
\end{figure}

\begin{proof}
	Since $ f(x) $ decays exponentially, 
	there exist $ a, K > 0 $ such that 
	$ \abs{f(x)} \le Ke^{-a\abs{x}} $ for any $ x \in \R^r $.
	Thus, the Fourier transform $ \widehat{f}(\xi) $ extends holomorphically to a strip $ \R^r + \iu (-a/2\pi, a/2\pi)^r $.
	
	Let $ 0 < \varepsilon < a/2\pi $ be arbitrary.
	In the case when $ r=1 $, the first equality can be verified by
	\begin{align}
		\sum_{m \in \Z} \sgn(m) f(m)
		&=
		\left( \sum_{m \ge 1} - \sum_{m \le -1} \right)
		\int_{\R} \widehat{f}(\xi) \bm{e} (m\xi) d\xi
		\\
		&=
		\sum_{m \ge 0} \int_{\R + \iu \varepsilon} \widehat{f}(\xi) \bm{e} (m\xi) d\xi
		- \sum_{m \le -1} \int_{\R - \iu \varepsilon} \widehat{f}(\xi) \bm{e} (m\xi) d\xi
		\\
		&=
		\int_{\R + \iu \varepsilon} \widehat{f}(\xi) \frac{d\xi}{1 - \bm{e} (\xi)}
		- \int_{\R - \iu \varepsilon} \widehat{f}(\xi) \frac{\bm{e} (-\xi)}{1 - \bm{e} (-\xi)} d\xi
		\\		
		&=
		2 \PV \int_{\R} \widehat{f}(\xi) \frac{d\xi}{1 - \bm{e}(\xi)}.
	\end{align}
	The first equality can be established for general $ r $ by a similar argument.
	
	Using
	\[
	2 \PV \int_{\R} d\xi_i
	=
	2 \int_{C_{\pm}} d\xi_i 
	\pm 2\pi \iu \sum_{\alpha \in \R} \sgn(\alpha) \Res_{\xi_i = \alpha},
	\]
	we obtain the second equality as follows:
	\begin{align}
		&\phant
		2^r \PV \int_{\R^r} \widehat{f}(\xi) \prod_{1 \le i \le r} \frac{d\xi_i}{1 - \bm{e}(\xi_i)}
		\\
		&=
		\left( \prod_{1 \le i \le r} \left(
			2 \int_{C_{e_i}} d\xi_i 
			+ 2\pi \iu e_i \sum_{n_i \in \Z} \sgn(n_i) \Res_{\xi_i = n_i}
		\right) \right)
		\left[ \widehat{f}(\xi) \prod_{1 \le i \le r} \frac{1}{1 - \bm{e}(\xi_i)} \right]
		\\
		&=
		\sum_{I \subset \{1, \dots, r\}}
		\left( \prod_{i \in I} e_i 2\pi \iu \sum_{n_i \in \Z} \sgn(n_i) \Res_{\xi_i = n_i} \right)
		\left( \prod_{j \in I^\complement} 2 \int_{C_{e_i}} d\xi_i  \right)
		\left[ \widehat{f}(\xi) \prod_{1 \le i \le r} \frac{1}{1 - \bm{e}(\xi_i)} \right]
		\\
		&=
		\sum_{I \subset \{1, \dots, r\}} \left( \prod_{i \in I} (-e_i) \right)
		\sum_{n_I^{} \in \Z^I} \sgn(n_I^{})
		\left( 
			\prod_{j \in I^\complement} 2 \int_{C_{e_j}}
			\frac{d\xi_j}{1 - \bm{e}(\xi_j)}
		\right)
		\widehat{f}(n_I^{}, \xi_{I^\complement}^{}).
	\end{align}	
\end{proof}


\section{A framework of ``modular series'' } \label{sec:modular_series_framework}


In this section, we introduce a framework of ``modular series.''


\subsection{Cyclotomic rational functions and quasi-polynomials} \label{subsec:modular_series_R_Q}


To begin with, we introduce cyclotomic rational functions and quasi-polynomials and give their correspondence.
We need this correspondence in order to handle false theta functions with complicated coefficients, such as GPPV invariants.

\begin{dfn} \label{dfn:cyclotomic_rat_func}
	We introduce the following notation:
	\begin{alignat}{3}
		\frakR &= & \frakR_r &\coloneqq
		\Q \left[ z_i^{\pm 1}, \frac{1}{1 - z_i^N} \relmiddle| 1 \le i \le r, \, N \in \Z_{>0} \right]
		& &\subset \Q(z_1, \dots, z_r), 
		\\
		\frakQ &= & \frakQ_r &\coloneqq
		\sum_{b \in \Z^r}
		\bm{1}_{b + \Z_{\ge 0}^r}(x) \cdot
		\Q[ x_i, \bm{1}_{a + N\Z} (x_i) \mid 1 \le i \le r, a, N \in \Z ]
		& &\subset \{ g \colon \Z^r \to \Q \}, 
		\\
		\overline{\frakQ} &= & \overline{\frakQ}_r &\coloneqq
		\Q[ x_i, \bm{1}_{a + N\Z} (x_i) \mid 1 \le i \le r, a, N \in \Z ]
		& &\subset \{ g \colon \Z^r \to \Q \}, 
		\\
		\widetilde{\frakQ} &= & \widetilde{\frakQ}_r &\coloneqq
		\sum_{b \in \Z^r}
		\sgn(x - b) \cdot
		\Q[ x_i, \bm{1}_{a + N\Z} (x_i) \mid 1 \le i \le r, a, N \in \Z ]
		& &\subset \{ g \colon \Z^r \to \Q \}.
	\end{alignat}
	Here, we denote by $ \bm{1}_A $ the characteristic function $ A $ and we adopt the convention that $ a + N\Z = \{ a \} $ when $ k=0 $.
	We call elements in $ \frakR, \frakQ, \overline{\frakQ}, \widetilde{\frakQ} $ \emph{cyclotomic rational functions},%
		\footnote{Professor Masanobu Kaneko suggested this terminology.}
		partial quasi-polynomials, quasi-polynomials, and false quasi-polynomials, respectively. 
\end{dfn}

When the index set of the variables is changed from $ \{ 1, \dots, r \} $ to a finite set $ I $,
we write $ \frakR_I, \frakQ_I, \overline{\frakQ}_I $, and $ \widetilde{\frakQ}_I $.

Although the indicator function $ \bm{1}_{\{ a \}} (x) $ of a singleton set is not usually referred to as a quasipolynomial, we will include it under that term in this paper for the sake of notational simplicity.

\begin{dfn}
	\begin{enumerate}
		\item 
		We define there $ \Q $-linear maps\footnote{
			``coe'' is  is an abbreviation of ``coefficient.''
		}
		\[
		\coe \colon R_{r} \to \frakQ, \quad
		\overline{\coe} \colon R_{r} \to \overline{\frakQ}, \quad
		\widetilde{\coe} \colon R_{r} \to \widetilde{\frakQ}
		\]
		by setting
		\begin{align}
			\coe[G] (m) 
			&\coloneqq
			\int_{\abs{z_i} = \varepsilon, \, 1 \le i \le r}
			G(z) \prod_{1 \le i \le r} \frac{z_i^{-m_i} dz_i}{2 \pi \iu z_i},
			\\
			\coe[G] (m) 
			&\coloneqq
			\left( 
			\prod_{1 \le i \le r} \frac{1}{2} 
			\left(
			\int_{\abs{z_i} = 1 - \varepsilon}
			- \int_{\abs{z_i} = 1 + \varepsilon}
			\right)
			\right)
			G(z) \prod_{1 \le i \le r} \frac{z_i^{-m_i} dz_i}{2 \pi \iu z_i},
			\\
			\widetilde{\coe}[G] (m) 
			&\coloneqq
			\PV \int_{\abs{z_i} = 1, \, 1 \le i \le r}
			G(z) \prod_{1 \le i \le r} \frac{z_i^{-m_i} dz_i}{2 \pi \iu z_i}.
		\end{align}
		
		\item 
		For $ e = (e_i)_{1 \le i \le r} \in \{ \pm 1 \}^r $, we define four $ \Q $-linear involutions as follows:
		\begin{align}
			&\begin{array}{rccc}
				\iota_e \colon & \frakR & \longrightarrow & \frakR \\
				& G(z_1, \dots, z_r) & \longmapsto & G(z_1^{e_1}, \dots, z_r^{e_r}),
			\end{array}
			\\
			&\begin{array}{rccc}
				\iota_e \colon & \frakQ & \longrightarrow & \frakQ \\
				& \bm{1}_{a_i} (x_i) & \longmapsto & \bm{1}_{e_i a_i} (x_i) \\
				& \bm{1}_{a_i + N_i \Z_{\ge 0}} (x_i) & \longmapsto & 
				\begin{cases}
					\bm{1}_{a_i + N_i \Z_{\ge 0}} (x_i) & \text{ if } e_i = 1, \\
					-\bm{1}_{a_i + N_i \Z_{\le -1}} (-x_i) & \text{ if } e_i = -1,
				\end{cases}
				\\
				& x_i^{m_i} & \longmapsto & (e_i x_i)^{m_i},
			\end{array}
			\\
			&\begin{array}{rccc}
				\iota_e \colon & \overline{\frakQ} & \longrightarrow & \overline{\frakQ} \\
				& \overline{g} (x) & \longmapsto & \left( \prod_{1 \le i \le r} e_i \right) \overline{g} (\transpose{e}x),
			\end{array}
			\\
			&\begin{array}{rccc}
				\iota_e \colon & \widetilde{\frakQ} & \longrightarrow & \widetilde{\frakQ} \\
				& \widetilde{g} (x) & \longmapsto & \widetilde{g} (\transpose{e}x),
			\end{array}
		\end{align}
		where $ 1 \le i \le r $, $ a_i \in \Z $, and $ N_i \in \Z_{>0} $. 
	\end{enumerate}
\end{dfn}

\begin{rem} \label{rem:coe_Laurent}
	For $ G(z) \in \frakR $, it holds that
	\[
	G(z)
	=
	\sum_{m \in \Z^r} \coe[G] (m) z_1^{l_1} \cdots z_r^{l_r}
	\]
	when $ \abs{z_i} < 1 $ for each $ 1 \le i \le r $.
\end{rem}

By using the above maps, we obtain the following correspondence between cyclotomic rational maps and (partial/false) quasi-polynomials.

\begin{lem} \label{lem:R_C_corresp}
	\begin{enumerate}
		\item \label{item:lem:R_C_corresp:R_C}
		The map $ \coe \colon \frakR \to \frakQ $ is an isomorphism of $ \Q $-vector spaces. 
		
		\item \label{item:lem:R_C_corresp:involution}
		For each $ e = (e_1, \dots, e_r) \in \{ \pm 1 \}^r $, the diagram
		\begin{equation}
			\begin{tikzcd}
				\frakR 
				\arrow[d, "\coe" ', "\sim" {anchor=south, rotate=-90}] \arrow[r, "\iota_e"]
				&
				\frakR 
				\arrow[d, "\coe" ', "\sim" {anchor=south, rotate=-90}] 
				\\
				\frakQ
				\arrow[r, "\iota_e"]
				&
				\frakQ
			\end{tikzcd}				
		\end{equation}
		commutes. 
		
		\item \label{item:lem:R_C_corresp:vp}
		For any cyclotomic rational function $ G(z) \in \frakR $, we have
		\begin{align}\label{eq:tilde_g}
			\overline{\coe}[G] (x)
			&=
			2^{-r} \sum_{e \in \{ \pm 1 \}^r} \left( \prod_{1 \le i \le r} e_i \right) \coe[G] (ex),
			\\
			\widetilde{\coe}[G] (x)
			&=
			2^{-r} \sum_{e \in \{ \pm 1 \}^r} \coe[G] (ex),
		\end{align}
		where $ ex \coloneqq (e_1 x_1, \dots, e_r x_r) $.
		
		\item \label{item:lem:R_C_corresp:vp_express}
		For a cyclotomic rational function
		\[
		G(z) 
		=
		p \left( z_1 \frac{\partial}{\partial z_1}, \dots, z_r \frac{\partial}{\partial z_r} \right)
		\frac{z_1^{a_1}}{1 - z_1^N} \cdots \frac{z_r^{a_r}}{1 - z_r^N}
		\]
		with $ a \in \Z^r, N \in \Z_{>0} $ and $ p(x) \in \Q[y] $, we have
		\begin{align}
			\coe[G] (x) 
			&= 
			\bm{1}_{a + N\Z_{\ge 0}^r} (x) p(x),
			\\
			\overline{\coe}[G] (x)
			&=
			2^{-r} \bm{1}_{a + N\Z^r} (x) p(x),
			\\
			\widetilde{\coe}[G] (x)
			&=
			2^{-r} \sgn(y - a) \bm{1}_{a + N\Z^r} (x) p(x).
		\end{align}
		
		\label{item:lem:R_C_corresp:C_tilde_rep}
		We have
		\begin{align}
			\widetilde{\frakQ} 
			&=
			\sprod{ 
				\sgn(y - a) \bm{1}_{a + N\Z^r} (x) p(x)
				\relmiddle| 
				a \in \Z^r, N \in \Z_{>0}, p(x) \in \Q[y]
			}_\Q
			\\
			&=
			\sprod{
				2^{-r} \sum_{e \in \{ \pm 1 \}^r} \iota_e g (ex)
				\relmiddle| 
				g(x) \in \frakQ
			}_\Q
			\\
			&=
			\left\{
			\widetilde{g} \colon \Z^r \to \Q
			\relmiddle|
			\begin{gathered}
				\text{for some }
				g(x) \in \Q[ x_i, \delta(x_i \in a + N\Z) \mid 1 \le i \le r, a, N \in \Z ],
				\\
				\widetilde{g}(x) - \sgn(x) g(x)
				\text{ has finite support}
			\end{gathered}
			\right\}.
		\end{align}
		
		\item \label{item:lem:R_C_corresp:R_C_tilde}
		$ \overline{\coe} \colon \frakR \to \overline{\frakQ} $, $ \widetilde{\coe} \colon \frakR \to \widetilde{\frakQ} $
		are isomorphisms of $ \Q $-vector spaces.
		
		\item \label{item:lem:R_C_corresp:involution_bilateral}
		For each $ e = (e_1, \dots, e_r) \in \{ \pm 1 \}^r $, the diagram
		\begin{equation}
			\begin{array}{cc}
				\begin{tikzcd}
					\frakR 
					\arrow[d, "\overline{\coe}" ', "\sim" {anchor=south, rotate=-90}] \arrow[r, "\iota_e"]
					&
					\frakR 
					\arrow[d, "\overline{\coe}" ', "\sim" {anchor=south, rotate=-90}] 
					\\
					\overline{\frakQ}
					\arrow[r, "\iota_e"]
					&
					\overline{\frakQ}
				\end{tikzcd}
				& \quad
				\begin{tikzcd}
					\frakR 
					\arrow[d, "\widetilde{\coe}" ', "\sim" {anchor=south, rotate=-90}] \arrow[r, "\iota_e"]
					&
					\frakR 
					\arrow[d, "\widetilde{\coe}" ', "\sim" {anchor=south, rotate=-90}] 
					\\
					\widetilde{\frakQ}
					\arrow[r, "\iota_e"]
					&
					\widetilde{\frakQ}
				\end{tikzcd}
			\end{array}
		\end{equation}
		commutes. 
	\end{enumerate}
\end{lem}

\begin{proof}
	\cref{item:lem:R_C_corresp:R_C}
	By using partial fraction decomposition, we have
	\begin{align}
		\frakR
		&=
		\sum_{\substack{
				a_1, \dots, a_r \in \Z, \\
				N_1, \dots, N_r \in \Z_{> 0}, \\
				m_1, \dots, m_r \in \Z_{\ge 0}
		}}
		\Q \cdot \frac{z_1^{a_1} \cdots z_r^{a_r}}{(1 - z_1^{N_1})^{m_1} \cdots (1 - z_r^{N_r})^{m_r}} \\
		&=
		\sum_{\substack{
				a_1, \dots, a_r \in \Z, \\
				N_1, \dots, N_r \in \Z_{\ge 0}, \\
				m_1, \dots, m_r \in \Z_{\ge 0}
		}}
		\Q \cdot 
		\prod_{1 \le i \le r} 
		\left( z_i \frac{\partial}{\partial z_i} \right)^{m_i}
		z_i^{a_i} g_{N_i} (z_i),
	\end{align}
	where
	\[
	g_{N_i} (z_i) \coloneqq
	\begin{dcases}
		\frac{1}{1 - z_i^{N_i}} & \text{ if } N_i \ge 1, \\
		1 & \text{ if } N_i = 0.
	\end{dcases}
	\]
	We also have
	\begin{align}
		\frakQ
		&=
		\sum_{\substack{
				a_1, \dots, a_r \in \Z, \\
				N_1, \dots, N_r \in \Z_{> 0}, \\
				m_1, \dots, m_r \in \Z_{\ge 0}
		}}
		\Q \cdot \prod_{1 \le i \le r} x_i^{m_i}
		\bm{1}_{a_i + N_i \Z_{\ge 0}} (x_i).
	\end{align}
	Since the equality
	\[
	\left( z_i \frac{\partial}{\partial z_i} \right)^{m_i}
	z_i^{a_i} g_{N_i} (z_i)
	=
	\sum_{m_i \in \Z}
	m_i^{m_i} \bm{1}_{a_i + N_i \Z_{\ge 0}} (x_i) z_i^{m_i}
	\]
	holds, 
	\[
	\prod_{1 \le i \le r} 
	\left( z_i \frac{\partial}{\partial z_i} \right)^{m_i}
	z_i^{a_i} g_{N_i} (z_i)
	\in \frakR
	\quad \text{ and } \quad
	\prod_{1 \le i \le r} x_i^{m_i} \bm{1}_{a_i + N_i \Z_{\ge 0}} (x_i)
	\in \frakQ
	\]
	correspond via the map $ \coe $. 
	Thus, we obtain the claim. 
	
	\cref{item:lem:R_C_corresp:involution}
	For $ a_i \in \Z, N_i \in \Z_{>0}, m_i \in \Z_{\ge 0} $, we have
	\begin{align}
		\restrict{\left(
			\left( z_i \frac{\partial}{\partial z_i} \right)^{m_i}
			\frac{z_i^{a_i}}{1 - z_i^{N_i}}
			\right)}{z_i \mapsto z_i^{-1}}
		&=
		\left( -z_i \frac{\partial}{\partial z_i} \right)^{m_i}
		\frac{z_i^{-a_i}}{1 - z_i^{-N_i}}
		\\
		&=
		- \left( -z_i \frac{\partial}{\partial z_i} \right)^{m_i}
		\frac{z_i^{N_i - a_i}}{1 - z_i^{N_i}}
		\\
		&=
		- \sum_{m_i \in \Z}
		\bm{1}_{a_i + N_i \Z_{\le -1}} (-m_i) (-m_i)^{m_i} z_i^{m_i}
		\\
		&=
		\sum_{m_i \in \Z}
		\iota_{-} [\bm{1}_{a_i + N_i \Z_{\ge 0}} (x_i) x_i^{m_i}] (m_i) z_i^{m_i}
	\end{align}
	from which the claim follows.
	
	\cref{item:lem:R_C_corresp:vp}
	Since both sides are $ \Q $-linear, we can assume
	\[
	G(z) = G_1 (z_1) \cdots G_r (z_r), \quad
	G_i (z_i) \in \Q \left[ z_i^{\pm 1}, \frac{1}{1 - z_i^N} \relmiddle| N \in \Z_{>0} \right].
	\]
	In this case, we compute that 
	\begin{align}
		\overline{\coe}[G] (m)
		&=
		\prod_{1 \le i \le r} \frac{1}{2} 
		\left(
		\int_{\abs{z_i} = 1 - \varepsilon} - \int_{\abs{z_i} = 1 + \varepsilon}
		\right)
		G_i(z_i) \frac{z_i^{-m_i} dz_i}{2\pi \iu z_i} \\
		&=
		\prod_{1 \le i \le r} \frac{1}{2} 
		\int_{\abs{z_i} = 1 - \varepsilon}
		\left(
		G_{i}(z_i) z_i^{-m_i} - G_{i}(z_i^{-1}) z_i^{m_i}
		\right)
		\frac{dz_i}{2\pi \iu z_i},
		\\
		\widetilde{\coe}[G] (m)
		&=
		\prod_{1 \le i \le r} \frac{1}{2} 
		\left(
		\int_{\abs{z_i} = 1 - \varepsilon} + \int_{\abs{z_i} = 1 + \varepsilon}
		\right)
		G_i(z_i) \frac{z_i^{-m_i} dz_i}{2\pi \iu z_i} \\
		&=
		\prod_{1 \le i \le r} \frac{1}{2} 
		\int_{\abs{z_i} = 1 - \varepsilon}
		\left(
		G_{i}(z_i) z_i^{-m_i} + G_{i}(z_i^{-1}) z_i^{m_i}
		\right)
		\frac{dz_i}{2\pi \iu z_i}.
	\end{align}
	Here, since $ G_{i}(z_i^{\pm 1}) \in \Q [ z_i^{\pm 1}, (1 - z_i^N)^{-1} \mid N \in \Z_{>0} ] $,
	they only have a pole at $ z_i = 0 $ in $ \abs{z_i} < 1 $. 
	By using this fact and \cref{item:lem:R_C_corresp:involution}, we conclude that
	\begin{align}
		\overline{\coe}[G] (m)
		&=
		\prod_{1 \le i \le r} \frac{1}{2} 
		\left(
		\coe[G_i] (m_i) - \iota_{-} \coe[G_i] (-m_i)
		\right)
		=
		2^{-r} \sum_{e \in \{ \pm 1 \}^r} \left( \prod_{1 \le i \le r} e_i \right)
		\coe[G] (el),
		\\
		\widetilde{\coe}[G] (m)
		&=
		\prod_{1 \le i \le r} \frac{1}{2} 
		\left(
		\coe[G_i] (m_i) + \iota_{-} \coe[G_i] (-m_i)
		\right)
		=
		2^{-r} \sum_{e \in \{ \pm 1 \}^r}
		\coe[G] (el).
	\end{align}
	
	\cref{item:lem:R_C_corresp:vp_express}
	The first equality follows from the above proof of \cref{item:lem:R_C_corresp:R_C}. 
	For $ a_i \in \Z, N_i \in \Z_{>0}, m_i \in \Z_{\ge 0} $, let $ g_i (x_i) = \bm{1}_{a_i + N_i \Z_{\ge 0}} (x_i) x_i^{m_i} $.
	Since we can write $ \iota_- g_i (-x_i) = - \bm{1}_{a_i + N_i \Z_{\le -1}} (x_i) x_i^{m_i} $, we obtain
	\begin{align}
		g_i (x_i) - \iota_- g_i (-x_i)
		&=
		\bm{1}_{a_i + N_i \Z} (x_i) x_i^{m_i},
		\\
		g_i (x_i) + \iota_- g_i (-x_i)
		&=
		\sgn(x_i - a_i) \bm{1}_{a_i + N_i \Z} (x_i) x_i^{m_i}.
	\end{align}
	By this and \cref{item:lem:R_C_corresp:vp}, we obtain the second and last equalities.
	
	\cref{item:lem:R_C_corresp:C_tilde_rep}
	It suffices to show the first equality. 
	Given $ b_i, a \in \Z, N \in \Z_{>0} $, we take $ 0 \le b_i' < N $ and $ n_i \in \Z $ with $ b_i = b_i' + N n_i $.
	To begin with, we prove
	\begin{equation} \label{eq:sgn}
		\sgn(x_i - b_i) \bm{1}_{a + N\Z} (x_i)
		\in \sgn(x_i - k n_i) \bm{1}_{a + N\Z} (x_i)
		+ \Q[\bm{1}_c(x_i) \mid c \in \Z].
	\end{equation}
	We can assume $ 0 \le a < k $. 
	For $ x_i = a + k m_i $ with some $ m_i \in \Z $, we have
	\begin{align}
		\sgn(x_i - b_i) \bm{1}_{a + N\Z} (x_i)
		&=
		\sgn(a_i - b_i' + k (m_i - n_i))
		\\
		&=
		\begin{cases}
			\sgn(N_i(m_i - n_i)) & \text{ if } m_i \neq n_i, \\
			\sgn(a_i - b_i') & \text{ if } m_i = n_i
		\end{cases}
		\\
		&=
		\sgn(N_i(m_i - n_i))
		+ \left( -1 + \sgn(a_i - b_i') \right) \bm{1}_{a + k n_i } (x_i)
		\\
		&=
		\sgn(x_i - k n_i)
		+ \left( -1 + \sgn(a_i - b_i') \right) \bm{1}_{a + k n_i } (x_i).
	\end{align}
	Thus, we obtain \cref{eq:sgn}. 
	
	By using \cref{eq:sgn} two times, we have
	\begin{align}
		\widetilde{\frakQ}
		&=
		\sprod{ \prod_{1 \le i \le r} \sgn(x_i - k n_i) \bm{1}_{a_i + N_i \Z} (x_i) x_i^{m_i} \relmiddle| 
			\begin{aligned}
				a_1, \dots, a_r &\in \Z, \\
				N_1, \dots, N_r &\in \Z_{> 0}, \\
				m_1, \dots, m_r &\in \Z_{\ge 0} 
			\end{aligned}
		}_\Q
		\\
		&=
		\sprod{ \prod_{1 \le i \le r} \sgn(x_i - a_i) \bm{1}_{a_i + N_i \Z} (x_i) x_i^{m_i} \relmiddle| 
			\begin{aligned}
				a_1, \dots, a_r &\in \Z, \\
				N_1, \dots, N_r &\in \Z_{> 0}, \\
				m_1, \dots, m_r &\in \Z_{\ge 0} 
			\end{aligned}
		}_\Q.
	\end{align}
	By \cref{item:lem:R_C_corresp:vp_express}, we obtain
	\[
	\widetilde{\frakQ}
	=
	\sprod{ 
		\sgn(y - a) \bm{1}_{a + N\Z^r} (x) p(x)
		\relmiddle| 
		a \in \Z^r, N \in \Z_{>0}, p(x) \in \Q[y]
	}_\Q.
	\]
	
	\cref{item:lem:R_C_corresp:R_C_tilde}
	The claim for $ \overline{\coe} $ follows from \cref{item:lem:R_C_corresp:vp_express}.  
	We need to prove the claim for $ \widetilde{\coe} $.
	Fix an arbitrary $ \widetilde{g}(x) \in \widetilde{\frakQ} $. 
	Let $ g_{\ge} (x) \coloneqq 2^r \bm{1}_{\Z_{\ge 0}^r} (x) \widetilde{g}(x) $.
	This is an element of $ \frakQ $. 
	By \cref{item:lem:R_C_corresp:R_C}, there exists the corresponded cyclotomic rational function $ G_{\ge} (z) \in \frakR $. 
	Let
	\[
	\widetilde{g}_{<} (x) \coloneqq
	\widetilde{g} (x) - 
	2^{-r} \sum_{e \in \{ \pm 1 \}^r} \iota_e g_{\ge}(el).
	\]
	By the definition of $ \widetilde{\frakQ} $, $ \widetilde{g}_{<} (x) $ has finite support. 
	Thus, 
	\[
	G_{<} (z) \coloneqq
	\sum_{m \in \Z^r} \widetilde{g}_{<} (m) z^l
	\]
	is an element of $ \Q[z_1^{\pm 1}, \dots, z_r^{\pm 1}] $. 
	Then, by putting $ G(z) \coloneqq G_{\ge} (z) + G_{<} (z) \in \frakR $, we have
	\[
	\widetilde{g} (m) 
	=
	\PV \int_{\abs{z_i} = 1, \, 1 \le i \le r}
	G(z) \prod_{1 \le i \le r} \frac{z_i^{-m_i} dz_i}{2 \pi \iu z_i}
	=
	\widetilde{\coe}[G] (m) 
	\]
	by \cref{item:lem:R_C_corresp:involution,item:lem:R_C_corresp:vp}.
	
	\cref{item:lem:R_C_corresp:involution_bilateral}
	For $ G(z) \in \frakR $, we have
	\begin{alignat}{2}
		\overline{\coe} [\iota_e G] (x)
		&=
		2^{-r} \sum_{e' \in \{ \pm 1 \}^r} \left( \prod_{1 \le i \le r} e'_i \right) \iota_{e'} \coe[\iota_e G] (\transpose{e'} x)
		& & \quad \text{ by \cref{item:lem:R_C_corresp:involution}}
		\\
		&=
		2^{-r} \sum_{e' \in \{ \pm 1 \}^r} \left( \prod_{1 \le i \le r} e'_i \right) \iota_{e'e} \coe[G] (\transpose{e'} x)
		& & \quad \text{ by \cref{item:lem:R_C_corresp:vp}}
		\\
		&=
		2^{-r} \sum_{e' \in \{ \pm 1 \}^r} \left( \prod_{1 \le i \le r} e'_i e_i \right) \iota_{e'} \coe[G] (\transpose{(e'e)} x)
		& & \quad \text{ by replacing $ e' $ by $ e'e $}
		\\
		&=
		\overline{\coe} [G] (ex)
		& & \quad \text{ by \cref{item:lem:R_C_corresp:involution}}.
	\end{alignat}
	By the same argument, we obtain the claim for $ \widetilde{\coe} $.
\end{proof}


\subsection{Modular series} \label{subsec:modular_series}


The main object of this section, a \emph{modular series}, is composed of (partial/false) quasi-polynomials and certain building blocks called \emph{kernel functions}.
A kernel function is defined as follows.

\begin{dfn} \label{dfn:kernel_func}
	A \emph{kernel function} with $ r $ variables is a continuous function $ \gamma \colon \bbH \times \R^r \to \bbC $ that satisfies the following conditions:
	\begin{itemize}
		\item The function $ \gamma $ is holomorphic in the variable $ \tau \in \bbH $.
		\item The function $ \gamma $ is of exponential decay in $ y \in \R^r $ uniformly in $ \tau \in \bbH $, that is, 
		for any $ \delta > 0 $, there exist $ a, K > 0 $ such that 
		$ \abs{\gamma(\tau; x)} \le Ke^{-a\abs{x}} $ for any $ \tau \in \bbH $ with $ \Im(\tau) > \delta $ and $ x \in \R^r $.
		\item Its Fourier transform 
		\[
		\calF[\gamma] (\tau; \xi) 
		\coloneqq
		\int_{\R^r} \gamma (\tau; x) \bm{e} (-\transpose{\xi}x) dx
		\]
		is absolutely integrable on $ \R^r + \iu b $ for any sufficiently small $ \varepsilon > 0 $ and any $ b \in (-\varepsilon, \varepsilon)^r $.
	\end{itemize}
\end{dfn}

\begin{dfn}
	We say that a function $ g(x) $ is \emph{of polynomial growth} if 
	there exists $ n \in \Z_{>0} $ such that $ \abs{g(x)} = O(\abs{x}^n) $ as $ \abs{x} \to \infty $. 
\end{dfn}

In what follows, we use the following notation. 

\begin{notn*}
	\begin{itemize}
		\item We denote by $ \frakK $ the set of kernel function with $ r $ variables.
		\item Let $ \calO(\bbH) $ denote the set of holomorphic functions on $ \bbH $. 
		\item We use the notation 
		$ C_{\mathrm{poly}}(\Z^r) =
		\left\{ 
		g \colon \Z^r \to \bbC \relmiddle| 
		\text{$ g $ is of polynomial growth}
		\right\} $.
	\end{itemize}
\end{notn*}

We define modular series as follows.

\begin{dfn} \label{dfn:modular_series}
	\begin{enumerate}
		\item We define  a $ \Q $-linear map $ \Phi \colon C_{\mathrm{poly}}(\Z^r)\times \frakK \to \calO(\bbH)$ as
		\[
		\Phi[g, \gamma] (\tau)
		\coloneqq
		\sum_{m \in \Z^r} g(m) \gamma (\tau; m).
		\]
		
		\item For $ g \in \frakQ, \overline{g} \in \overline{\frakQ} $, and $ \widetilde{g} \in \widetilde{\frakQ} $, we call
		$ \Phi[g, \gamma] (\tau), \Phi[\overline{g}, \gamma] (\tau) $, and $ \Phi[\widetilde{g}, \gamma] (\tau) $ 
		\emph{partial modular series, modular series, and false modular series of the kernel function $ \gamma (\tau; y) $}. 
		
	\end{enumerate}
\end{dfn}

\begin{ex} \label{ex:ker_func}
	\begin{enumerate}
		\item A typical example of a kernel function is $ q^{x^2/2} $.
		Partial modular series, modular series, and false modular series of the kernel function $ q^{x^2/2} $ are called \emph{partial theta functions, theta functions, and false theta functions}.
		
		\item \label{item:ex:ker_func:Eisenstein_level_4} 
		A function $ 1/\cosh (\iu \tau x) = 2/(q^{x/2} + q^{-x/2}) $ is a kernel function. 
		When we choose a quasi-polynomial as $ \overline{g}(x) \coloneqq 1 $, the modular series can be written as
		\[
		\Phi[\overline{g}, \gamma] (\tau) 
		=
		1 + 4 \sum_{m=1}^\infty \frac{1}{q^{m/2} + q^{-m/2}}
		=
		1 + 4 \sum_{n=1}^\infty \left( \sum_{d \mid n, \, d \text{ odd}} (-1)^{(d-1)/2} \right) q^{n/2}
		\eqqcolon
		E_2^{(4)} \left( \frac{\tau}{2} \right),
		\]
		which is called \emph{Eisenstein series of weight $ 2 $ and level $ 4 $}.
		
	\end{enumerate}
\end{ex}

We give further examples in \cref{sec:false_theta,sec:indefinite_false,sec:Eisenstein}. 

In the next subsection, we will prove modular transformation formulas for modular series and false modular series, which are variations of the Poisson summation formula with signature (\cref{thm:PSF_sgn}).
To achieve this, let us first study the Fourier transformation of kernel functions.

We denote $ \widehat{\frakK} \coloneqq \left\{ \calF[\gamma] \colon \bbH \times \R^r \to \bbC \relmiddle| \gamma \in \frakK \right\} $.

\begin{lem} \label{lem:ker_func}
	The set $ \widehat{\frakK} $ is the set of functions $ \widehat{\gamma} \colon \bbH \times \R^r \to \bbC $ that satisfies the followings:
	\begin{enumerate}
		\item \label{item:lem:ker_func:1}
		For any $ \delta > 0 $, there exists $ a>0 $ such that $ \widehat{\gamma} $ extends holomorphically to
		\[
		\{ \tau \in \bbH \mid \Im (\tau) > \delta \}
		\times \{ \xi \in \bbC \mid \abs{\Im(\xi)} < a \}^r.
		\]
		
		\item \label{item:lem:ker_func:2}
		Given $ a>0 $ as above, for any $ b \in (-a, a)^r $, the function $ \widehat{\gamma} (\tau; u + b \iu) $ is absolutely integrable in $ u \in \R^r $ and its inverse Fourier transform
		\[
		\calF^{-1}[\widehat{\gamma}(\tau; u + b \iu)] (x)
		\coloneqq
		\int_{\R^r} \widehat{\gamma} (\tau; u + b \iu) \bm{e} (\transpose{u}x) du
		\]
		is continuous.
		
		\item \label{item:lem:ker_func:3}
		For $ a>0 $ as above, 
		\[
		\sup_{b \in (-a, a)^r} \norm{\calF^{-1}[\widehat{\gamma}(\tau; u + b \iu)] (x)}_{L^\infty} < \infty.
		\]
	\end{enumerate}
\end{lem}

\begin{proof}
	Let $ \gamma (\tau; y) \in \frakK $ and $ \delta > 0 $ be arbitrary.
	There exist $ a, K > 0 $ such that $ \abs{f(x)} \le Ke^{-a\abs{x}} $ for any $ x \in \R^r $.
	For any $ b \in (-a/2\pi, a/2\pi)^r $, Fourier transform $ \calF[\gamma] (\tau; x) $ converges uniformly in any compact subset of $ \R^r + b \iu  $ since $ \abs{\gamma (\tau; x) \bm{e} (-\transpose{\xi}x)}\le K e^{-a\abs{x} + \abs{\transpose{b}x}} $.
	
	Since $ \calF^{-1} [\calF[\gamma]] (\tau; x) = \gamma (\tau; x) $ holds as elements of $ L^2(\R^r) $ and 
	the map $ C^0(\R^r) \to L^2(\R^r) $ is injective,
	$ \calF^{-1} [\calF[\gamma]] (\tau; x) = \gamma (\tau; x) $ holds as functions.
	In particular, $ \calF^{-1} [\calF[\gamma]] (\tau; x) $ is continuous.
	The function $ \calF^{-1} [\calF[\gamma] (\tau; u+ b\iu)] (x) = \gamma(\tau; y) e^{2\pi \transpose{b}y} $ is bounded in $ b \in (-a/2\pi, a/2\pi)^r $ and $ x \in \R^r $.
	Thus, $ \calF[\gamma] (\tau; x) \in \widehat{\frakK} $ satisfies the conditions \cref{item:lem:ker_func:1,item:lem:ker_func:2,item:lem:ker_func:3}.
	
	Conversely, let $ \widehat{\gamma} (\tau; \xi) $ be a function satisfying the conditions \cref{item:lem:ker_func:1,item:lem:ker_func:2,item:lem:ker_func:3}.
	 and $ \delta > 0 $ be arbitrary and $ a > 0 $ be as in \cref{item:lem:ker_func:1}.
	Let $ x \in \R^r $ and $ 0 < \varepsilon < a $ be arbitrary and put $ b \coloneqq (a-\varepsilon) (\sgn(x_1), \dots, \sgn(x_r)) $.	
	By \cref{item:lem:ker_func:1}, we can write
	\[
	\calF^{-1}[\widehat{\gamma}] (\tau; x)
	=
	\int_{\R^r + b\iu} \widehat{\gamma} (\tau; \xi) \bm{e} (\transpose{\xi}x) d\xi
	=
	e^{-2\pi (a-\varepsilon) \abs{x}} \calF^{-1}[\widehat{\gamma}(\tau; u + b\iu)] (x).
	\]
	This function is continuous by \cref{item:lem:ker_func:2}.
	By \cref{item:lem:ker_func:3},
	\[
	K \coloneqq \sup_{b \in (-a, a)^r, \, x \in \R^r} \abs{\calF^{-1}[\widehat{\gamma}(\tau; u + b \iu)] (x)}
	\]
	is finite.
	Since we have $ \abs{\calF^{-1}[\widehat{\gamma}] (\tau; x)} \le K e^{-2\pi (a-\varepsilon) \abs{x}} $ for any $ 0 < \varepsilon < a $,
	we obtain $ \abs{\calF^{-1}[\widehat{\gamma}] (\tau; x)} \le K e^{-2\pi a \abs{x}} $.
	Thus, we have $ \calF^{-1}[\widehat{\gamma}] (\tau; x) \in \frakK $.
\end{proof}

\begin{rem}
	We can consider \cref{lem:ker_func} as a variant of Paley--Wiener's theorem.
	A related result is given in \cite[Theorem IX.12--14]{Reed-Simon:Fourier}.
\end{rem}

Next, we will give integral representations of (partial/false) modular series.

\begin{notn*}
	\begin{itemize}
		\item For $ x = (x_1, \dots, x_r) \in \bbC^r $, we write $ \bm{e} (x) \coloneqq (\bm{e} (x_i))_{1 \le i \le r} $.
		
		\item We define three $ \Q $-linear maps $ \calL, \overline{\calL}, \widetilde{\calL} \colon \frakR \times \widehat{\frakK} \to \calO(\bbH) $ as
		\begin{align}
			\calL[G, \widehat{\gamma}] (\tau)
			&\coloneqq
			\int_{(\R + \iu \varepsilon)^r}
			G (\bm{e} (x)) 
			\widehat{\gamma} (\tau; x) dx, 
			\\
			\overline{\calL}[G, \widehat{\gamma}] (\tau)
			&\coloneqq
			\left( 
			\prod_{1 \le i \le r} \frac{1}{2} 
			\left(
			\int_{\R + \iu \varepsilon}
			- \int_{\R - \iu \varepsilon}
			\right)
			\right)
			G (\bm{e} (x))
			\widehat{\gamma} (\tau; x) dx, 
			\\
			\widetilde{\calL}[G, \widehat{\gamma}] (\tau)
			&\coloneqq
			\PV \int_{\R^r}
			G (\bm{e} (x))
			\widehat{\gamma} (\tau; x) dx.
		\end{align}
	\end{itemize}
\end{notn*}

As shown in the following proposition, the above holomorphic functions are integral representations of (partial/false) modular series.
The result for false modular series is an extension of the first equality in \cref{thm:PSF_sgn}.

\begin{prop} \label{prop:modular_series_Laplace}
	Diagrams
	\begin{equation}
		\begin{array}{ccc}
			\begin{tikzcd}
				\frakQ \times \frakK \arrow[d, "\coe^{-1} \times \calF" ', "\sim" {anchor=south, rotate=-90}] \arrow[r, "\Phi"] & \calO(\bbH) \\
				\frakR \times \widehat{\frakK} \arrow[ru, "\calL" '] &
			\end{tikzcd}
			&
			\begin{tikzcd}
				\overline{\frakQ} \times \frakK \arrow[d, "\overline{\coe}^{-1} \times \calF" ', "\sim" {anchor=south, rotate=-90}] \arrow[r, "\Phi"] & \calO(\bbH) \\
				\frakR \times \widehat{\frakK} \arrow[ru, "\overline{\calL}" '] &
			\end{tikzcd}
			&
			\begin{tikzcd}
				\widetilde{\frakQ} \times \frakK \arrow[d, "\widetilde{\coe}^{-1} \times \calF" ', "\sim" {anchor=south, rotate=-90}] \arrow[r, "\Phi"] & \calO(\bbH) \\
				\frakR \times \widehat{\frakK} \arrow[ru, "\widetilde{\calL}" '] &
			\end{tikzcd}
		\end{array}
	\end{equation}
	commute. 
\end{prop}

\begin{proof}
	Take any cyclotomic rational function $ G(z) \in \frakR $ and kernel function $ \gamma(\tau; x) \in \frakK $.
	By \cref{rem:coe_Laurent}, it holds that
	\[
	G(\bm{e}(\xi))
	=
	\sum_{m \in \Z^r} \coe[G] (m) \bm{e} \left( \transpose{m} \xi \right)
	\]
	if $ \Im(\xi_i) > 0 $ for each $ 1 \le i \le r $.
	Thus, we have
	\[
	\Phi [\coe[G], \gamma] (\tau)
	=
	\sum_{m \in \Z^r} \coe[G] (m)
	\int_{\R^r}
	\calF[\gamma] (\tau; \xi)
	\bm{e} \left( \transpose{m} \xi \right) d\xi
	=
	\int_{(\R + \iu \varepsilon)^r}
	\calF[\gamma] (\tau; \xi)
	G(\bm{e}(\xi)) d\xi,
	\]
	which proves commutativity of the first diagram. 
	
	Next, we will prove the commutativity of the second diagram. 
	Since 
	\[
	\overline{\coe}[G] (x)
	=
	2^{-r} \sum_{e \in \{ \pm 1 \}^r} \left( \prod_{1 \le i \le r} e_i \right) \coe[G] (ex)
	\]
	by \cref{lem:R_C_corresp} \cref{item:lem:R_C_corresp:vp}, we have
	\begin{align}
		\Phi [\overline{\coe}[G], \gamma] (\tau)
		&=
		2^{-r} \sum_{e \in \{ \pm 1 \}^r} \left( \prod_{1 \le i \le r} e_i \right)
		\sum_{m \in \Z^r} \coe[G] (em) \gamma (\tau; m)
		\\
		&=
		2^{-r} \sum_{e \in \{ \pm 1 \}^r} \left( \prod_{1 \le i \le r} e_i \right)
		\sum_{m \in \Z^r} \coe[G] (m) \gamma (\tau; em).
	\end{align}
	By integral representation of $ \Phi_\gamma[g] (\tau) $ shown above, we have
	\begin{align}
		\Phi [\overline{\coe}[G], \gamma] (\tau)
		&=
		2^{-r} \sum_{e \in \{ \pm 1 \}^r} \left( \prod_{1 \le i \le r} e_i \right)
		\int_{(\R + \iu \varepsilon)^N}
		G (\bm{e} (e\xi)) \widehat{\gamma} (\tau; \xi) dx\\xi
		\\
		&=
		2^{-r} \sum_{e \in \{ \pm 1 \}^r}
		\left(
		\prod_{1 \le i \le r}
		e_i \int_{\R + e_i \iu \varepsilon} d\xi_i
		\right)
		G (\bm{e} (\xi)) \widehat{\gamma} (\tau; \xi)
		\\
		&=
		\left(
		\prod_{1 \le i \le r}
		\frac{1}{2}
		\left(
		\int_{\R + \iu \varepsilon} d\xi_i - \int_{\R - \iu \varepsilon} d\xi_i
		\right)
		\right)
		G (\bm{e} (\xi)) \widehat{\gamma} (\tau; \xi)
		\\
		&=
		\overline{\calL}[G, \gamma] (\tau),
	\end{align}
	which proves commutativity of the second diagram. 
	
	We can prove commutativity of the last diagram by applying the same argument for the equality
	\[
	\widetilde{\coe}[G] (x)
	=
	2^{-r} \sum_{e \in \{ \pm 1 \}^r} \coe[G] (ex)
	\]
	in \cref{lem:R_C_corresp} \cref{item:lem:R_C_corresp:vp}.
\end{proof}

\begin{ex} \label{ex:false_theta_PV_rep}
	Take a cyclotomic rational function
	\[
	G(z) 
	=
	p \left( z_1 \frac{\partial}{\partial z_1}, \dots, z_r \frac{\partial}{\partial z_r} \right)
	\left[ \frac{1}{1 - z_1^N} \cdots \frac{1}{1 - z_r^N} \right]
	\]
	with $ N \in \Z_{>0} $ and $ p(x) \in \Q[y] $.
	Let $ S \in \Sym_r^+ (\R) $ be a positive definite symmetric matrix and $ \gamma(\tau; x) \coloneqq q^{\transpose{x}Sx/2} $.
	Then, we have
	\begin{align}
		\Phi[\widetilde{\coe}[G], \gamma] (\tau)
		&=
		\PV \int_{\abs{z_i} = 1, \, 1 \le i \le r}
		\left(
		\sum_{m \in \Z^r} z_i^{-m_i} q^{\transpose{m}Sm/2}
		\right)
		G(z) \prod_{1 \le i \le r} \frac{dz_i}{2 \pi \iu z_i}
		\\
		&=
		\sum_{m \in \Z^r} \sgn(m) p(m) q^{\transpose{x}Sx/2}
		\\
		&=
		\frac{1}{\sqrt{\det S}} \sqrt{\frac{\iu}{\tau}}^r
		\PV \int_{\R^r} \widetilde{q}^{\transpose{\xi}S^{-1} \xi/2}
		G(\bm{e}(\xi)) d\xi
	\end{align}
	by \cref{lem:R_C_corresp} \cref{item:lem:R_C_corresp:vp_express} and \cref{prop:modular_series_Laplace},
	where $ \widetilde{q} \coloneqq \bm{e} (-1/\tau) $.
	
	As a special case, we obtain an integral representation of GPPV invariants (\cref{dfn:GPPV_inv,rem:GPPV_int_rep}).
\end{ex}


\section{Modular transformation formula for modular series} \label{sec:modular_series_modular}


In this section, we give modular transformation formulas for modular series and false modular series.


\subsection{Modular transformation formula for modular series} \label{subsec:modular_series_modular}


We use the following notation.

\begin{notn*}
	\begin{itemize}
		\item Denote
		\[
		\overline{\frakQ}^0 
		\coloneqq
		\Q[ x_i, \bm{1}_{a + N\Z} (x_i) \mid 1 \le i \le r, a \in \Z, N \in \Z_{>0} ]
		\subset \overline{\frakQ}.
		\]
		
		\item For a quasi-polynomial $ \overline{g}(x) \in \overline{\frakQ}^0 $, choose a positive integer $ N \in \Z_{>0} $ such that it can be expressed as
		\[
		\overline{g}(x)
		=
		\sum_{a \in (\Z/N\Z)^r} \bm{1}_{a + N\Z^r} (x) P_a(x)
		\in \overline{\frakQ}^0
		\]
		with a polynomial $ P_a (x) \in \Q[x_1, \dots, x_r] $ for $ a \in (\Z/N\Z)^r $.
		Then, we define an operator $ T_{\overline{g}, N} \colon \widehat{\frakK} \to \widehat{\frakK} $ by mapping
		$ \widehat{\gamma} (\tau; \xi) \in \widehat{\frakK} $ to
		\[
		T_{\overline{g}, N} \widehat{\gamma} (\tau; \xi)
		\coloneqq
		\frac{1}{N^r} \sum_{a \in (\Z/N\Z)^r} \bm{e} \left( \transpose{a} \xi \right)
		P_a \left( -\frac{1}{2\pi\iu} \frac{\partial}{\partial \xi_1}, \dots, -\frac{1}{2\pi\iu} \frac{\partial}{\partial \xi_r} \right)
		\widehat{\gamma} \left( \tau; \xi \right).
		\]
		
		\item $ \bm{e} (\xi) \coloneqq (\bm{e} (\xi_1), \dots, \bm{e} (\xi_r)) $.
	\end{itemize}
\end{notn*}

We regard the following proposition as a modular transformation formula for modular series, since the right hand side is typically expressed as a function of $-1/\tau$.
For example, in the situation in \cref{ex:false_theta_PV_rep}, the following proposition implies well-known modular transformation formulas for theta functions.

\begin{prop} \label{prop:modular_series_modular}
	For any $ \overline{g}(x) \in \overline{\frakQ}^0 $ and $ N \in \Z_{>0} $ choosed as above, we have
	\[
	\Phi[\overline{g}, \gamma] (\tau) = \Phi[1, T_{\overline{g}, N} \widehat{\gamma} (\tau; N^{-1} \xi)] (\tau).
	\]
	Moreover, the right hand side equals
	\[
	\sum_{n \in \calP} T_{\overline{g}, N} \widehat{\gamma} (\tau; n),
	\]
	where $ \calP \subset \Q^r $ denotes the set of poles of $ \overline{\coe} [\overline{g}] (\bm{e} (\xi)) $.
\end{prop}

\begin{proof}
	The first statement follows from the Poisson summation formula and 
	\[
	\calF \left[ P_a (a+Nx) \gamma(\tau; a+Nx) \right]
	=
	\frac{1}{N^r} \bm{e} \left( \frac{1}{N} \transpose{a} \xi \right)
	P_a \left( -\frac{1}{2\pi\iu} \frac{\partial}{\partial \xi_1}, \dots, -\frac{1}{2\pi\iu} \frac{\partial}{\partial \xi_r} \right)
	\widehat{\gamma} \left( \tau; \frac{1}{N} \xi \right).
	\]
	On the other hand, by \cref{prop:modular_series_Laplace} we have
	$ \Phi[\overline{g}, \gamma] (\tau) = \calL[G, \widehat{\gamma}] (\tau) $, where $ G(z) \coloneqq \overline{\coe}^{-1} [\overline{g}] (z) $.
	We prove the first and last statement by applying the residue theorem for this equality.
	We have
	\[
	\Phi[\overline{g}, \gamma] (\tau) = 
	\sum_{n \in \calP} T'_{\overline{g}} \widehat{\gamma} (\tau; n),
	\]
	where $ \widehat{\gamma} (\tau; \xi) $ is the Fourier transform of $ \gamma (\tau; x) $ and denote
	\[
	T'_{\overline{g}} \widehat{\gamma} (\tau; n)
	\coloneqq
	\left( \prod_{1 \le i \le r} (-2\pi\iu) \Res_{\xi_i = n_i} \right)
	\left[ G(\bm{e} (\xi)) \widehat{\gamma} (\tau; \xi) \right].
	\]
	Thus, it suffices to show that 
	$ T_{\overline{g}, N} \widehat{\gamma} (\tau; Nn) = T'_{\overline{g}} \widehat{\gamma} (\tau; n) $ for any $ n \in \calP $.
	
	First, we reduce to the case where $ \overline{g}(x) $ is a periodic map.
	For each $ 1 \le i \le r $, we have
	\[
	T_{\overline{g}(x) x_i, N} \widehat{\gamma} (\tau; \xi) 
	= 
	T_{\overline{g}(x) x_i, N} \left[ -\frac{1}{2\pi\iu} \frac{\partial \widehat{\gamma}}{\partial \xi_i} \right] (\tau; \xi).
	\]
	On the other hand, since
	\[
	\overline{\coe}[\overline{g}(x) x_i] (z)
	=
	\sum_{m \in \Z_{\ge 0}^r} \overline{g} (m) m_i z_1^{m_1} \cdots z_r^{m_r}
	=
	z_i \frac{\partial G}{\partial z_i} (z),
	\]
	we have
	\[
	T'_{\overline{g}(x) x_i} \widehat{\gamma} (\tau; \xi) 
	= 
	\left( \prod_{1 \le i \le r} (-2\pi\iu) \Res_{\xi_i = n_i} \right)
	\left[ \frac{1}{2\pi\iu} \frac{\partial G}{\partial \xi_i} (\bm{e} (\xi)) \widehat{\gamma} (\tau; \xi) \right].
	\]
	In general, for any meromorphic function $ \varphi(z) $ and $ \psi(z) $ and a complex number $ \alpha \in \bbC $, we have
	\[
	\Res_{z=\alpha} \varphi'(z) \psi(z)
	=
	- \Res_{z=\alpha} \varphi(z) \psi'(z)
	\]
	since $ \Res_{z=\alpha} \left( \varphi(z) \psi(z) \right)' = 0 $.
	Hence, we have
	\[
	T'_{\overline{g}(x) x_i} \widehat{\gamma} (\tau; \xi) 
	= 
	T'_{\overline{g}(x) x_i} \left[ -\frac{1}{2\pi\iu} \frac{\partial \widehat{\gamma}}{\partial \xi_i} \right] (\tau; \xi).
	\]
	Therefore, we can reduce to the case where $ \overline{g}(x) $ is a periodic map.
	
	Second, we prove the claim when $ \overline{g}(x) $ is a periodic map.
	We can assume $ r=1 $ and $ \overline{g}(x) = \bm{1}_{a + N\Z} (x) $ for $ a \in \Z $ and $ N \in \Z_{>0} $.
	In this case, we can write $ G(z) = z^a / (1 - z^N) $ and its any pole $ n $ is an element of $ \frac{1}{N} \Z $.
	Then, we have
	\[
	T'_{\overline{g}} \widehat{\gamma} (\tau; n)
	=
	-2\pi\iu \Res_{\xi = n} \frac{\bm{e} (a\xi)}{1 - \bm{e} (N\xi)} \widehat{\gamma} (\tau; \xi)
	=
	\frac{1}{N} \bm{e} (an) \widehat{\gamma} (\tau; n)
	=
	T_{\overline{g}, N} \widehat{\gamma} (\tau; n).
	\]
\end{proof}

\begin{ex}
	Let us consider the case when $ \gamma(\tau; x) \coloneqq 1/\cosh (\iu \tau x) = 2/(q^{x/2} + q^{-x/2}) $ and a quasi-polynomial as $ \overline{g}(x) \coloneqq 1 $. 
	In this case, by \cref{ex:ker_func} \cref{item:ex:ker_func:Eisenstein_level_4}, the modular series is $ \Phi[\overline{g}, \gamma] (\tau) = E_2^{(4)} \left( \tau/2 \right) $.
	Since $ \calF[1/\cosh(x)] = 1/\cosh(\xi) $,
	by applying the Poisson summation formula or \cref{prop:modular_series_modular}, we obtain
	\[
	E_2^{(4)} \left( \frac{\tau}{2} \right) 
	= \frac{\iu}{\tau} E_2^{(4)} \left( -\frac{1}{2\tau} \right),
	\]
	which is the modular transformation formula for $ E_2^{(4)} (\tau) $.
\end{ex}


\subsection{Modular transformation formula for false modular series} \label{subsec:false_modular_series_modular}


We use the following notation.

\begin{notn*}
	Let $ I \subset \{1, \dots, r\} $ be a subset.
	\begin{itemize}
		\item We denote $ I^\complement \coloneqq \{ 1, \dots, r \} \smallsetminus I $.
		\item For a variable $ x = (x_1, \dots, x_r) $, we denote 
		$ x_{I} \coloneqq (x_i)_{i \in I} $ and $ x_{I^\complement}^{} \coloneqq (x_j)_{j \in I^\complement} $. 
	\end{itemize}
\end{notn*}

The following theorem is our modular transformation formula.

\begin{thm} \label{thm:false_modular_series_modular}
	Let $ \overline{g}(x) \in \overline{\frakQ}^0 $ and choose a positive integer $ N $ such that
	\[
	\overline{g}(x) \in \sprod{ \bm{1}_{a + N \Z^r} (x) \mid a \in (\Z/ N\Z)^I }_{\Q[x]}.
	\]
	Then, for any $ e \in \{ \pm 1 \}^r $, we have
	\[
	\Phi[\sgn(x) \overline{g}(x), \gamma] (\tau)
	=
	\sum_{I \subset \{1, \dots, r\}} \left( \prod_{i \in I} e_i \right)
	\sum_{n_I \in \Z^I} \sgn(n_I)
	\left( 
	\prod_{j \in I^\complement} 2 \int_{C_{e_j}}
	\frac{d\xi_j}{1 - \bm{e}(\xi_j)}
	\right)
	T_{\overline{g}, N} \widehat{\gamma} \left( \tau; \frac{1}{N} n_I, \frac{1}{N} \xi_{I^\complement}^{} \right),
	\]
	where $ C_+ $ and $ C_- $ are integration paths defined in \cref{fig:C+,fig:C-}.
	
	Moreover, for each $ I \subset \{1, \dots, r\} $, we choose a positive integer $ N_I^{} $ such that 
	\[
	\overline{g}(x)
	=
	\sum_{ a_I^{} \in (\Z/ N_I^{} \Z)^I } \bm{1}_{a_I^{} + N_I^{}\Z^I} (x_I^{}) P_{a_I^{}} (x_I^{}) g_{a_I^{}} (x_{I^\complement}^{}) 
	\]
	with polynomials $ P_{a_I^{}} (x_I^{})  \in \Q[x_I^{}] $ and 
	quasi-polynomials $ g_{a_I^{}} (x_{I^\complement}^{}) \in \overline{\frakQ}^0_{I^\complement} $ 
	for $ a \in (\Z/N\Z)^r $.
	We denote
	\begin{align}
		&\phant
		T_{\overline{g}, I, N_I^{}} \widehat{\gamma} (\tau; \xi)
		\\
		&\coloneqq
		\frac{1}{N_I^{\abs{I}}} \sum_{ a_I^{} \in (\Z/ N_I^{} \Z)^I } \bm{e} \left( \transpose{a_I^{}} \xi_I^{} \right)
		P_{a_I^{}} \left( \left(-\frac{1}{2\pi\iu} \frac{\partial}{\partial \xi_i} \right)_{i \in I} \right)
		\widehat{\gamma} \left( \tau; \xi_I^{}, \xi_{I^\complement}^{} \right)
		\cdot \overline{\coe}[g_{a_I^{}}] ( \bm{e} ( \xi_{I^\complement}^{} ) ).
	\end{align}
	Let $ \calP_I^{} \subset \Q^I $ denote the set of poles of $ \overline{\coe} [\overline{g}] (\bm{e} (\xi)) $ in $ (\xi_i)_{i \in I} $.
	Then, we have
	\[
	\Phi[\sgn(x) \overline{g}(x), \gamma] (\tau)
	=
	\sum_{I \subset \{1, \dots, r\}} \left( \prod_{i \in I} e_i \right)
	\sum_{n_I \in \calP_I^{}} \sgn(n_I)
	\left( 
	\prod_{j \in I^\complement} 2 \int_{C_{e_j}}
	\right)
	T_{\overline{g}, I, N_I^{}} \widehat{\gamma} (\tau; n_I^{}, \xi_{I^\complement}^{})
	d\xi_{I^\complement}^{}.
	\]
\end{thm}

This theorem follows from the same argument in proof of \cref{thm:PSF_sgn,prop:modular_series_modular} for $ \Phi[\widetilde{g}, \gamma] (\tau) = \widetilde{\calL}[G, \widehat{\gamma}] (\tau)$ proved in \cref{prop:modular_series_Laplace}.
As in the proof of \cref{prop:modular_series_modular}, the final expression in this theorem is proved using the framework of modular series, rather than the Poisson summation formula with signature (\cref{thm:PSF_sgn}). 
The final expression is used for the precise determination of the asymptotic expansion of the asymptotic coefficients of false theta functions (\cref{thm:main_asymptotic_false_theta,thm:false_theta_asymp}), in particular the asymptotic expansion of WRT invariants (\cref{thm:main_WRT_asymptotic}).

\begin{rem}
	In the above formula, integral terms appear.
	Typically, they have analytic continuation to $ \bbC \smallsetminus \R_{\le 0} $ or $ \bbC \smallsetminus \R_{\ge 0} $ and thus, \cref{thm:false_modular_series_modular} implies quantum modularity for false modular series. 
	The difference between \cref{prop:modular_series_modular} and \cref{thm:false_modular_series_modular} lies in the choice of signature or integration path---that is, in the parity.
	This motivates the following slogan:
	\begin{quote}
		\centering
		\emph{Quantum modularity arises from the parity.}
	\end{quote}
\end{rem}


\section{Asymptotic expansion of modular series} \label{sec:asymptotic}


In this section, we establish asymptotic expansion formulas for partial and false modular series as $ \tau \to 0 $.


\subsection{Notation for asymptotic expansion} \label{subsec:asymptotic_notation}


To begin with, we prepare notation for asymptotic expansion as in \cite[Chapter I]{Wong}. 
In what follows, we fix 
a set $ \Omega \subset \bbC $,
two maps $ f, g \colon \Omega \to \bbC $,
and a point $ z_0 \in \overline{\Omega} \cup \{ \infty \} \subset \bbP^1(\bbC) $.

\begin{dfn} \label{dfn:asymptotic}
	\begin{enumerate}
		\item (Landau symbol)
		We write 
		\[
		f(z) = O(g(z)) \text{ as } z \to z_0
		\]
		to express that there exists an open neighborhood $ U $ of $ z_0 $ such that 
		$ \abs{f/g} $ is bounded on $ \Omega \cap U $.
		
		
		\item \label{item:dfn:asymptotic:Poincare}
		(Poincar\'{e}'s notation)
		For a sequence $ (a_j)_{j=j_0}^\infty $ of complex numbers with $ j_0 \in \Z $, we write 
		\[
		f(z) \sim \sum_{j=j_0}^{\infty} a_j z^{-j} \text{ as } z \to z_0
		\]
		to express that
		\[
		f(z) = \sum_{j=j_0}^{j_1} a_j z^{-j} + O(z^{-j_1 - 1}) \text{ as } z \to z_0
		\]
		holds for any integer $ j_1 \ge j_0 $.
		
		In this case, we call the above equation the \emph{asymptotic expansion of $ f(z) $ as $ z \to z_0 $}. 
	\end{enumerate}	
\end{dfn}

We need the following terminologies to state our results.

\begin{dfn}
	Let $ D \subset \R^r $ be an unbounded open domain.
	A $ C^\infty $ function $ f \colon D \to \bbC $ is called \emph{of rapid decay} as $ x_1, \dots, x_r \to \infty $ 
	if $ x_1^{m_1} \cdots x_r^{m_r} f^{(n)} (x) $ is bounded as $ x_1, \dots, x_r \to \infty $ for any $ m, n \in \Z_{\ge 0}^r $.
\end{dfn}

\begin{dfn}
	Let $ f \colon \R^r \to \bbC $ be a $ C^\infty $ function of rapid decay as $ x_1, \dots, x_r \to \infty $.
	\begin{enumerate}
		\item Define the \emph{$ (-1) $st derivative} or \emph{antiderivative} of $ f $ as
		\[
		f^{(-1)} (x) =
		\frac{\partial^{-1} f}{\partial x_i^{-1}} (x) \coloneqq -\int_{x_i}^\infty f dx'_i.
		\]
		Moreover, for $ j \in \Z^r $, define
		\[
		f^{(j)}(x) \coloneqq \frac{\partial^{j_1 + \cdots + j_r} f}{\partial x_1^{j_1} \cdots \partial x_r^{j_r}} (x).
		\]
		
		\item For a Laurent series
		\[
		\varphi(t_1, \dots, t_r) = \sum_{j \in \Z^r} \varphi_j t_1^{j_1} \cdots t_r^{j_r} \in \Q((t_1, \dots, t_r)),
		\]
		define {the Hadamard product with one variable} $ \varphi \odot f (t) $ as
		\[
		\varphi \odot f (t)
		\coloneqq
		\sum_{j \in \Z^r} \varphi_j f^{(j)}(0) t^{j_1 + \cdots + j_r}
		\in \bbC((t)).
		\]
	\end{enumerate}
\end{dfn}


\subsection{Asymptotic expansion formula of Euler--Maclaurin type} \label{subsec:asymptotic_EM}


When considering asymptotic expansions, we assume that each function takes the form 
$ \gamma(\tau; x) = f(\tau^\kappa x) $ for some function $ f $ and rational number $ \kappa $. 
This assumption causes no issues for applications.
For simplicity, we replace $ \tau^\kappa $ by $ \tau $ and $ f(x) $ by $ f(x/\iu) $.
In this case, by letting $ \tau = \iu t $, we can write $ \gamma(\tau; x) = f(tx) $.
We state asymptotic formulas for functions of this type.
Our asymptotic formulas are as follows.

\begin{prop} \label{prop:asymp_EM}
	Let $ g(x) \in \frakQ, \overline{g}(x) \in \overline{\frakQ}, \widetilde{g}(x) \in \widetilde{\frakQ} $, and 
	$ G(z) \in \frakR $ be \textup{(}partial/false\textup{)} quasi-polynomials and a cyclotomic rational function
	which correspond via $ \Q $-isomorphisms $ \coe, \overline{\coe}, \widetilde{\coe} $.
	Let $ f \colon \bbC^r \to \bbC $ be a $ C^\infty $ function and
	$ -\pi/2 < \theta_0 < 0 < \theta_1 < \pi/2 $ be two angles.
	We assume that $ f $ is of rapid decay as $ \abs{x_1}, \dots, \abs{x_r} \to \infty $ 
	on the domain $ \{ x \in \bbC^r \mid \arg(x_1), \dots, \arg(x_r) \in (\theta_0, \theta_1) \} $.
	Then, for a vector $ \alpha \in \R^r $ and a complex number $ t \in \bbC \smallsetminus \{ 0 \} $ with $ \theta_0 < \arg \left( t \right) < \theta_1 $, we have the following asymptotic formulas.
	\begin{enumerate}
		\item \label{item:prop:asymp_EM:1}
		\textup{(\cite[Proposition 3.6]{M:GPPV})}
		\[
		\sum_{m \in \Z^r} g (m) f(t(m + \alpha))
		\sim
		\varphi \odot f (t)
		\quad \text{ as } t \to 0,
		\]
		where
		\[
		\varphi(t_1, \dots, t_r)
		\coloneqq
		e^{\alpha_1 t_1 + \dots + \alpha_r t_r} G (e^{t_1}, \dots, e^{t_r})
		\in \Q[\alpha_1, \dots, \alpha_r] ((t_1, \dots, t_r)).
		\]
		\item \label{item:prop:asymp_EM:2}
		Moreover, if $ f $ is also of rapid decay as $ \abs{x_1}, \dots, \abs{x_r}\to \infty $ 
		on the domain $ \{ x \in \bbC^r \mid \arg(x_1), \dots, \arg(x_r) \in (\theta_0 + \pi, \theta_1 + \pi) \} $, we have 
		\begin{alignat}{2}
			\sum_{m \in \Z^r} \widetilde{g} (m) f(t(m + \alpha))
			&\sim
			\varphi \odot f (t)
			& &\quad \text{ as } t \to 0,
			\\
			\sum_{m \in \Z^r} \overline{g} (m) f(t(m + \alpha))
			&\sim
			O(t^{R})
			& &\quad \text{ as $ t \to 0 $ for any $ R>0 $}. 
		\end{alignat}
	\end{enumerate}
\end{prop}

Surprisingly, in particular $ \Phi[g, \gamma] (\tau) $ and $ \Phi[\sgn(x) g(x), \gamma] (\tau) $ share the same asymptotic expansion
for $ g(x) = \bm{1}_{\Z_{\ge 0}^r} (x) p(x) \in \frakQ $ with $ p(x) \in \Q[x_1, \dots, x_r] $.
This phenomenon relates to the symmetry of Bernoulli polynomials, as we discuss in \cref{rem:asymptotic_Bernoulli_poly_symmetry} later.

\begin{rem} \label{rem:asymp_coeff_Bernoulli}
	In \cref{prop:asymp_EM}, the Laurent coefficients $ B_m $ of $ G(e^{t_1}, \dots, e^{t_r}) $ can be written by using Bernoulli polynomials.
	Indeed, by considering partial fraction decomposition, any $ G(z_1, \dots, z_r) \in \frakR $ can be written as a finite $ \Q $-linear sum of finite products of $ z_i^{a} / (1 - z_i^N) $ and we have 
	\[
	\eval{\frac{z_i^{a}}{1 - z_i^N}}_{z_i=e^t}
	=
	-\sum_{j=-1}^\infty \frac{B_{j+1} (a/k)}{(j+1)!} (Nt)^{j},
	\]
	where $ B_m (\alpha) $ is the \emph{$ m $-th Bernoulli polynomial} defined as
	\[
	\sum_{j=0}^\infty \frac{B_j (\alpha)}{j!} t^j
	\coloneqq
	\frac{t e^{\alpha t}}{e^t - 1}.
	\]
\end{rem}

\begin{rem}
	It follows immediately that \cref{prop:asymp_EM} also holds in the case of $ g(x) \in \frakQ \otimes_{\Q} \bbC $ and
	$ G(z) \in \frakR \otimes_{\Q} \bbC $.
\end{rem}

Before the proof, we discuss the relation to previous works.
A prototype of \cref{prop:asymp_EM} is the following statement:

\begin{lem} \label{lem:EM_asymptotic}
	For a $ C^\infty $ function $ f \colon \R^r \to \bbC $ of rapid decay as $ x_1, \dots, x_r \to \infty $, 
	a vector $ \alpha \in \R^r $, and a variable $ t \in \R_{>0} $, we have an asymptotic expansion 
	\[
	\sum_{m \in \Z_{\ge 0}^r} f(t(m+\alpha))
	\sim 
	\sum_{j \in \Z_{\ge -1}^r} \frac{B_{j_1 + 1}(\alpha_ 1) \cdots B_{j_r + 1}(\alpha_ r)}{(j_1 + 1)! \cdots (j_r + 1)!}
	f^{(j)}(0) t^{j_1 + \cdots + j_r}
	\quad \text{ as } t \to +0.
	\]
\end{lem}

We can prove this formula by using the Euler--Maclaurin summation formula as in a proof of \cite[Equation (44)]{Zagier:asymptotic}.
This asymptotic formula is stated in \cite[Equation (44)]{Zagier:asymptotic} for the case when $ r=1 $ and Bringmann--Kaszian--Milas~\cite[Equation (2.8)]{BKM} and Bringmann--Mahlburg--Milas~\cite[Lemma 2.2]{BMM:high_depth} for the case when $ r=2 $.
\cref{lem:EM_asymptotic} are generalized by Murakami in \cite[Proposition 5.4]{M:plumbed} and \cite[Proposition 3.6]{M:GPPV} and gave a variant in \cite[Proposition 3.8]{M:GPPV}.
Our \cref{prop:asymp_EM} is an extension of all these formulas.

\begin{proof}[Proof of $ \cref{prop:asymp_EM} $]
	\cref{item:prop:asymp_EM:1} is essentially the same statement of \cite[Proposition 3.6]{M:GPPV}. 
	However, for the convenience of the reader, we include a reorganized version of the proof below.
	We prove the claim by reducing it to \cref{lem:EM_asymptotic}.
	
	We need to give the asympotic formula for a variable $ t \in e^{\iu \theta} \R_{>0} $ with $ \theta_0 < \theta < \theta_1 $.
	By replacing $ f(x) $ by $ f(\theta^{-1} x) $, we can assume $ \theta = 0 $.
	Since both sides are $ \R $-linear in $ g(x) $ and $ G(z) $, we can assume
	\[
	g(x) = (x_1 + \alpha_1)^{d_1} \cdots (x_r + \alpha_r)^{d_r} \bm{1}_{a_1 + N_1 \Z_{\ge 0}} (x_1) \cdots \bm{1}_{a_r + N_r \Z_{\ge 0}} (x_r)
	\]
	for some $ d_i, N_i \in \Z_{\ge 0} $, and $ a_i \in \Z $.
	If $ N_i = 0 $ for some $ 1 \le i \le r $, then the claim is trivial in the variable $ x_i $.
	Thus, we can assume $ N_1, \dots, N_r > 0 $.
	By $ \Q $-linearity, we can assume $ N_1 = \cdots N_r \eqqcolon N $.
	In this case, we have
	\begin{align}
		\sum_{m \in \Z^r} g(m) f(t(m + \alpha))
		&=
		\sum_{m \in \Z_{\ge 0}^r} (a_1 + N m_1 + \alpha_1)^{d_1} \cdots (a_r + N m_r + \alpha_r)^{d_r} f(t(a + km + \alpha))
		\\
		&=
		N^{-d_1 - \cdots - d_r} \sum_{m \in \Z_{\ge 0}^r} f_*(t(m + \beta)),
	\end{align}
	where $ f_* (x) \coloneqq N^{-d_1 - \cdots - d_r} x_1^{d_1} \cdots x_r^{d_r} $ and $ \beta \coloneqq N^{-1} a + \alpha $.
	By applying \cref{lem:EM_asymptotic}, we obtain an asymptotic formula
	\[
	\sum_{m \in \Z^r} g(m) f(t(m+\alpha))
	\sim 
	\sum_{j \in \Z_{\ge -1}^r} \frac{B_{j_1 + 1}(\alpha_ 1) \cdots B_{j_r + 1}(\alpha_ r)}{(j_1 + 1)! \cdots (j_r + 1)!}
	f_*^{(j)}(0) t^{j_1 + \cdots + j_r - d_1 - \cdots - d_r}
	\quad \text{ as } t \to +0.
	\]
	For any $ j \in \Z^r $, the derivative $ f_*^{(j)} (0) $ can be written as a $ \Q $-linear combination of $ f^{(j')} (0) $ with $ j' \in \Z^r $ such that
	$ j'_1 + \cdots + j'_r = j_1 + \cdots + j_r - d_1 - \cdots - d_r $.
	Thus, we have
	\[
	\sum_{m \in \Z^r} g(m) f(t(m+\alpha))
	\sim 
	\sum_{j' \in \Z^r} c_{j'} f^{(j')}(0) t^{j'_1 + \cdots + j'_r}
	\quad \text{ as } t \to +0
	\]
	for some $ c_{j'} \in \R $ independent of $ f $.
	
	It suffices to show $ c_{j} = \varphi_{j} $ for any $ j \in \Z^r $.
	Fix arbitrary $ u_1 , \dots, u_r < 0 $.
	In the case when $ f(x) = e^{u_1 x_1 + \cdots + u_r x_r} $, since $ \coe[G](x) = g(x) $, we have
	\[
	\sum_{m \in \Z^r} g(m) f(t(m+\alpha))
	=
	e^{\alpha_1 t u_1 + \dots + \alpha_r t u_r} G (e^{t u_1}, \dots, e^{t u_r})
	=
	\varphi(t u_1, \dots, t u_r).
	\]
	Let $ \varphi(t_1, \dots, t_r) \eqqcolon \sum_{j \in \Z^r} \varphi_{j} t_1^{j_1} \cdots t_r^{j_r} $.
	Then, we have
	\[
	\varphi(t u_1, \dots, t u_r)
	=
	\sum_{j \in \Z^r} \varphi_{j} u_1^{j_1} \cdots u_r^{j_r} t^{j_1 + \cdots + j_r}.
	\]
	On the other hand, we have $ f^{(j)}(0) = u_1^{j_1} \cdots u_r^{j_r} $ for any $ j \in \Z^r $.
	Thus, by comparing asymptotic expansions as $ t \to +0 $, we have $ c_{j} = \varphi_{j} $ for any $ j \in \Z^r $.
	Therefore, we obtain the claim for any $ f $.
	
	We prove \cref{item:prop:asymp_EM:2}.
	By \cref{lem:R_C_corresp} \cref{item:lem:R_C_corresp:vp}, we have
	\[
	\widetilde{g} (x)
	=
	2^{-r} \sum_{e \in \{ \pm 1 \}^r} g(ex),
	\]
	where $ ex \coloneqq (e_1 x_1, \dots, e_r x_r) $.
	Thus, we have
	\[
	\sum_{m \in \Z^r} \widetilde{g} (m) f(t(m + \alpha))
	=
	2^{-r} \sum_{e \in \{ \pm 1 \}^r}
	\sum_{m \in \Z^r} g (em) f (t(m + \alpha)).
	\]
	By replacing $ m $ by $ em $ and letting $ f_{e} (x) \coloneqq f(ex) $, this is equal to
	\[
	2^{-r} \sum_{e \in \{ \pm 1 \}^r}
	\sum_{m \in \Z^r} g (m) f_{e} (t(m + \alpha)).	
	\]
	Define $ \varphi (t_1, \dots, t_r) $ as in \cref{item:prop:asymp_EM:1} and
	let $ \varphi_e (t_1, \dots, t_r) \coloneqq \varphi (e_1 t_1, \dots, e_r t_r) $.
	Then, by \cref{item:prop:asymp_EM:1} we have
	\[
	\sum_{m \in \Z^r} \widetilde{g} (m) f(t(m + \alpha))
	\sim
	2^{-r} \sum_{e \in \{ \pm 1 \}^r}
	\varphi_e \odot f_{e} (t)
	\quad \text{ as } t \to 0.
	\]
	Let
	\[
	\varphi(t_1, \dots, t_r) \eqqcolon 
	\sum_{j \in \Z^r} \varphi_j t_1^{j_1} \cdots t_r^{j_r}.
	\]
	Since 
	\[
	\varphi_e (t_1, \dots, t_r)
	=
	\sum_{j \in \Z^r} e_1^{j_1} \cdots e_r^{j_r} \varphi_j t_1^{j_1} \cdots t_r^{j_r}
	\]
	and $  f_{e}^{(j)} (0) = e_1^{j_1} \cdots e_r^{j_r} f^{(j)} (0) $,
	we have
	\[
	\varphi_e \odot f_{e} (t)
	=
	\sum_{j \in \Z^r} \varphi_j f^{(j)} (0) t^{j_1 + \cdots + j_r}
	=
	\varphi \odot f (t),
	\]
	which proves the claim for $ \widetilde{g} $.
	
	The claim for $ \overline{g} $ follows from the same argument by using 
	\[
	\overline{g} (x)
	=
	2^{-r} \sum_{e \in \{ \pm 1 \}^r} \left( \prod_{1 \le i \le r} e_i \right) g(ex),
	\]
	which proved in \cref{lem:R_C_corresp} \cref{item:lem:R_C_corresp:vp}.
\end{proof}

In \cref{prop:asymp_EM} \cref{item:prop:asymp_EM:2}, we essentially treat signed sums involving $ \sgn(m) $.
The following corollary extends this to general signed sums with $ \sgn(Am) $.

\begin{cor} \label{cor:asymp_EM}
	Let $ r $ and $ s $ be positive integers and $ A \in \Mat_{s, r} (\Z) $ be a matrix of rank $ d $.
	Let $ \overline{g}(x) \in \overline{\frakQ} $ be a quasi-polynomial, 
	$ f \colon \bbC^r \to \bbC $ be a $ C^\infty $ function, and
	$ -\pi/2 < \theta_0 < 0 < \theta_1 < \pi/2 $ be two angles.
	We assume that $ f $ is of rapid decay as $ \abs{x_1}, \dots, \abs{x_r}\to \infty $ 
	on the domain $ \{ x \in \bbC^r \mid \arg(x_1), \dots, \arg(x_r) \in (\theta_0 + \pi, \theta_1 + \pi) \} $.
	Then, for a vector $ \alpha \in \R^r $ and a complex number $ t \in \bbC \smallsetminus \{ 0 \} $ with $ \theta_0 < \arg \left( t \right) < \theta_1 $, we have the asymptotic formula
	\[
	\sum_{m \in \Z^r} \sgn(Am) \overline{g} (m) f(t(m + \alpha))
	\sim
	\sum_{1 \le d' \le d} a_{d'}
	\varphi_{d'} \odot f_{B_{d'}}^{} (t)
	\quad \text{ as } t \to 0,
	\]
	with the following notation:
	\begin{itemize}
		\item Let $ (a_{d'})_{1 \le d' \le d} $ be a family of rational numbers, 
		$ (B_{d'})_{1 \le d' \le d} $ be a family of matrices in $ \Mat_r (\Z) \cap \GL_r (\Q) $, and
		$ (L_{d'})_{1 \le d' \le d} $ be a family of free $ \Z $-submodules with $ \Z^{d'} \subset L_{d'} \subset \Q^{d'} $ which detemined in \cref{prop:sign_Am_to_m}.
		\item Let
		\begin{align}
			\varphi_{d'} (t_1, \dots, t_{d'})
			&\coloneqq
			\sum_{m' \in L_{d'} \cap \Q_{\ge 0}^{d'}} 
			\overline{g} \left( B_{d'} \pmat{m' \\ 0} \right)
			e^{m'_1 t_1 + \cdots + m'_{d'} t_{d'}}
			\in \Q[\alpha_1, \dots, \alpha_r] ((t_1, \dots, t_{d'})).
		\end{align}
		\item Let $ f_{B_{d'}}^{} (x) \coloneqq f \left( B_{d'} \transpose{(x_1, \dots, x_{d'}, 0, \dots, 0)} \right) $.
	\end{itemize}
\end{cor}

The proof follows from \cref{prop:sign_Am_to_m}, which will be proved in \cref{sec:sgn}.
We omit the proof.

\begin{rem} \label{rem:cor:asymp_EM}
	\begin{enumerate}
		\item In the case when $ s=d $, by \cref{lem:sign_Am_to_m}, we have
		\[
		\sum_{m \in \Z^r} \sgn(Am) \overline{g} (m) f(t(m + \alpha))
		\sim
		\varphi \odot f_B^{} (t)
		\quad \text{ as } t \to 0,
		\]
		with the following notation:
		\begin{itemize}
			\item Let $ D(A) \coloneqq \abs{\Z^d / A(\Z^r)} $.
			\item Let $ B \in \Mat_r (\Z) \cap \GL_r (\Q) $ be a matrix determined in \cref{lem:sign_Am_to_m}.
			\item For $ x \in \R^r $, denote $ x' \coloneqq \transpose{(x_1, \dots, x_d)} $.
			\item Let
			\begin{align}
				&\phant
				\varphi(t_1, \dots, t_d)
				\\
				&\coloneqq
				\left(
				\sum_{\mu \in D(A)^{-1} A(\Z^r) / \Z^d} 
				\exp \left( \sum_{1 \le i \le d} (\mu_i + ((A^*)^{-1} \alpha)_i) t_i \right)
				\right)
				\overline{\coe}^{-1} \left[ \overline{g} \left( A^* \pmat{x' \\ 0} \right)  \right] (e^{t_1}, \dots, e^{t_d})
				\\
				&\in \Q[\alpha_1, \dots, \alpha_r] ((t_1, \dots, t_d)).
			\end{align}
			\item Let $ f_B^{} (x) \coloneqq f \left( B \smat{x' \\ 0} \right) $.
		\end{itemize}
		
		\item In the case when $ A = \pmat{I_d & C} $ with the identity matrix $ I_d \in \GL_d(\Z) $, we have 
		\[
		B = \pmat{I_d & -C \\ o & I_{r-d}}.
		\]
		Thus, we have $ f_B^{} (x') = f(x', 0) $ and
		\[
		\varphi(t_1, \dots, t_d)
		=
		\exp \left( \sum_{1 \le i \le d} (\alpha_i + (C \alpha'')_i) t_i \right)
		\overline{\coe}^{-1} \left[ \overline{g} \right] (e^{t_1}, \dots, e^{t_d}, 0, \dots, 0),
		\]
		where $ \alpha'' \coloneqq (\alpha_{d+1}, \dots, \alpha_r) $.
		
	\end{enumerate}
\end{rem}

Next, we discuss the relation between the symmetry of Bernoulli polynomials and the coincidence of the asymptotic expansions of partial modular series and false modular series.

\begin{rem} \label{rem:asymptotic_Bernoulli_poly_symmetry}
	The Bernoulli polynomials have a symmetry $ B_m (1-u) = (-1)^m B_m (u) $,
	which follows from
	\begin{equation} \label{eq:Bernoulli_poly_symmetry}
		\frac{e^{(1-u)t}}{e^t - 1}
		= -\frac{e^{-ut}}{e^{-t} - 1}.
	\end{equation}
	Fix integers $ a $ and $ k $ and let $ r=1 $ and $ g(x) = \bm{1}_{a + N\Z_{\ge 0}} (x) $.
	In this case, we have
	\[
	\coe[G] (x) = \frac{1} {2} \sgn(x-a) \bm{1}_{a + N\Z} (x), \quad
	G(z) = \frac{z^a}{1 - z^N}
	\]
	by \cref{lem:R_C_corresp} \cref{item:lem:R_C_corresp:vp_express}.
	By replacing $ f(x) $ by $ f(x/N) $, we can write the left hand side in \cref{prop:asymp_EM} \cref{item:prop:asymp_EM:1,item:prop:asymp_EM:2} as
	\[
	\sum_{m=0}^\infty f \left( t\left( \frac{a}{N} + m \right) \right), \quad
	\frac{1}{2} \sum_{m=-\infty}^\infty \sgn(m) f \left( t\left( \frac{a}{N} + m \right) \right)
	\]
	respectively.
	By \cref{prop:asymp_EM,rem:coe_Laurent}, these have the same asymptotic expansion as $ t \to +0 $
	\[
	\sum_{j=-1}^\infty \frac{B_{j+1}(a/N)}{(j+1)!} f^{(j)}(0) t^j.
	\]
	Thus, the difference
	\[
	-\sum_{m \le -1} f \left( t\left( \frac{a}{N} + m \right) \right)
	\]
	also has the same asymptotic expansion.
	On the other hand, the last infinite sum can be written as
	\[
	-\sum_{m =0}^\infty f \left( -t\left( 1 - \frac{a}{N} + m \right) \right).
	\]
	By applying \cref{prop:asymp_EM} \cref{item:prop:asymp_EM:1}, this sum has asymptotic expansion
	\[
	\sum_{j=-1}^\infty (-1)^{j+1} \frac{B_{j+1}(1 - a/N)}{(j+1)!} f^{(j)}(0) t^j.
	\]
	Since the asymptotic expansion is unique, we have
	\[
	(-1)^{j+1} B_{j+1} \left( 1 - \frac{a}{N} \right)
	=
	B_{j+1} \left( \frac{a}{N} \right),
	\]
	which is the symmetry of Bernoulli polynomials.
	
	The above argument can be regraded as using $ \bm{1}_{a+N\Z_{\ge -1}} (-m) = \bm{1}_{k-a + N\Z_{\ge 0}} (m) $, which follows from
	\begin{equation} \label{eq:Bernoulli_poly_symmetry2}
		\frac{z^a}{1 - z^N} = -\frac{z^{N-a}}{1 - z^N}
	\end{equation}
	and \cref{lem:R_C_corresp} \cref{item:lem:R_C_corresp:involution}.
	Since \cref{eq:Bernoulli_poly_symmetry,eq:Bernoulli_poly_symmetry2} are same equation, we can conclude that:
	\begin{quote}
		Partial modular series and false modular series share the same asymptotic expansions,
		for the same reason as the symmetry of Bernoulli polynomials.
	\end{quote}
\end{rem}

\begin{rem}
	To consider quantum modular forms or quantum invariants, we need to consider the asymptotic expansions as
	$ \tau \to \alpha $ for any rational number $ \alpha \in \Q $.
	In general, it is impossible to provide such a formula.
	However, when $ \gamma(\tau; x) = q^{\transpose{x} Sx/2} $ for some positive definite symmetric matrix $ S \in \Sym_r^+(\Q) $,
	we can deduce the asymptotic fomula as $ \tau \to \alpha $ of partial and false modular series of the kernel function $ \gamma(\tau; x) $ 
	from \cref{prop:asymp_EM}.
	Indeed, in this case we have $ \gamma(\tau + \alpha; x) = \bm{e} (\alpha\transpose{x} Sx/2) q^{\transpose{x} Sx/2} $.
	Since $ \bm{e} (\alpha\transpose{x} Sx/2) $ is a periodic function on $ \Z^r $,
	for any (partial/false) quasi-polynomial $ g \in \frakQ \cup \overline{\frakQ} \cup \widetilde{\frakQ} $,
	a map $ g(x) \bm{e} (\alpha\transpose{x} Sx/2) $ is also a (resp. partial/false) quasi-polynomial.
	Thus, we can apply \cref{prop:asymp_EM} to 
	$ \Phi[g, \gamma] (\tau + \alpha) = \Phi[g(x) \bm{e} (\alpha\transpose{x} Sx/2), \gamma] (\tau) $
	and obtain the asymptotic fomula as $ \tau \to \alpha $.
\end{rem}

To close this section, we discuss the order of the asymptotic series $ \varphi \odot f (t) $ in \cref{prop:asymp_EM}.
By definition, we have
\[
\ord_{t=0} \varphi \odot f (t)
\ge \ord_{t=0} \varphi(t, \dots, t)
= \ord_{z=1} G(z, \dots, z).
\]
We rephrase this quantity in terms of the false quasi-polynomial $ \widetilde{g}(x) $.

\begin{dfn}
	\begin{enumerate}
		\item We define the degree map $ \deg \colon \overline{\frakQ} \to \Z_{\ge -r} $ of a quasi-polynomial as follows.
		Writing a quasi-polynomial $ \overline{g}(x) \in \overline{\frakQ} $ as
		\[
		\overline{g}(x)
		=
		\sum_{I \subset \{ 1, \dots, r \}} C_I^{} (x_I^{}) P_I^{} (x_I^{}) \bm{1}_{A_{I^\complement}^{}}^{} ( x_{I^\complement}^{} )
		\]
		for some period map $ C_I^{} \colon (\Z/N_I^{} \Z)^I \to \Q $ with positive integer $ N_I^{} \in \Z_{>0} $,
		polynomial $ P_I^{} (x_I^{}) \in \Q[ x_I^{} ] $, and
		a finite set $ A_{I^\complement}^{} \subset \Z^{I^\complement} $,
		where $ I^\complement \coloneqq I \smallsetminus \{ 1, \dots, r \} $.
		We define
		\[
		\deg \overline{g}
		\coloneqq
		\max \left\{ \deg P_I^{} - \abs{I^\complement} \mid I \subset \{ 1, \dots, r \} \right\}.
		\]
		
		\item Similarly, we define the degree map $ \deg \colon \widetilde{\frakQ} \to \Z_{\ge -r} $ of a false quasi-polynomial as 
		\[
		\deg \widetilde{g}
		\coloneqq
		\max \left\{ \deg P_I^{} - \abs{I^\complement} \mid I \subset \{ 1, \dots, r \}, 1 \le j \le M \right\}
		\]
		for
		\[
		\widetilde{g}(x)
		=
		\sum_{I \subset \{ 1, \dots, r \}} \sum_{1 \le j \le M} \sgn( x - b_j )
		C_{I, j}^{} (x_I^{}) P_{I, j}^{} (x_I^{}) \bm{1}_{A_{I^\complement, j}^{}}^{} ( x_{I^\complement}^{} ),
		\]
		where $ b_j \in \Z^r $ and other data are as above.
	\end{enumerate}
\end{dfn}

\begin{rem} \label{rem:deg_q-poly}
	\begin{enumerate}
		\item \label{item:rem:deg_q-poly:comm}
		The above two degree maps are commutative via the isomorphism $ \overline{\frakQ} \cong \widetilde{\frakQ} $ in \cref{lem:R_C_corresp} \cref{item:lem:R_C_corresp:R_C_tilde}.
		
		\item \label{item:rem:deg_q-poly:inequality}
		For any $ g(x), h(x) \in \overline{\frakQ} $ (resp.~$ g(x), h(x) \in \widetilde{\frakQ} $), we have
		$ \deg (g+h) \le \max \{ \deg g, \deg h \} $.
		
		\item \label{item:rem:deg_q-poly:order}
		For $ G(z_1, \dots, z_r) \in \frakR $, we have
		$ \ord_{z=1} G(z, \dots, z) = -\deg \overline{\coe}[G] - r = -\deg \widetilde{\coe}[G] - r $.
		Thus, for the asymptotic series $ \varphi \odot f (t) $ in \cref{prop:asymp_EM}, we have
		\[
		\ord_{t=0} \varphi \odot f (t) \ge \deg -\widetilde{g} - r. 
		\]
	\end{enumerate}
\end{rem}


\subsection{Asymptotic expansion formula of stationary phase type} \label{subsec:asymptotic_stationary_phase}


Since modular series have integral representations, we can reformulate the above asymptotic formula (\cref{prop:asymp_EM}) in terms of integrals as follows.

\begin{prop} \label{prop:asymp_stationary_phase}
	Let $ f \colon \R^r \to \bbC $ be a continuous function of exponential decay as $ \abs{x} \to \infty $ and real analytic at $ 0 \in \R^r $,
	$ \R^r \subset U \subset \bbC^r $ be an open subset, 
	$ \varphi \colon U \to \bbC $ be a meromorphic function such that all its poles lie on $ \R^r $ and 
	for some $ K, c > 0 $ it satisfies $ \abs{\varphi(\xi)} \le K e^{c\abs{\xi}} $ for any $ \xi \in U $.
	We assume that the Fourier transform $ \widehat{f} $ satisfies the following conditions:
	\begin{enumerate}
		\item \label{item:prop:asymp_stationary_phase:1}
		For any sufficiently small $ \varepsilon > 0 $ and any $ m \in \Z_{\ge 0}^r $, an integral
		\[
		\int_{(\R + \iu \varepsilon)^r} \widehat{f} (\xi) \frac{d\xi}{\xi_1^{m_1} \cdots \xi_r^{m_r}}
		\]
		converges.
		
		\item \label{item:prop:asymp_stationary_phase:2}
		There exist $ \lambda \colon \R_{>0} \to \R $, $ h \colon U \to \R_{>0} $ and $ K' > 0 $ such that the followings hold:
		\begin{enumerate}
			\item \label{item:prop:asymp_stationary_phase:2a}
			For any $ \xi \in U $ and $ u>0 $, we have
			$ \abs{\widehat{f} (u\xi)} \le K' e^{-\lambda(u) g(\xi)} $.
			
			\item \label{item:prop:asymp_stationary_phase:2b}
			We have $ e^{-\lambda(u)} = O(u^{-R}) $ as $ u \to \infty $ for any $ R>0 $.
			
			\item \label{item:prop:asymp_stationary_phase:2c}
			For any sufficiently small $ \varepsilon' > 0 $, we have
			$ \mu \coloneqq \inf \{  g(\xi) \mid \xi \in U, \abs{\Re(\xi)} \ge \varepsilon' \} $.
			
			\item \label{item:prop:asymp_stationary_phase:2d}
			For any sufficiently small $ \varepsilon' > 0 $, there exists $ c' > 0 $ such that for any $ e \in \{ \pm 1 \}^r $, we have
			\[
			\int_{\R^r + \iu \varepsilon' e} e^{-c' g(\xi) + c \abs{\xi}} d\xi < \infty.
			\]
		\end{enumerate}
	\end{enumerate}
	Then, for any sufficiently small $ \varepsilon > 0 $, we have following asymptotic formulas
	\begin{align}
		\int_{(\R + \iu \varepsilon)^r} \widehat{f}(u\xi) \varphi(x) dx
		&\sim 
		\varphi \odot f \left( \frac{1}{2\pi\iu u} \right) \cdot u^{-r}
		\quad \text{ as } u \to \infty,
		\\
		\PV \int_{\R^r} \widehat{f}(u\xi) \varphi(x) dx
		&\sim 
		\varphi \odot f \left( \frac{1}{2\pi\iu u} \right) \cdot u^{-r}
		\quad \text{ as } u \to \infty.
	\end{align}
\end{prop}

In this proposition, the variable $ u $ corresponds to $ t^{-1} $ in \cref{prop:asymp_EM}.
This formula is an extension of the stationary phase approximation.

To prove this proposition, we need the following lemma.

\begin{lem} \label{lem:Fourier_derivative}
	Let $ f \colon \R \to \bbC $ be a continuous function of exponential decay as $ \abs{x} \to \infty $
	and $ \widehat{f} $ be its Fourier transform.
	Then, for any sufficiently small $ \varepsilon > 0 $, we have
	\[
	f^{(-1)} (x) =
	\int_{\R + \iu \varepsilon} \widehat{f} (\xi) \bm{e} (\transpose{\xi} x) \frac{d\xi}{2\pi\iu \xi}
	\]
	if the right hand side converges uniformly.
	As a consequence, for any $ m \in \Z $, we have
	\[
	f^{(m)} (x) =
	\int_{\R + \iu \varepsilon} \widehat{f} (\xi) \bm{e} (\transpose{\xi} x) (2\pi\iu \xi)^m d\xi
	\]
	if the right hand side converges uniformly.
\end{lem}

\begin{proof}
	First, we remark that the Fourier transform $ \widehat{f}(\xi) $ extends holomorphically to a strip $ \R^r + \iu (-a/2\pi, a/2\pi)^r $
	since $ f(x) $ decays exponentially.
	Let $ g(x) $ denotes the right hand side.
	Its derivative coincides with $ f(x) $ by the Fourier inversion formula and continuity of $ f $.
	Since $ \bm{e} (\transpose{\xi} x) \to 0 $ as $ x \to \infty $ for $ \xi \in \R + \iu \varepsilon $,
	we have $ g(x) \to 0 $ as $ x \to \infty $.
	Since $ f^{(-1)} (x) $ has the same properties, we obtain the claim.
\end{proof}

\begin{proof}[Proof of $ \cref{prop:asymp_stationary_phase} $]
	First, we prove the first asymptotic formula.
	When $ \varphi(\xi) = \xi_1^{j_1} \cdots \xi_r^{j_r} $ for some $ j \in \Z^r $,
	by \cref{lem:Fourier_derivative} and the condition \cref{item:prop:asymp_stationary_phase:1}, we have
	\[
	\int_{(\R + \iu \varepsilon)^r} \widehat{f} (u\xi) \varphi(\xi) d\xi
	=
	\int_{(\R + \iu \varepsilon)^r} \widehat{f} (u\xi) \xi_1^{j_1} \cdots \xi_r^{j_r} d\xi
	=
	f^{(j)} (0) (2\pi\iu u)^{-j_1 - \cdots - j_r} u^{-r}.
	\]
	Thus, our asymptotic formula is proved, and in this case, it is an equality.
	
	We now turn to the general case of $ \varphi(\xi) $.
	Fix an arbitrary $ l \in \Z_{\ge 0} $ and let
	\[
	\varphi_{>l} (\xi) \coloneqq
	\varphi(\xi) - \sum_{j \in \Z^r, \, j_1 + \cdots + j_r \le l} \varphi_j \xi_1^{j_1} \cdots \xi_r^{j_r},
	\quad \text{ where } 
	\varphi (\xi) \eqqcolon
	\sum_{j \in \Z^r} \varphi_j \xi_1^{j_1} \cdots \xi_r^{j_r}.
	\]
	The preceding argument reduces the proof of the asymptotic formula to establishing the estimate
	\begin{equation} \label{eq:proof_stationary_phase}
		\int_{\R^r + \iu \varepsilon} \widehat{f} (u\xi) \varphi_{>l} (\xi) d\xi
		= O(u^{- l - r - 1}) \quad \text{ as } u \to \infty.
	\end{equation}
	Since the integrand is holomorphic at $ \xi = 0 $, the contour $ \R + \iu \varepsilon $ can be deformed into 
	$ C_{\varepsilon'} $ and $ (-\varepsilon', \varepsilon') $ for any sufficiently small $ \varepsilon' > 0 $,
	where $ C_{\varepsilon'} $ is an integration path shown in \cref{fig:int_path_stationoary_phase}.
	
	\begin{figure}[htbp]
		\centering
		\begin{tikzpicture}
			\tikzset{midarrow/.style={postaction={decorate},
					decoration={markings,
						mark=at position 0.5 with {\arrow{stealth}},
					}
				}
			}
			\draw[->] (-3, 0) -- (3, 0) node[right]{$ \Re (x_i) $};
			\draw[->] (0, -2) -- (0, 2) node[above]{$ \Im (x_i) $};
			
			\draw (0,0) node[below left]{0};
			
			
			\draw[very thick, midarrow] (-3,-0.5) node[below]{$ -\infty - \iu \varepsilon $}
			--(-1,-0.5);
			
			\draw[very thick, midarrow] (-1,-0.5) arc (-90:0:0.5);
			\node[above] at (-0.5,0) {$ -\varepsilon' $};
			
			\draw[very thick, midarrow] (0.5,0) arc (-180:-90:0.5);
			\node[above] at (0.5,0) {$ \varepsilon' $};
			
			\draw[very thick, midarrow] (1,-0.5)--(3,-0.5) node[below]{$ \infty + \iu \varepsilon $};
			
			\node at (0.5,-0.8) {$ C_{\varepsilon'} $};
			
		\end{tikzpicture}
		\caption{The integration path $ C_{\varepsilon'} $}
		\label{fig:int_path_stationoary_phase}
	\end{figure}
	
	For sufficiently small $ \varepsilon' > 0 $, there exists $ K' >0 $ such that
	$ \abs{\varphi_{>l} (\xi)} \le K' \abs{\xi}^{l+1} $ for any $ \xi \in (-\varepsilon', \varepsilon')^r $.
	Then, we have
	\[
	\abs{
		\int_{(-\varepsilon', \varepsilon')^r} \widehat{f} (u\xi) \varphi_{>l} (\xi) d\xi
	}
	\le
	K' u^{-l-r-1} \int_{(-\varepsilon', \varepsilon')^r} \abs{ \widehat{f} (\xi) \xi^{l+1} } d\xi.
	\]
	Thus, to prove \cref{eq:proof_stationary_phase}, it suffices to show 
	\[
	\int_{C_{\varepsilon'}^r} \widehat{f} (u\xi) \varphi_{>l} (\xi) d\xi
	= O(u^{-R})
	\quad \text{ as } u \to \infty
	\]
	for any $ R>0 $.
	By the condtion \cref{item:prop:asymp_stationary_phase:2} \cref{item:prop:asymp_stationary_phase:2a}, we have
	\begin{align}
		\abs{
			\int_{C_{\varepsilon'}^r} \widehat{f} (u\xi) \varphi_{>l} (\xi) d\xi
		}
		&\le
		KK' \int_{C_{\varepsilon'}^r} e^{-\lambda(u) g(\xi) + c \abs{\xi}} d\xi
		\\
		&=
		KK' \int_{C_{\varepsilon'}^r} e^{-(\lambda(u) - c') g(\xi)} e^{-c' g(\xi) + c \abs{\xi}} d\xi.
	\end{align}
	Since $ \lambda(u) \to \infty $ as $ u \to \infty $ by the condtion \cref{item:prop:asymp_stationary_phase:2} \cref{item:prop:asymp_stationary_phase:2b}, 
	there exists $ u_0 > 0 $ such that for any $ u > u_0 $, we have $ \lambda(u) > c' $.
	Thus, by the condtion \cref{item:prop:asymp_stationary_phase:2} \cref{item:prop:asymp_stationary_phase:2c}, for any $ u > u_0 $, the above integral is bounded by
	\[
	\le KK' e^{-(\lambda(u) - c') \mu} \int_{C_{\varepsilon'}^r} e^{-c' g(\xi) + c \abs{\xi}} d\xi.
	\]
	The last integral converges by the condition \cref{item:prop:asymp_stationary_phase:2} \cref{item:prop:asymp_stationary_phase:2d}.
	Thus, by the condtion \cref{item:prop:asymp_stationary_phase:2} \cref{item:prop:asymp_stationary_phase:2b}, the above integral has the estimate
	$ O(u^{-R}) $ as $ u \to \infty $ for any $ R>0 $.	
	Therefore, we obtain the first asymptotic formula.
	
	The second asymptotic formula follows from the same argument in the proof of \cref{prop:asymp_EM}.
\end{proof}

\begin{ex}
	We observe that our formula also covers cases beyond the stationary phase approximation.
	All assumptions in \cref{prop:asymp_stationary_phase} hold when
	$ f(x) = 1/\cosh(\pi x) $ and
	$ \varphi(\xi) = G(\bm{e} (\xi_1), \dots, \bm{e} (\xi_1)) $ for a cyclotomic rational function $ G(z) \in \frakR $.
	In this case, our asymptotic formula in \cref{prop:asymp_stationary_phase} is not a type of stationary phase approximation 
	since $ \widehat{f}(u\xi) = 1/\cosh(\pi u\xi) $ cannot be written down as a form $ e^{-u h(\xi)} $.
\end{ex}

\begin{ex}
	All assumptions in \cref{prop:asymp_stationary_phase} hold when
	$ f(x) = e^{-\pi \transpose{x}Sx} $ for a positive definite symmetric matrix $ S \in \Sym_r^+ (\Z) $ and
	$ \varphi(\xi) = G(\bm{e} (\xi_1), \dots, \bm{e} (\xi_1)) $ for a cyclotomic rational function $ G(z) \in \frakR $.
	In this case, our asymptotic formula is an extension of the following stationary phase approximation formula of Gaussian type.
\end{ex}

\begin{lem}[{\cite[Chapter 7, Lemma 7.7.3]{Hormander}}] \label{lem:stationary_phase}
	Let $ S \in \Sym_r (\bbC) $ be a non-degenerate symmetric matrix such that $ \Re(S) $ is positive semidefinite matrix and
	$ \varphi \colon \R^r \to \bbC $ be a $ C^\infty $-function satisfying that 
	there exists $ r_0 \in \R $ such that for any $ r_0 < u \in \R $, the integral
	\[
	\int_{\R^r} e^{-\pi u \transpose{\xi} S^{-1} \xi} \varphi(\xi) d\xi
	\]
	converges absolutely.
	Then, we have an asymptotic expansion
	\[
	\int_{\R^r} e^{-\pi u \transpose{\xi} S^{-1} \xi} \varphi(\xi) d\xi
	\sim 
	\frac{u^{-r/2}}{\sqrt{\det S}}
	\sum_{j=0}^{\infty} \frac{(4\pi u)^{-j}}{j!} \sprod{SD, D}^j \varphi (0)
	\quad \text{ as } u \to \infty,
	\]
	where
	\[
	S = (S_{i, j})_{1 \le i, j \le r},
	\quad
	D \coloneqq \left( \frac{\partial}{\partial x_i} \right)_{1 \le i \le r},
	\quad
	\sprod{SD, D}
	\coloneqq
	\sum_{1 \le i, j \le r} a_{i, j} \frac{\partial^2}{\partial x_i \partial x_j}.
	\]
\end{lem}

Here we remark that the right hand side of the above asymptotic formula can be written as the same form in \cref{prop:asymp_EM} as follows.

\begin{lem} \label{lem:asymp_exp_comparison}
	In the setting of \cref{lem:stationary_phase}, we have
	\[	
	\sum_{j=0}^{\infty} \frac{\left( 4\pi u \right)^{-j}}{j!} \sprod{SD, D}^j \varphi (0)
	=
	\varphi \left( -\frac{x}{2\pi\iu} \right) \odot e^{-\pi \transpose{x} S x} (u^{-1/2}).
	\]
\end{lem}

\begin{proof}
	We can calculate as follows:
	\begin{align}
		&\phant
		\varphi \left( -\frac{x}{2\pi\iu} \right) \odot e^{-\pi \transpose{x} S x} (u^{-1/2})
		\\
		&=
		\sum_{m \in \Z_{\ge 0}^r} 
		\left( -2\pi\iu \right)^{-m_1 - \cdots - m_r}
		\frac{\varphi^{(m)} (0)}{m_1 ! \cdots m_r !}
		\restrict{
			\frac{\partial^{m_1 + \cdots + m_r}}{\partial x_1^{m_1} \cdots \partial x_r^{m_r}}
			e^{-\pi \transpose{x} S x}
		}{x=0}
		u^{-(m_1 + \cdots + m_r)/2}
		\\
		&=
		\sum_{m \in \Z_{\ge 0}^r}
		\left( -2\pi\iu \sqrt{u} \right)^{-m_1 - \cdots - m_r}
		\frac{\varphi^{(m)} (0)}{m_1 ! \cdots m_r !}
		\restrict{
			\frac{\partial^{m_1 + \cdots + m_r}}{\partial x_1^{m_1} \cdots \partial x_r^{m_r}}
			\left(
			\sum_{l=0}^\infty \frac{(-\pi \transpose{x} S x)^l}{l!}
			\right)
		}{x=0}
		\\
		&=
		\sum_{l=0}^\infty \frac{(-\pi)^l}{l!}
		\sum_{\substack{
				m \in \Z_{\ge 0}^r, \\
				m_1 + \cdots + m_r = 2l
		}}
		\left( -2\pi\iu \sqrt{u} \right)^{-m_1 - \cdots - m_r}
		\frac{\varphi^{(m)} (0)}{m_1 ! \cdots m_r !}
		\restrict{
			\frac{\partial^{l}}{\partial x_1^{m_1} \cdots \partial x_r^{m_r}}
			(\transpose{x} S x)^l
		}{x=0}
		\\
		&=
		\sum_{l=0}^\infty \frac{(4\pi u)^{-l}}{l!}
		\sum_{\substack{
				m \in \Z_{\ge 0}^r, \\
				m_1 + \cdots + m_r = 2l
		}}
		\frac{\varphi^{(m)} (0)}{m_1 ! \cdots m_r !}
		\restrict{
			\frac{\partial^{l}}{\partial x_1^{m_1} \cdots \partial x_r^{m_r}}
			\left(
			\sum_{\substack{
					(n_{i, j})_{1 \le i, j \le r} \in \Z_{\ge 0}^{r \times r}, \\
					\sum_{1 \le i, j \le r} n_{i, j} = l
			}}
			l! \prod_{1 \le i, j \le r} \frac{(S_{i, j} x_i x_j)^{n_{i, j}}}{n_{i, j}!}
			\right)
		}{x=0}
		\\
		&=
		\sum_{l=0}^\infty (4\pi u)^{-l}
		\sum_{\substack{
				m \in \Z_{\ge 0}^r, \\
				m_1 + \cdots + m_r = l
		}}
		\varphi^{(m)} (0)
		\sum_{\substack{
				(n_{i, j})_{1 \le i, j \le r} \in \Z_{\ge 0}^{r \times r}, \\
				m_i = \sum_{1 \le j \le r} (n_{i, j} + n_{j, i}) \text{ for each } 1 \le i \le r
		}}
		\prod_{1 \le i, j \le r} \frac{S_{i, j}^{n_{i, j}}}{n_{i, j} !}
		\\
		&=
		\sum_{l=0}^\infty (4\pi u)^{-l}
		\sum_{\substack{
				m \in \Z_{\ge 0}^r, \\
				m_1 + \cdots + m_r = 2l
		}}
		\varphi^{(m)} (0)
		\sum_{\substack{
				(n_{i, j})_{1 \le i, j \le r} \in \Z_{\ge 0}^{r \times r}, \\
				m_i = \sum_{1 \le j \le r} (n_{i, j} + n_{j, i}) \text{ for each } 1 \le i \le r
		}}
		\prod_{1 \le i, j \le r} \frac{S_{i, j}^{n_{i, j}}}{n_{i, j} !}
		\\
		&=
		\sum_{l=0}^\infty (4\pi u)^{-l}
		\sum_{\substack{
				(n_{i, j})_{1 \le i, j \le r} \in \Z_{\ge 0}^{r \times r}, \\
				\sum_{1 \le i, j \le r} n_{i, j} = l
		}}
		\left(
		\prod_{1 \le i, j \le r} \frac{S_{i, j}^{n_{i, j}}}{n_{i, j} !}
		\frac{\partial^{n_{i, j} + n_{j, i}}}{\partial x_i^{n_{i, j} + n_{j, i}}}
		\right)
		\varphi (0)
		\\
		&=
		\sum_{l=0}^\infty (4\pi u)^{-l}
		\sum_{\substack{
				(n_{i, j})_{1 \le i, j \le r} \in \Z_{\ge 0}^{r \times r}, \\
				\sum_{1 \le i, j \le r} n_{i, j} = l
		}}
		\left(
		\prod_{1 \le i, j \le r} \frac{1}{n_{i, j} !} 
		\left(
		S_{i, j}
		\frac{\partial^{2}}{\partial x_i \partial x_j}
		\right)^{n_{i, j}}
		\right)
		\varphi (0)
		\\
		&=
		\sum_{l=0}^\infty \frac{(4\pi u)^{-l}}{l!}
		\sprod{SD, D}^l
		\varphi (0).
	\end{align}
\end{proof}


\part{Application to quantum modularity and asymptotic expansions} \label{part:quantum_modularity_asymptotics}


In this part, we apply the results in \cref{part:modular_series} for many cases.

\section{Quantum modularity and asymptotic expansions of false theta functions} \label{sec:false_theta}


In this section, we derive the modular transformation formula for false theta functions
as an application of the Poisson summation formula with signature (\cref{thm:PSF_sgn}) or the modular transformation formula for false modular series (\cref{thm:false_modular_series_modular}) prepared in \cref{part:modular_series}.


\subsection{Previous results} \label{subsec:false_theta_previous_results}


To begin with, we discuss previous works on modular transformation formulas for false theta functions.

A pioneering contribution is due to Lawrence--Zagier~\cite{Lawrence-Zagier}, who studied Witten's asymptotic expansion conjecture for the Poincar\'{e} homology sphere.
For rank one false theta functions, they constructed non-holomorphic Eichler integrals whose asymptotic expansions at all rational numbers agree with those of false theta functions and derived their modular transformation formulas.
As a consequence, they proved quantum modularity of the radial limits of rank-one false theta functions.
They did not, however, provide a modular transformation formula for the false theta functions themselves.


The first explicit modular transformation formulas for false theta functions were given by Creutzig--Milas~\cite[Theorem 4]{Creutzig-Milas} in the context of characters of vertex algebras.
Their proof is based on the idea of \emph{regularization}, in which one considers the integral of $ q^{\xi^2 /2} / \sin (\pi(\xi + \iu \varepsilon)) $ along $ \R $.

Bringmann--Nazaroglu~\cite[Theorem 1.2]{Bringmann-Nazaroglu} introduced \emph{modular completions} for false theta functions of general rank and established their modular transformation formulas.
Their approach is based on the method of constructing modular completions for indefinite theta functions by Zwegers~\cite{Zwegers:thesis}.
The approach by Bringmann--Nazaroglu~\cite{Bringmann-Nazaroglu} was later extended by
Bringmann--Kaszian--\linebreak[0]Milas--\linebreak[0]Nazaroglu~\cite[Theorem 1.1]{BKMN:False_modular} to what they call \emph{higher depth false modular forms}.
Building on the approach by  Bringmann--Nazaroglu~\cite{Bringmann-Nazaroglu}, Matsusaka--Terashima~\cite[Proposition 3.9 and 3.14]{Matsusaka-Terashima} derived modular transformation formulas for rank one false theta functions with polynomial factors in summands, not studied in earlier works, and used them to give an alternative proof of Witten's asymptotic expansion conjecture for Seifert homology spheres.

In a different approach, Andersen--Misteg\aa{}rd~\cite[p.~751]{Andersen-Mistegard} essentially obtained modular transformation formulas for rank one false theta functions with polynomial factors, carrying this out in particular for the GPPV invariants of Seifert homology spheres.
Their method is based on ideas from resurgence theory.
Han--Li--Sauzin--Sun~\cite[Theorem 3]{Han-Li-Sauzin-Sun} also proved similar formulas in a more systematic form using the same method.


\subsection{Rank two case} \label{subsec:false_theta_rk2}


For the reader's convenience, we first state the explicit modular transformation formulas for false theta functions in a simple rank-two case.
The general case will be treated in the next subsection.

Throughout this subsection, we fix a positive definite symmetric matrix $ S = \smat{a & b \\ b & c} \in \Sym_2^+ (\Z) $
and a vector $ \alpha \in \R^2 $.
For these data, we define a false theta function as
\[
\widetilde{\Theta} (\tau)
\coloneqq
\sum_{m \in \alpha + \Z^2} \sgn(m_1) \sgn(m_2) q^{\transpose{m} S m/2}.
\]
We denote $ D \coloneqq \det S \in \Z_{>0} $ and $ \widetilde{\bbC} $ the universal cover of $ \bbC \smallsetminus \{ 0 \} $.

Our modular transformation formula is as follows.

\begin{thm} \label{thm:false_theta_rk2_modular}
	We have
	\begin{align}
		\frac{\iu}{\tau} \sqrt{D} \cdot
		\widetilde{\Theta} \left(  -\frac{1}{\tau}\right)
		=
		\widetilde{\Theta}^* (\tau)
		&+ \sum_{\mu \in a^{-1} \Z / \Z} \left( 
		\widetilde{\Theta}_{1, \mu} (\tau) 
		+ \widetilde{\theta}_{1, \mu} (\tau) \Omega_{1, \mu} (\tau)
		\right)
		\\
		&+ \sum_{\mu \in c^{-1} \Z / \Z} \left( 
		\widetilde{\Theta}_{2, \mu} (\tau) 
		+ \widetilde{\theta}_{2, \mu} (\tau) \Omega_{2, \mu} (\tau)
		\right)
		+ \Omega (\tau),
	\end{align}
	where
	\begin{align}
		\widetilde{\Theta}^* (\tau)
		&\coloneqq
		\sum_{n \in \Z^2} \sgn(n_1) \sgn(n_2) \bm{e} (\transpose{\alpha} m) q^{\transpose{n} S^{-1} n/2},
		\\
		\widetilde{\Theta}_{1, \mu} (\tau)
		&\coloneqq
		\sum_{n \in \transpose{(\mu, -b \mu)} + \Z^2} 
		\left( \sgn(n_2) - \sgn(n_2 - b n_1) \right)
		\bm{e} ((a \alpha_1 + b \alpha_2) n_1 + \alpha_2 n_2) q^{a (n_1^2 + n_2^2/D)/2},
		\\
		\widetilde{\theta}_{1, \mu} (\tau)
		&\coloneqq
		\sum_{n_1 \in \mu + \Z} \sgn(n_1) \bm{e} (a \alpha_1 n_1) q^{a n_1^2 /2},
		\\
		\Omega_{1, \mu} (\tau)
		&\coloneqq
		2 \int_{C_+} \bm{e} (\alpha_2 \xi_2) q^{a \xi_2^2 / 2D} \frac{d\xi_2}{1 - \bm{e} (\xi_2 + \mu)},
		\\
		\widetilde{\Theta}_{2, \mu} (\tau)
		&\coloneqq
		\sum_{n \in \transpose{(-b \mu, \mu)} + \Z^2} 
		\left( \sgn(n_1) - \sgn(n_1 - b n_2) \right)
		\bm{e} (\alpha_1 n_1 + (b \alpha_1 + c \alpha_2) n_2) q^{c (n_1^2 /D + n_2^2)/2},
		\\
		\widetilde{\theta}_{2, \mu} (\tau)
		&\coloneqq
		\sum_{n_2 \in \mu + \Z} \sgn(n_2) \bm{e} (c \alpha_2 n_2) q^{c n_2^2 /2},
		\\
		\Omega_{2, \mu} (\tau)
		&\coloneqq
		2 \int_{C_+} \bm{e} (\alpha_1 \xi_1) q^{c \xi_1^2 / 2D} \frac{d\xi_1}{1 - \bm{e} (\xi_1 + \mu)},
		\\
		\Omega (\tau)
		&\coloneqq
		4 \int_{C_+^2} \bm{e} (\transpose{\alpha} \xi) q^{\transpose{\xi} S^{-1} \xi / 2}
		\frac{d\xi_1}{1 - \bm{e} (\xi_1)} \frac{d\xi_2}{1 - \bm{e} (\xi_2)}.
	\end{align}
	Moreover, the functions $ \Omega_{1, \mu} (\tau), \Omega_{2, \mu} (\tau) $ and $ \Omega (\tau) $
	extend to holomorphic functions on $ \widetilde{\bbC} $.
\end{thm}

The last statement follows from the following lemma.

\begin{lem} \label{lem:analy_cont}
	Let $ S \in \Sym_r^+ (\Q) $ be a positive definite symmetric matrix and
	$ \varphi(\xi) $ be a holomorphic function on an open set containing $ C_+^r $, 
	which satisfies the estimate $ \varphi(\xi) = O(e^{c \abs{\xi}})$ as $ \abs{\xi} \to \infty $ for some constant $ c > 0 $.
	Then, the integral
	\[
	\int_{C_+^r} q^{\transpose{\xi} S^{-1} \xi / 2} \varphi(\xi) d\xi
	\]
	extends to a holomorphic function in $ \tau $ on $ \widetilde{\bbC} $.
\end{lem}

\begin{proof}
	We deform the integration path $ C_+ $ into $ C_\varepsilon $ defined as in \cref{fig:contour_lem_analy_cont}.
	For $ \tau \in \bbC \smallsetminus \{ 0  \} $ such that 
	$ -2\varepsilon + \varepsilon' < \arg(\tau) < \pi $ for some $ \varepsilon' > 0 $, 
	we have
	\begin{alignat}{2}
		q^{\transpose{\xi} S^{-1} \xi / 2} \varphi(\xi)
		&=
		O \left( \exp \left( -\abs{\tau \transpose{\xi} S^{-1} \xi} \sin (2\varepsilon + \arg(\tau) ) + c \abs{\xi} \right) \right) & &
		\\
		&=
		O \left( \exp \left( -\abs{\tau \transpose{\xi} S^{-1} \xi} \sin \varepsilon' \right) \right)
		& &\quad \text{ as } \abs{\xi} \to \infty.
	\end{alignat}
	Therefore, the integral extends holomorphically on this domain.
	By repeatedly deforming the integration path, we can analytically continue the integral to $ \widetilde{\bbC} $.
\end{proof}

\begin{figure}[htbp]
	\centering
	\begin{tikzpicture}
		\tikzset{midarrow/.style={postaction={decorate},
				decoration={markings,
					mark=at position 0.15 with {\arrow{stealth}},
					mark=at position 0.5 with {\arrow{stealth}},
					mark=at position 0.85 with {\arrow{stealth}}
				}
			}
		}
		\draw[->] (-3, 0) -- (3, 0) node[right]{$ \Re (\xi_i) $};
		\draw[->] (0, -2) -- (0, 2) node[above]{$ \Im (\xi_i) $};
		
		\draw (0,0) node[below left]{0};
		
		\draw[very thick, midarrow] (200:3.2) node[below]{$ -e^{\iu \varepsilon} \infty $}
		-- (200:0.5) arc[start angle=200, end angle=20, radius=0.5] 
		-- (20:3.2) node[above]{$ e^{\iu \varepsilon} \infty $};
		
		\node at (1,1) {$C_{\varepsilon}$};
		
	\end{tikzpicture}
	\caption{The integration path $ C_\varepsilon $}
	\label{fig:contour_lem_analy_cont}
\end{figure}

\begin{proof}[Proof of \cref{thm:false_theta_rk2_modular}]
	We need to prove the modular transformation formula.
	We can assume $ \alpha \in \left[ 0, 1 \right)^2 $.
	Let $ \gamma (\tau; x) \coloneqq q^{\transpose{x} S x/2} $ be a kernel function.
	We can write 
	\[
	\widetilde{\Theta}(\tau) = \sum_{m \in \Z^2} \sgn(m_1) \sgn(m_2) \gamma(\tau; m + \alpha).
	\]
	We can calculate as
	\[
	\calF \left[ \gamma \left( -\frac{1}{\tau}; x + \alpha \right) \right] (\xi)
	=
	\frac{\tau}{\iu} \frac{1}{\sqrt{D}} \cdot 
	\bm{e} (\transpose{\alpha} \xi) q^{\transpose{\xi} S^{-1} \xi/2}.
	\]
	Thus, by the Poisson summation formula with signature (\cref{thm:PSF_sgn}) or the modular transformation formula for false modular series (\cref{thm:false_modular_series_modular}), we have
	\[
	\frac{\iu}{\tau} \sqrt{D} \cdot
	\widetilde{\Theta} \left(  -\frac{1}{\tau}\right)
	=
	\widetilde{\Theta}^* (\tau)
	+ \calI_1(\tau) + \calI_2 (\tau)
	+ \Omega (\tau),
	\]
	where
	\begin{align}
		\calI_1(\tau) &\coloneqq
		\sum_{n_1 \in \Z} \sgn(n_1) \cdot 2 \int_{C_+} 
		\bm{e} (\alpha_1 n_1 + \alpha_2 \xi_2) q^{(n_1, \xi_2) S^{-1} \transpose{(n_1, \xi_2)} /2}
		\frac{d\xi_2}{1 - \bm{e} (\xi_2)},
		\\
		\calI_2(\tau) &\coloneqq
		\sum_{n_2 \in \Z} \sgn(n_2) \cdot 2 \int_{C_+} 
		\bm{e} (\alpha_1 \xi_1 + \alpha_2 n_2) q^{(\xi_1, n_2) S^{-1} \transpose{(\xi_1, n_2)} /2}
		\frac{d\xi_1}{1 - \bm{e} (\xi_1)}.
	\end{align}
	We need to calculate $ \calI_1(\tau) $ and $ \calI_2(\tau) $.
	Since
	\[
	\transpose{\xi} S^{-1} \xi
	=
	\frac{a}{D} \left( \xi_2 - \frac{b}{a} \xi_1 \right)^2 + \frac{1}{a} \xi_1^2,
	\]
	we have
	\begin{align}
		\calI_1(\tau) 
		&=
		\sum_{n_1 \in \Z} \sgn(n_1) \bm{e} (\alpha_1 n_1) q^{n_1^2 /2a}
		\cdot 2 \int_{C_+} 
		\bm{e} (\alpha_2 \xi_2) q^{a (\xi_2 - bn_1/a)^2 /2D}
		\frac{d\xi_2}{1 - \bm{e} (\xi_2)}
		\\
		&=
		\sum_{n_1 \in \Z} \sgn(n_1) \bm{e} \left( \left( \alpha_1 + \frac{b}{a} \alpha_2 \right) n_1 \right) q^{n_1^2 /2a}
		\cdot 2 \int_{C_+ - b n_1/a} 
		\bm{e} (\alpha_2 \xi_2) q^{a \xi_2^2 /2D}
		\frac{d\xi_2}{1 - \bm{e} (\xi_2 + b n_2 /a)}.
	\end{align}
	By deforming integration path, for any $ \mu \in \R $, we have
	\begin{align}
		2 \int_{C_+ + \mu} d\xi_2 - 2 \int_{C_+} d\xi_2
		&=
		\begin{dcases}
			2 \sum_{\beta \in \left[ \mu, 0 \right)} 2\pi\iu \Res_{\xi_2 = \beta} & \text{ if } \mu < 0, \\
			2 \sum_{\beta \in \left[ 0, \mu \right)} 2\pi\iu \Res_{\xi_2 = \beta} & \text{ if } \mu \ge 0
		\end{dcases}
		\\
		&=
		\sum_{\beta \in \R} ( 1 - \sgn(\beta) \sgn(\beta - \mu) ) 2\pi\iu \Res_{\xi_2 = \beta}.
	\end{align}
	Moreover, if $ 1 - \sgn(\beta) \sgn(\beta - \mu) \neq 0 $, then $ \sgn(\beta) = \sgn(\mu) $.
	Thus, we have
	\begin{equation} \label{eq:contour_deforming_false_theta}
		2 \int_{C_+ + \mu} d\xi_2 - 2 \int_{C_+} d\xi_2
		=
		\sgn(\mu) \sum_{\beta \in \R} ( \sgn(\beta) - \sgn(\beta - \mu) ) 2\pi\iu \Res_{\xi_2 = \beta}.
	\end{equation}
	Therefore, we obtain
	\begin{align}
		&\phant
		\calI_1(\tau) - \sum_{\mu \in a^{-1} \Z / \Z} \widetilde{\theta}_{1, \mu} (\tau) \Omega_{1, \mu} (\tau)
		\\
		&=
		\sum_{n_1 \in \Z} \sgn(n_1) \bm{e} \left( \left( \alpha_1 + \frac{b}{a} \alpha_2 \right) n_1 \right) q^{n_1^2 /2a}
		\left( 2 \int_{C_+ - b n_1/a} - 2 \int_{C_+}  \right)
		\bm{e} (\alpha_2 \xi_2) q^{a \xi_2^2 /2D}
		\frac{d\xi_2}{1 - \bm{e} (\xi_2 + b n_1 /a)}		
		\\
		&=
		\sum_{n_1 \in \Z} \sgn(n_1) \bm{e} \left( \left( \alpha_1 + \frac{b}{a} \alpha_2 \right) n_1 \right) q^{n_1^2 /2a}
		\sgn \left( -\frac{b n_1}{a} \right)
		\sum_{n_2 \in -b n_1 /a + \Z} \left( \sgn(n_2) - \sgn \left(n_2 + \frac{b n_1}{a} \right) \right)
		\bm{e} (\alpha_2 n_2) q^{a n_2^2 /2D}.
	\end{align}
	Since $ \sgn(x) \sgn(\pm x) = 1 $ for any $ x \in \R $, we have
	\begin{align}
		&\phant
		\calI_1(\tau) - \sum_{\mu \in a^{-1} \Z / \Z} \widetilde{\theta}_{1, \mu} (\tau) \Omega_{1, \mu} (\tau)
		\\
		&=
		\sum_{n_1 \in \Z} \bm{e} \left( \left( \alpha_1 + \frac{b}{a} \alpha_2 \right) n_1 \right) q^{n_1^2 /2a}
		\sum_{n_2 \in -b n_1 /a + \Z} \left( \sgn(n_2) - \sgn \left(n_2 + \frac{b n_1}{a} \right) \right)
		\bm{e} (\alpha_2 n_2) q^{a n_2^2 /2D}.
	\end{align}
	By replacing $ n_1 $ by $ a m_1 $ with $ \mu \in a^{-1} \Z / \Z $ and $ m_1 \in \mu + \Z $, we have
	\begin{align}
		&\phant
		\calI_1(\tau) - \sum_{\mu \in a^{-1} \Z / \Z} \widetilde{\theta}_{1, \mu} (\tau) \Omega_{1, \mu} (\tau)
		\\
		&=
		\sum_{\mu \in a^{-1} \Z / \Z} \sum_{m_1 \in \mu + \Z} \sum_{n_2 \in -b \mu + \Z} 
		\left( \sgn(n_2) - \sgn \left(n_2 + b m_1 \right) \right)
		\bm{e} \left( \left( a \alpha_1 + b \alpha_2 \right) m_1 + \alpha_2 n_2 \right) q^{na m_1^2 /2 + a n_2^2 /2D}
		\\
		&=
		\sum_{\mu \in a^{-1} \Z / \Z} \widetilde{\Theta}_{1, \mu} (\tau).
	\end{align}
	By the same calculation, we have
	\[
	\calI_2 (\tau) - \sum_{\mu \in c^{-1} \Z / \Z} \widetilde{\theta}_{2, \mu} (\tau) \Omega_{2, \mu} (\tau)
	=
	\sum_{\mu \in c^{-1} \Z / \Z} \widetilde{\Theta}_{2, \mu} (\tau).
	\]
	Thus, we obtain the claim.
\end{proof}

\begin{rem}
	We can also write
	\begin{align}
		&\phant
		\calI_1(\tau) - \sum_{\mu \in a^{-1} \Z / \Z} \widetilde{\theta}_{1, \mu} (\tau) \Omega_{1, \mu} (\tau)
		\\
		&=
		\sum_{\mu \in a^{-1} \Z / \Z}
		\left(
		\sum_{\substack{
				n \in \transpose{(\mu, -b \mu)} + \Z^2, \\
				-bn_1 \le n_2 < 0
		}}
		+ \sum_{\substack{
				n \in \transpose{(\mu, -b \mu)} + \Z^2, \\
				0 \le n_2 < -bn_1
		}}
		\right)
		\bm{e} ((a \alpha_1 + b \alpha_2) n_1 + \alpha_2 n_2) q^{a (n_1^2 + n_2^2/D)/2}.
	\end{align}
	One may call such functions \emph{multiple theta functions}, in analogy with multiple zeta functions or multiple Eisenstein series.
\end{rem}


\subsection{General rank case} \label{subsec:false_theta_rk_general}


We are now in a position to state the modular transformation formula for false theta functions in the general rank case.

To begin with, we recall our main result stated in \cref{subsec:intro_number_theory}.

\begin{dfn} \label{dfn:false_theta}
Let $ r \ge d \ge 1 $ be positive integers, $ S \in \Sym_r^+ (\Q) $ be a positive definite symmetric matrix,
$ A \in \Mat_{d, r} (\Z) $ be a matrix of rank $ d $, 
$ P(x_1, \dots, x_r) \in \Q[x_1, \dots, x_r] $ be a polynomial, and $ \alpha \in \Q^r $ be a vector.
For these data, we define
\[
\widetilde{\Theta} (\tau)
\coloneqq
\sum_{m \in \alpha + \Z^r} \sgn(Am) P(m) q^{\transpose{m} Sm /2}
\]
to be the \emph{false theta function of rank $ r $ and depth $ d $ associated to the symmetric matrix $ S $}.
\end{dfn}

The following result follows from \cref{prop:sign_Am_to_m}, which will be proved in \cref{sec:sgn}.

\begin{prop} \label{prop:false_theta_rep}
Let $ r $ and $ s $ be positive integers, $ S \in \Sym_r^+ (\Q) $ be a positive definite symmetric matrix,
$ A \in \Mat_{s, r} (\Z) $ be a matrix of rank $ d \ge 0 $, 
$ \overline{g}(x) \in \overline{\frakQ}^0 $ be a quasi-polynomial, and $ \alpha \in \Q^r $ be a vector.
For these data, we define
\[
\widetilde{\Theta} (\tau)
\coloneqq
\sum_{m \in \alpha + \Z^r} \sgn(Am) \overline{g}(m) q^{\transpose{m} Sm /2}.
\]
Then, the function $ \widetilde{\Theta} (\tau) $ can be written as a finite $ \Q $-linear combination of false theta functions of rank $ r $ and depth $ d' $ which have forms
\[
\widetilde{\Theta} (\tau)
\coloneqq
\sum_{m \in \beta + \Z^r} \sgn(m_1) \cdots \sgn(m_{d'}) h(m) q^{\transpose{m} \transpose{B} SBm /2},
\]
where $ d' \le d $, $ \beta \in \Q^r $, $ h(x) \in \overline{\frakQ} $, and $ B \in \Mat_r(\Z) \cap \GL_r(\Q) $.
\end{prop}

Our main theorem is as follows.

\begin{thm} \label{thm:false_theta}
Let $ \widetilde{\Theta}(\tau) $ be a false theta function of rank $ r $ and depth $ d $ 
associated to a positive definite symmetric matrix $ S \in \Sym_r^+ (\Q) $.
Then, there exist finitely many false theta functions 
$ \widetilde{\Theta}_1 (\tau), \dots, \widetilde{\Theta}_M (\tau), \widetilde{\Theta}'_1 (\tau), \dots, \widetilde{\Theta}'_{M'} (\tau) $ of rank $ \le r $ and depth $ \le d $, 
non-negative integers $ l_1, \dots, l_M \in \Z_{\ge 0} $, and 
holomorphic functions $ \Omega(\tau), \Omega_1 (\tau), \dots, \Omega_N (\tau) \in \calO(\widetilde{\bbC}) $ such that
\[
\sqrt{\det S} {\sqrt{\frac{\iu}{\tau}}}^{\, r} \cdot \widetilde{\Theta} \left( -\frac{1}{\tau} \right)
=
\sum_{1 \le i \le M} \widetilde{\Theta}_i (\tau) \tau^{l_i} 
+ \sum_{1 \le i \le M'} \widetilde{\Theta}'_i (\tau) \Omega_j (\tau) 
+ \Omega(\tau).
\]

Furthermore, the above false theta functions admit a rather explicit description.
Suppose that 
\[
\widetilde{\Theta} (\tau)
=
\sum_{m \in \Z^r} \widetilde{g}(m) q^{\transpose{(m + \alpha)} S (m + \alpha) /2},
\]
where $ S \in \Sym_r^+ (\Q) $ is a positive definite symmetric matrix,
$ \alpha \in \Q^r $ is a vector, and
$ \widetilde{g}(x) = \widetilde{\coe}[G](x) \in \overline{\frakQ} $ is a false quasi-polynomial with
$ G(z) = G_1 (z_1) \cdots G_r (z_r) \in \frakR $ such that $ \overline{g}(x) \coloneqq \overline{\coe}[G](x) \in \overline{\frakQ}^0 $.
In this case, the above modular transformation formula takes the form
\[
\sqrt{\det S} {\sqrt{\frac{\iu}{\tau}}}^{\, r} \cdot \widetilde{\Theta} \left( -\frac{1}{\tau} \right)
=
\sum_{\pi \in \Pi_r} \sum_{1 \le i \le N_\pi} \widetilde{\Theta}_{\pi, i} (\tau) \Omega_{\pi, i} (\tau)
+ \Omega(\tau),
\]
with the following notation:
\begin{itemize}
	\item We denote $ \Pi_r \coloneqq \left\{ (I_1, \dots, I_s) \mid I_1 \sqcup \cdots \sqcup I_s = \{ 1, \dots, r \} \right\} $ 
	the set of ordered partitions of $ \{ 1, \dots, r \} $.
	
	\item For each $ \pi \in \Pi_r $, $ N_\pi $ is a positive integer.
	
	\item For each $ \pi \in \Pi_r $ and $ 1 \le i \le N_\pi $, 
	$ \widetilde{\Theta}_{\pi, i} (\tau) $ is a false theta function of rank $ \le r $ and depth $ \le d $
	and $ \Omega_{\pi, i} (\tau) $ is a holomorphic function on $ \widetilde{\bbC} $.
	
	
	\item For $ \pi = (\{ 1, \dots, r \}) $, we can write 
	$ \Omega_{\pi, i} (\tau) = \tau^{l_i} $ for some non-negative integers $ l_1, \dots, l_M $ and 
	\[
	\sum_{1 \le i \le N_\pi} \widetilde{\Theta}_{\pi, i} (\tau) \Omega_{\pi, i} (\tau)
	=
	\sum_{\beta \in \calP_\pi /\Z} \sum_{n \in \beta + \Z^r}  \sgn(n) \widehat{g} (\tau; n) \bm{e} ( \transpose{\alpha} n) 
	q^{\transpose{n} S_\pi n /2},
	\]
	where $ S^*_\pi \coloneqq S^{-1} $,
	$ \calP_\pi = \calP \subset \Q^r $ is the set of poles of meromorphic function $ G( \bm{e} (\xi) ) $,
	$ T_{\overline{g}} \colon \widehat{\frakK} \to \widehat{\frakK} $ is the operator defined in \cref{subsec:modular_series_modular}, and
	$ \widehat{g} (\tau; \xi_1, \dots, \xi_r) \in \Q[\tau, \xi_1, \dots, \xi_r] $ is a polynomial defined as
	\[
	\widehat{g} (\tau; \xi) \bm{e} ( \transpose{\alpha} \xi) q^{\transpose{\xi} S^{-1} \xi /2}
	\coloneqq
	T_{\overline{g}} \left[ \bm{e} ( \transpose{\alpha} \xi) q^{\transpose{\xi} S^{-1} \xi /2} \right].
	\]
	
	\item We can write
	\[
	\Omega (\tau)
	=
	2^r \int_{C_+^r} \bm{e} (\transpose{\alpha} \xi) \widehat{g} (\tau; \xi) q^{\transpose{\xi} S^{-1} \xi / 2}
	G (\bm{e} (\xi)) d\xi.
	\]
	
	\item For $ \pi = (I_1, \dots, I_s) \in \Pi_r $ with $ s \ge 2 $ and $ 1 \le i \le N_\pi $, we can write 
	\begin{align}
		\widetilde{\Theta}_{\pi, i} (\tau)
		&=
		\sum_{n \in \beta_{\pi, i} + \Z^{I_s^\complement}} \sgn( A_{\pi, i} n) \bm{e} ( \transpose{\alpha} U_\pi n) 
		\widehat{g}_{\pi, i} (n) q^{\transpose{n} S^*_\pi n /2},
		\\
		\Omega_{\pi, i} (\tau)
		&=
		\int_{C_+^{I_s}}
		\bm{e} (\transpose{\alpha_{I_s}^{}} \xi_{I_s}^{}) 
		h_{\pi, i} \left( \tau; \xi_{I_s}^{} \right) 
		q^{\transpose{\xi_{I_s}^{}} S_{I_s}^* \xi_{I_s}^{}/2}
		G_{I_s}^{} \left( \bm{e} \left( \xi_{I_s}^{} + \beta^*_{\pi, i} \right) \right) 
		d\xi_{I_s}^{},
	\end{align}
	with the following notation:
	\begin{itemize}
		\item For a subset $ I \subset \{ 1, \dots, r \} $, denote $ I^\complement \coloneqq \{ 1, \dots, r \} \smallsetminus I $ its complement.
		
		\item Some matrix $ A_{\pi, i} \in \Mat_{d, \abs{I_s^\complement}} (\Z) $ of rank $ d $.
		
		\item Some quasi-polynomial $ \widehat{g}_{\pi, i} (x) \in \Q[x_j \mid j \in I_s^\complement] $.
		
		\item Some polynomial $ h_{\pi, i} \left( \tau; \xi_{I_s}^{} \right) \in \Q[\tau, \xi_j \mid j \in I_s ] $.
		
		\item For a subset $ I \subsetneq \{ 1, \dots, r \} $, let 
		\[
		S \eqqcolon \pmat{S_I^{} & B_I^{} \\ \transpose{B_I^{}} & S_{I^\complement}^{}}, \quad
		S^{-1} \eqqcolon \pmat{S_I^* & * \\ * & S_{I^\complement}^*}.
		\]
		be block submatrices.
		
		\item More generally, for a subset $ J \subset I \subset \{ 1, \dots, r \} $ and a $ I \times I $ matrix $ A \in \GL_I^{} (\Q) $, 
		denote $ A_J^{} $ and $ A_J^* $ the $ J \times J $-submatrices of $ A $ and $ A^{-1} $ respectively.
		
		\item For a subset $ I, J \subset \{ 1, \dots, r \} $, 
		denote $ B_{I, J}^{} $ and $ B_{I, J}^* $ the $ I \times J $-submatrices of $ S $ and $ S^{-1} $ respectively.
		
		\item A symmetric matrix
		\[
		S^*_\pi \coloneqq
		\pmat{ (S_{I_1})^{-1} & & & \\ & S^*_{I_1, I_2} & & \\ & & \ddots & & \\ & & & 
			S^*_{I_1 \sqcup \cdots \sqcup I_{s-2}, I_{s-1}}}
		\in \Sym_{I_s^\complement}^{} (\Q),
		\]
		where $ S_{I, J}^* \coloneqq ((S_{I^\complement}^*)_J^*)^{-1} $ for $ J \subset I \subset \{  1, \dots, r \} $.
		
		\item A lower triangular matrix
		\begin{equation}
			U_\pi =
			\begin{tikzpicture}[baseline=(M.center)]
				\matrix (M) [
				matrix of math nodes,
				nodes in empty cells,
				left delimiter={(}, right delimiter={)},
				row sep=1.2em, column sep=1.6em,
				nodes={text height=1.8ex, text depth=0.6ex} 
				]{
					1_{I_1} &        		   &        &              & \\
					& 1_{I_2}   	&        &              & \\
					U_{I_1} &           		& \ddots &              & \\
					& U_{I_1 \sqcup I_2,I_2} & \ddots & 1_{I_{s-1}} & \\
					&        			 &        & U_{I_1 \sqcup \cdots \sqcup I_{s-1}, I_{s-1}} & 1_{I_s} \\
				};
				
				\node[draw,rounded corners,inner xsep=10pt,fit=(M-2-1)(M-5-1)] {};
				\node[draw,rounded corners,inner xsep=20pt,fit=(M-3-2)(M-5-2)] {};
				\node[draw,rounded corners,inner xsep=5pt,fit=(M-5-4)(M-5-4)] {};
			\end{tikzpicture}
			\in \GL_r(\Q),
		\end{equation}
		where $ 1_{I}^{} $ is the $ I \times I $ identity matrix,
		$ U_{I_1}^{} \coloneqq -\transpose{B_I^{}} \left( S_I^{} \right)^{-1} $,
		and $ U_{I, J} \coloneqq ((S_{I^\complement}^*)_J^*)^{-1} \transpose{B_{J, I^\complement}^*} $ for $ J \subset I \subset \{  1, \dots, r \} $.
		
		\item For each $ 1 \le i \le r $, denote $ \calP_i \subset \Q $ is the set of poles of the meromorphic function $ G_i (\bm{e} (\xi_i) ) $.
		
		\item For each $ I \subset \{ 1, \dots, r \} $, denote $ \calP_I^{} \coloneqq \prod_{i \in I} \calP_i \subset \Q^I $.
		
		\item Let $ \calP_\pi \coloneqq U_\pi (\calP_{I_s^\complement} \oplus \Q^{I_s}) \cap \Q^{I_s^\complement} $.
		
		\item Some vector $ \beta_{\pi, i} \in \calP_\pi / \Z^{I_s^\complement} $.
		
		\item Some vector $ \beta^*_{\pi, i} \in U_\pi^{-1} (\calP_\pi) \cap \Q^{I_s} / \Z^{I_s} $.
		
	\end{itemize}
\end{itemize}
\end{thm}

\begin{rem}
For $ \pi = (I_1, \dots, I_s) $, the definition of $ S^*_\pi $ and $ U_\pi $ differs between the cases $ s = 1 $ and $ s \ge 2 $.
As we will see in the proof below, 
for $ s = 1 $ we ``invert'' the quadratic form via the Fourier transform.
In contrast, for $ s \ge 2 $ we do not apply an additional Fourier transform but instead expand it directly.
\end{rem}

\begin{proof}[Proof of \cref{thm:false_theta}]
By \cref{lem:sign_Am_to_m}, the first statement for general false theta functions follows from the last statement.
By \cref{lem:R_C_corresp} \cref{item:lem:R_C_corresp:C_tilde_rep},  we can assume that $ \widetilde{\Theta} (\tau) $ has a form
\[
\widetilde{\Theta} (\tau)
=
\sum_{m \in \alpha + \Z^r} \sgn(m) \overline{g}(m) q^{\transpose{m} Sm /2}
\]
with $ \overline{g}(x) \in \overline{\frakQ}^0 $ and $ \overline{\coe}^{-1} [\overline{g}] (z) = G(z) = G_1 (z_1) \cdots G_r (z_r) $.
By the modular transformation formula for false modular series (\cref{thm:false_modular_series_modular}), we have
\[
\sqrt{\frac{\iu}{\tau}}^{\, r} \frac{1}{\sqrt{\det S}} \cdot
\widetilde{\Theta} \left(  -\frac{1}{\tau}\right)
=
\sum_{I \subset \{ 1, \dots, r \}} \calI_I^{} (\tau).
\]
Here, $ \calI_I^{} (\tau) $ for $ I \subset \{ 1, \dots, r \} $ is defined as follows:
\begin{itemize}
	\item For a vector $ x = (x_1, \dots, x_r) \in \bbC^r $, denote $ x_I^{} \coloneqq (x_i)_{i \in I}^{} $.
	\item Let $ \overline{g}_I^{} (x_I^{}) \coloneqq \prod_{i \in I} \overline{g}_i (x_i) \in \overline{\frakQ}^0_I $.
	\item Let $ G_I^{} (z_I^{}) \coloneqq \overline{\coe}^{-1} [\overline{g}_I^{}] (z_I^{}) \in \frakR_I^{} $.
	\item Let $ T_{\overline{g}_I^{}} \colon \widehat{\frakK}_I^{} \to \widehat{\frakK}_I^{} $ be the operator defined in \cref{subsec:modular_series_modular}. 
	\item Let $ \widehat{g}_I^{} (\tau; \xi_1, \dots, \xi_r) \in \Q[\tau, \xi_1, \dots, \xi_r] $ is a polynomial defined as
	\[
	\widehat{g}_I^{} (\tau; \xi) \bm{e} ( \transpose{\alpha} \xi) q^{\transpose{\xi} S^{-1} \xi /2}
	\coloneqq
	T_{\overline{g}_I^{}} \left[ \bm{e} ( \transpose{\alpha} \xi) q^{\transpose{\xi} S^{-1} \xi /2} \right].
	\]
	\item Define a function $ \calI_I^{} (\tau) $ as
	\[
	\calI_I^{} (\tau) \coloneqq
	\sum_{n_I^{} \in \calP_I^{}} \sgn(n_I^{}) \bm{e} (\transpose{\alpha_I^{}} n_I^{})
	\cdot 2^{r - \abs{I}} \int_{C_+^{I^\complement}}
	\bm{e} (\transpose{\alpha_{I^\complement}^{}} \xi_{I^\complement}^{}) 
	\widehat{g}_I^{} (\tau; n_I^{}, \xi_{I^\complement}^{}) 
	q^{\frac{1}{2} (\transpose{n_I^{}}, \transpose{\xi_{I^\complement}^{}}) S^{-1} \smat{n_I^{} \\ \xi_{I^\complement}^{}}}
	G_{I^\complement}^{} (\bm{e}(\xi_{I^\complement}^{})) d\xi_{I^\complement}^{}.
	\]
\end{itemize}
We have
\[
\calI_{\{ 1, \dots, r \}}^{} (\tau)
=
\sum_{n \in \calP} \sgn(n) \bm{e} ( \transpose{\alpha} n) \widehat{P} (\tau; n) q^{\transpose{n} S^{-1} n /2},
\quad
\calI_\emptyset^{} (\tau) = \Omega (\tau).
\]
We need to calculate $ \calI_I^{} (\tau) $ for $ \emptyset \neq I \subsetneq \{ 1, \dots, r \} $.
To achieve this, we need the following lemma.

\begin{lem}[{\cite[Subsection A.5.5]{Boyd-Vandenberghe}, \cite[Lemma 2.3]{M:plumbed}}] \label{lem:block_matrix}
	For a symmetric block matrix
	\[
	X = \pmat{A & B \\ \transpose{B} & C} \in \GL_{r+s} (\bbC)
	\]
	such that $ A \in \GL_r(\bbC) $ and $ C \in \GL_s(\bbC) $ be symmetric matrices and $ B \in \Mat_{r, s}(\bbC) $, 
	let $ C^* $ be the lower-right $ s \times s $ submatrix of $ X^{-1} $\textup{:}
	\[
	X^{-1} = \pmat{* & * \\ * & C^*}.
	\]
	Then, we have
	\[
	X^{-1} = \pmat{-A^{-1} B \\ I_s} C^* \pmat{-\transpose{B}A^{-1} & I_s} + \pmat{A^{-1} & \\ & O}, \quad
	C^* =  (C - \transpose{B} A^{-1} B)^{-1}, \quad
	\det C^* = \frac{\det A}{\det X}.
	\]
\end{lem}

By this lemma, we have
\[
\transpose{\xi} S^{-1} \xi
=
\transpose{( \xi_{I^\complement}^{} + U_I^{} \xi_I^{} )}
S_{I^\complement}^*
( \xi_{I^\complement}^{} + U_I^{} \xi_I^{} )
+ \transpose{\xi_I^{}} (S_I^{})^{-1} \xi_I^{},
\]
where we denote $ U_I^{} \coloneqq -\transpose{B_I^{}} (S_I^{})^{-1} $ as in the statement.
Thus, by replacing $ \xi_{I^\complement}^{} $ by $ \xi_{I^\complement}^{} - U_I^{} n_I^{} $, we have
\begin{align}
	\calI_I(\tau) 
	&=
	\sum_{n_I^{} \in \calP_I^{}} \sgn(n_I^{}) \bm{e} (\transpose{\alpha_I^{}} U_\pi n_I^{}) q^{\transpose{n_I^{}} (S_I^{})^{-1} n_I^{} /2}
	\\
	&\phant
	\cdot 2^{r - \abs{I}} \int_{C_+^{I^\complement} - U_I^{} n_I^{} }
	\bm{e} (\transpose{\alpha_{I^\complement}^{}} \xi_{I^\complement}^{}) 
	\widehat{g}_I^{} \left( 
	\tau; n_I^{}, \xi_{I^\complement}^{} - U_I^{} n_I^{} 
	\right) 
	q^{\transpose{\xi_{I^\complement}^{}} S_{I^\complement}^* \xi_{I^\complement}^{} /2}
	G_{I^\complement}^{} \left( \bm{e} \left( \xi_{I^\complement}^{} - U_I^{} n_I^{} \right) \right) d\xi_{I^\complement}^{}.
\end{align}
By using \cref{eq:contour_deforming_false_theta}, the integral term can be written as
\begin{align}
	&\phant
	\left( \prod_{j \in I^\complement} 
	\left(
	2 \int_{C_+} G_j \left( \bm{e} \left( \xi_j - (U_I^{} n_I^{})_j \right) \right) d\xi_j
	+ \sgn((U_I^{} n_I^{})_j) \sum_{n_j \in (U_I^{} n_I^{})_j + \calP_j}
	\left( \sgn(n_j) - \sgn(n_j - (U_I^{} n_I^{})_j) \right) \ev_{\xi_j = n_j}
	\right)
	\right)
	\\
	&\phant
	\left[
	\bm{e} (\transpose{\alpha_{I^\complement}^{}} \xi_{I^\complement}^{}) 
	\widehat{g} \left( 
	\tau; n_I^{}, \xi_{I^\complement}^{} - U_I^{} n_I^{}
	\right) 
	q^{\transpose{\xi_{I^\complement}^{}} S_{I^\complement}^* \xi_{I^\complement}^{} /2}
	\right]
	\\
	&=
	\sum_{J \subset I^\complement} \sum_{n_J^{} \in -(U_I^{} n_I^{})_J^{} + \calP_J^{}}
	\sigma_{I, J}^{} (n_I^{}, n_J^{}) \bm{e} (\transpose{\alpha_J^{}} n_J^{}) 
	\\
	&\phant
	\cdot 2^{\abs{(I \sqcup J)^\complement}}
	\int_{C_+^{(I \sqcup J)^\complement}}
	\bm{e} (\transpose{\alpha_{(I \sqcup J)^\complement}^{}} \xi_{(I \sqcup J)^\complement}^{}) 
	\widehat{g}_I^{} \left( 
	\tau; n_I^{}, \pmat {n_J^{} \\ \xi_{(I \sqcup J)^\complement}^{} } - U_I^{} n_I^{}
	\right) 
	q^{\smat{ \transpose{n_J^{}} & \transpose{\xi_{(I \sqcup J)^\complement}^{}} }
		S_{I^\complement}^* \smat{n_J^{} \\ \xi_{(I \sqcup J)^\complement}^{} } /2}
	\\
	&\phant
	\cdot G_{(I \sqcup J)^\complement}^{} 
	\left( \bm{e} \left( \xi_{(I \sqcup J)^\complement}^{} - (U_I^{} n_I^{})_{(I \sqcup J)^\complement}^{}  \right) \right) 
	d\xi_{(I \sqcup J)^\complement}^{},
\end{align}
where
\[
\sigma_{I, J}^{} (n_I^{}, n_J^{})
\coloneqq
\sgn(U_I^{} n_I^{}) 
\left(
\prod_{j \in J} \left( \sgn(n_j) + \sgn(n_j - (U_I^{} n_I^{})_j) \right) 
\right).
\]
Here, we have
\begin{align}
	&\phant
	\pmat{ \transpose{n_J^{}} & \transpose{\xi_{(I \sqcup J)^\complement}^{}} }
	S_{I^\complement}^* \pmat{n_J^{} \\ \xi_{(I \sqcup J)^\complement}^{} }
	\\
	&=
	\transpose{n_J^{}} 
	\left( 
	S_J^* - B^*_{J, (I \sqcup J)^\complement} \left( S^*_{(I \sqcup J)^\complement} \right)^{-1} \transpose{B^*_{J, (I \sqcup J)^\complement}}
	\right) 
	n_J^{}
	\\
	&\phant
	+
	\transpose{
		\left( 
		\xi_{(I \sqcup J)^\complement}^{} 
		+ \left( S^*_{(I \sqcup J)^\complement} \right)^{-1} \transpose{B^*_{J, (I \sqcup J)^\complement}} n_J^{} 
		\right)
	}
	S_{(I \sqcup J)^\complement}^* 
	\left( 
	\xi_{(I \sqcup J)^\complement}^{} 
	+ \left( S^*_{(I \sqcup J)^\complement} \right)^{-1} \transpose{B^*_{J, (I \sqcup J)^\complement}} n_J^{} 
	\right)
	\\
	&=
	\transpose{n_J^{}} S^*_{I, J} n_J^{}
	+
	\transpose{
		\left( 
		\xi_{(I \sqcup J)^\complement}^{} + U_{I, J}^{} n_J^{} 
		\right)
	}
	S_{(I \sqcup J)^\complement}^* 	
	\left( 
	\xi_{(I \sqcup J)^\complement}^{} + U_{I, J}^{} n_J^{} 
	\right)
\end{align}
by \cref{lem:block_matrix}.
Therefore, we obtain
\begin{align}
	\calI_I(\tau) 
	&=
	\sum_{n_I^{} \in \calP_I^{}} 
	\sum_{J \subset I^\complement} 
	\sum_{n_J^{} \in \left( U_I^{} n_I^{} \right)_J^{} + \calP_J^{}}
	\sigma_{I, J}^{} (n_I^{}, n_J^{})
	\bm{e} (\transpose{\alpha_I^{}} U_\pi n_I^{} + \transpose{\alpha_J^{}} n_J^{}) 
	q^{\transpose{n_I^{}} (S_I^{})^{-1} n_I^{} /2 + \transpose{n_J^{}} S^*_{I, J} n_J^{} /2}
	\\
	&\phant
	\cdot 2^{\abs{(I \sqcup J)^\complement}}
	\int_{C_+^{(I \sqcup J)^\complement}}
	\bm{e} (\transpose{\alpha_{(I \sqcup J)^\complement}^{}} \xi_{(I \sqcup J)^\complement}^{}) 
	\widehat{g}_I^{} \left( 
	\tau; n_I^{}, \pmat {n_J^{} \\ \xi_{(I \sqcup J)^\complement}^{} } - U_I^{} n_I^{}
	\right) 
	\\
	&\phant
	\cdot 
	q^{\frac{1}{2} \transpose{
			\left( 
			\xi_{(I \sqcup J)^\complement}^{} + U_{I, J}^{} n_J^{} 
			\right)
		}
		S_{(I \sqcup J)^\complement}^* 	
		\left( 
		\xi_{(I \sqcup J)^\complement}^{} + U_{I, J}^{} n_J^{} 
		\right)}
	G_{(I \sqcup J)^\complement}^{} 
	\left( \bm{e} \left( \xi_{(I \sqcup J)^\complement}^{} - (U_I^{} n_I^{})_{(I \sqcup J)^\complement}^{}  \right) \right) 
	d\xi_{(I \sqcup J)^\complement}^{}.
\end{align}
By repeating this operation until the condition $ J = \emptyset $ is satisfied, we obtain the expression
\[
\calI_I^{} (\tau)
=
\sum_{\pi = (I_1, \dots, I_s) \in \Pi_r, \, I_1 = I} \calI_\pi (\tau),
\]
where
\begin{align}
	\calI_\pi (\tau)
	&\coloneqq
	\sum_{n_{I_s^\complement}^{} \in L_\pi}
	\sigma_\pi \left( n_{I_s^\complement}^{} \right)
	\bm{e} \left( \transpose{\alpha} U_\pi^{-1} n_{I_s^\complement}^{} \right)
	q^{\transpose{n_{I_s^\complement}^{}} S^*_\pi n_{I_s^\complement}^{} /2}
	\\
	&\phant
	\cdot 2^{\abs{I_s}}
	\int_{C_+^{I_s}}
	\bm{e} (\transpose{\alpha_{I_s}^{}} \xi_{I_s}^{}) 
	\widehat{g}_I^{} \left( 
	\tau; U_\pi^{-1} \pmat{n_{I_s^\complement}^{} \\ \xi_{I_s^{} }}
	\right) 
	q^{\transpose{\xi_{I_s}^{}} S_{I_s}^* \xi_{I_s}^{}/2}
	G_{I_s}^{} \left( \bm{e} \left( \xi_{I_s}^{} - \left( U_\pi n_{I_s^\complement}^{} \right)_{I_s}^{} \right) \right) 
	d\xi_{I_s}^{},
	\\
	\sigma_\pi \left( n_{I_s^\complement}^{} \right)
	&\coloneqq
	\sgn(n_{I_1}^{}) \prod_{2 \le s' \le s} \sigma_{s'} (n_{I_{s' - 1}^\complement}^{}, n_{I_{s'}^\complement}^{}),
	\\
	\sigma_{s'} \left( n_{I_{s' - 1}^\complement}^{}, n_{I_{s'}^\complement}^{} \right)
	&\coloneqq
	\sgn \left( \left( U_{I_{s'-1}}^{} n_{s'-1}^{} \right)_{I_{s'}}^{} \right)
	\prod_{j \in I_{s'}} \left( \sgn(n_j) - \sgn \left( n_j - \left( U_{I_{s'-1}}^{} n_{s'-1}^{} \right)_j \right) \right).
\end{align}
By \cref{lem:analy_cont,lem:sign_Am_to_m_rank_less,prop:signs_more_than_rank}, we can write
\[
\calI_\pi (\tau)
=
\sum_{1 \le i \le N_\pi} \widetilde{\Theta}_{\pi, i} (\tau) \Omega_{\pi, i} (\tau)
\]
for some false theta function $ \widetilde{\Theta}_{\pi, i} (\tau) $ of rank $ \le r $ and depth $ \le d $ 
and some holomorphic function $ \Omega_{\pi, i} (\tau) $ on $ \widetilde{\bbC} $ which admit representations as in the statement.
\end{proof}

By the same argument, we obtain the following similar statement over $ \R $.

\begin{thm} \label{thm:false_theta_R}
Let $ \widetilde{\Theta}(\tau) $ be a false theta function defined as
\[
\widetilde{\Theta} (\tau)
\coloneqq
\sum_{m \in \alpha + \Z^r} \sgn(Am) \bm{e} (\transpose{\beta} m) P(m) q^{\transpose{m} Sm /2}
\]
for positive integers $ r \ge r' \ge 1 $,
a positive definite symmetric matrix $ S \in \Sym_r^+ (\R) $,
a matrix $ A \in \Mat_{r', r} (\Z) $ of rank $ d \ge 0 $, 
a polynomial $ P(x_1, \dots, x_r) \in \R[x_1, \dots, x_r] $, and 
two vectors $ \alpha, \beta \in \R^r $.
Then, there exist finitely many false theta functions 
$ \widetilde{\Theta}_1 (\tau), \dots, \widetilde{\Theta}_M (\tau), \widetilde{\Theta}'_1 (\tau), \dots, \widetilde{\Theta}'_{M'} (\tau) $ of rank $ \le r $ and depth $ \le d $ of the above type, 
non-negative integers $ l_1, \dots, l_M \in \Z_{\ge 0} $ and 
holomorphic functions $ \Omega(\tau), \Omega_1 (\tau), \dots, \Omega_N (\tau) \in \calO(\widetilde{\bbC}) $ such that
\[
\sqrt{\det S} {\sqrt{\frac{\iu}{\tau}}}^{\, r} \cdot \widetilde{\Theta} \left( -\frac{1}{\tau} \right)
=
\sum_{1 \le i \le M} \widetilde{\Theta}_i (\tau) \tau^{l_i} 
+ \sum_{1 \le i \le M'} \widetilde{\Theta}'_i (\tau) \Omega_j (\tau) 
+ \Omega(\tau).
\]
Moreover, if $ A = I_r \in \Mat_{r} (\Z) $ which appears in the definition of $ \widetilde{\Theta}(\tau) $, we can write
\begin{align}
	\sum_{1 \le i \le M} \widetilde{\Theta}_i^* (\tau) \tau^{l_i} 
	&=
	\sum_{n \in \Z^r} \sgn(n) \bm{e} ( \transpose{\alpha} n) \widehat{P} (\tau; n + \beta) q^{\transpose{(n + \beta)} S^{-1} (n + \beta) /2}
	\\
	\Omega (\tau)
	&=
	2^r \int_{C_+^r} \bm{e} (\transpose{\alpha} \xi) \widehat{P} (\tau; \xi + \beta) q^{\transpose{(\xi + \beta)} S^{-1} (\xi + \beta) / 2}
	\prod_{1 \le i \le r} \frac{d\xi_i}{1 - \bm{e} (\xi_i)},
\end{align}
where $ \widehat{P} (\tau; \xi_1, \dots, \xi_r) \in \R[\tau, \xi_1, \dots, \xi_r] $ is a polynomial defined as
\[
\widehat{P} (\tau; \xi) q^{\transpose{\xi} S^{-1} \xi /2}
\coloneqq
P \left( -\frac{1}{2\pi\iu} \frac{\partial}{\partial \xi_1}, \dots, -\frac{1}{2\pi\iu} \frac{\partial}{\partial \xi_1} \right) 
\left[ q^{\transpose{\xi} S^{-1} \xi /2} \right].
\]
\end{thm}


\subsection{Asymptotic expansions of asymptotic coefficients of false theta functions} \label{subsec:false_theta_asymptotic}


To conclude this section, we derive asymptotic expansions as $ k \to \infty $ of asymptotic coefficients of false theta functions as $ \tau \to h/k $ for fixed $ h \in \Z \smallsetminus \{ 0 \} $.
We need such a result to derive the asymptotic expansions of WRT invariants.

Our main result in this section is the following statement.

\begin{thm} \label{thm:false_theta_asymp}
	Let $ r $ and $ s $ be positive integers, 
	$ S \in \Sym_r^+ (\Q) $ be a positive definite symmetric matrix,
	$ A \in \Mat_{s, r} (\Z) $ be a matrix of rank $ d \ge 0 $, 
	Let $ \overline{g}(x) \in \overline{\frakQ}^0 $ be a quasi-polynomial, 
	and $ \alpha \in \Q^r $ be a vector.
	For these data, we define
	\[
	\widetilde{\Theta} (\tau)
	\coloneqq
	\sum_{m \in \alpha + \Z^r} \sgn(Am) \overline{g}(m) q^{\transpose{m} Sm /2}.
	\]
	Then, we have the following asymptotic formulas.
	
	\begin{enumerate}
		\item \label{item:thm:false_theta_asymp:1}
		For any rational number $ \rho \in \Q $, we have
		\[
		\widetilde{\Theta} (\tau + \rho)
		\sim
		\sum_{j \in \frac{1}{2} \Z}
		c_j (\rho) \tau^j
		\quad \text{ as } \tau \to 0
		\]
		for some complex number
		\[
		c_j (\rho) \in 
		\sprod{ \bm{e} \left( \frac{\rho}{2} \transpose{m} Sm \relmiddle| m \in \alpha + \Z^r, \overline{g}(m) \neq 0 \right) }_\Q,
		\]
		which vanishes if $ j < -(r - d + \deg P)/2 $ or
		$ j \in \frac{1}{2} \Z_{\ge 0} \smallsetminus \Z_{\ge 0} $.
		
		\item \label{item:thm:false_theta_asymp:2}
		In the case where $ \rho = h/k \neq 0 $ for fixed $ h \in \Z \smallsetminus \{ 0 \} $, for any $ j $ we have
		\[
		c_j (\rho)
		\sim
		\sum_{\theta \in \calS} \bm{e} \left( \frac{\theta}{\rho} \right) \widetilde{\varphi}_{\theta, h, j} \left( \frac{1}{\sqrt{k}} \right) 
		\quad \text{ as } k \to \infty
		\]
		for some finite set $ \calS \subset \Q / \Z $ independent of $ \rho $ and $ j $ and 
		some formal power series 
		$ \widetilde{\varphi}_{\theta, h, j} (u) \in\bbC((u)) $.
		
		\item \label{item:thm:false_theta_asymp:3}
		Suppose that we can write
		\[
		\widetilde{\Theta} (\tau)
		=
		\sum_{m \in \Z^r} \widetilde{g}(m) q^{\transpose{(m + \alpha)} S (m + \alpha) /2},
		\]
		where $ \widetilde{g}(x) = \widetilde{\coe}[G](x) \in \overline{\frakQ} $ is a false quasi-polynomial with
		$ G(z) = G_1 (z_1) \cdots G_r (z_r) \in \frakR $ such that $ \overline{g}(x) \coloneqq \overline{\coe}[G](x) \in \overline{\frakQ}^0 $.
		Then, we can write $ \calS $ in \cref{item:thm:false_theta_asymp:2} explicitly as
		\[
		\calS = \{ 0 \} \cup \bigcup_{\pi \in \Pi_r} \calS_\pi^{},
		\]
		where
		\[
		\calS_\pi^{} \coloneqq
		\left\{
		-\frac{1}{2} \transpose{n} S_\pi^* n
		\bmod \Z
		\relmiddle|
		n \in \calP_\pi
		\right\}
		\]
		and $ S_\pi^* $ and $ \calP_\pi $ are as in \cref{thm:false_theta}.
	\end{enumerate}
\end{thm}

Here, we use a tilde for the asymptotic series $ \widetilde{\varphi}_{\theta, h, j} (u) $, following the standard notation in resurgence theory.

\begin{rem}
	For the case where $ r=1 $, we have $ S \in \Q_{>0} $ and 
	\[
	\calS = \{ 0 \} \cup 
	\left\{
	-\frac{a^2}{2S}
	\bmod \Z
	\relmiddle|
	a \in \Q \text{ is a pole of $ G(\bm{e}(\xi)) $}
	\right\}.
	\]
	In this case, our asymptotic formula is essentially obtained by Andersen--Misteg\aa{}rd~\cite[Equation (5.22)]{Andersen-Mistegard}.
\end{rem}

As shown below, \cref{thm:false_theta_asymp} \cref{item:thm:false_theta_asymp:1} follows readily from \cref{prop:asymp_EM}.

\begin{proof}[Proof of \cref{thm:false_theta_asymp}]
	First, we prove \cref{item:thm:false_theta_asymp:1}.
	By applying \cref{cor:asymp_EM} \cref{item:prop:asymp_EM:2} for 
	\[
	f(x) = e^{-\pi \transpose{x} Sx}, \quad
	t = \sqrt{\frac{\tau}{\iu}}, \quad
	\overline{g}(x) \text{ as } \overline{g} (x) \bm{e} \left( \frac{\rho}{2} \transpose{x}Sx \right),
	\]
	we have
	\[
	\widetilde{\Theta} (\tau + \alpha)
	\sim
	\sum_{1 \le d' \le d} a_{d'}
	\varphi_{d'} \odot f_{B_{d'}}^{} (t)
	\quad \text{ as } \tau \to 0
	\]
	with the following notation:
	\begin{itemize}
		\item Let $ (a_{d'})_{1 \le d' \le d} $ be a family of rational numbers, 
		$ (B_{d'})_{1 \le d' \le d} $ be a family of matrices in $ \Mat_r (\Z) \cap \GL_r (\Q) $, and
		$ (L_{d'})_{1 \le d' \le d} $ be a family of free $ \Z $-submodules with $ \Z^{d'} \subset L_{d'} \subset \Q^{d'} $ which detemined in \cref{prop:sign_Am_to_m}.
		\item Let
		\begin{align}
			\varphi_{d'} (t_1, \dots, t_{d'})
			&\coloneqq
			\sum_{m' \in L_{d'} \cap \Q_{\ge 0}^{d'}} 
			\overline{g} \left( B_{d'} \pmat{m' \\ 0} \right)
			\bm{e} \left( \frac{\rho}{2} (\transpose{m'}, 0) \transpose{B_{d'}} S B_{d'} \pmat{m' \\ 0} \right)
			e^{m'_1 t_1 + \cdots + m'_{d'} t_{d'}}
			\\
			&\in \Q[\alpha_1, \dots, \alpha_r] ((t_1, \dots, t_{d'})).
		\end{align}
		\item Let $ f_{B_{d'}}^{} (x) \coloneqq f \left( B_{d'} \transpose{(x_1, \dots, x_{d'}, 0, \dots, 0)} \right) $.
	\end{itemize}
	By a choice in \cref{prop:sign_Am_to_m}, $ B_{d'} $ induces an isomorphism $ L_{d'} \oplus \Z^{r-d'} \xrightarrow{\sim} \Z^r $.
	Thus, each asymptotic coefficient is an element of
	\[
	\sprod{ \bm{e} \left( \frac{\rho}{2} \transpose{m} Sm \relmiddle| m \in \Z^r, \overline{g}(m) \neq 0 \right) }_\Q.
	\]
	Moreover, by \cref{rem:deg_q-poly} \cref{item:rem:deg_q-poly:order}, we have
	$ \ord_{t=0} \varphi \odot f (t) \ge \deg -\overline{g} - r $.
	Moreover, since $ f(x) $ is an even function, the $ j $-th coefficients of $ \varphi \odot f (t) $ vanish for positive odd integer $ j $.
	
	Next, we prove \cref{item:thm:false_theta_asymp:2,item:thm:false_theta_asymp:3}.
	We have the modular transformation formula
	\[
	\sqrt{\det S} {\sqrt{\frac{\iu}{\tau}}}^{\, r} \cdot \widetilde{\Theta} \left( -\frac{1}{\tau} \right)
	=
	\sum_{\pi \in \Pi_r} \sum_{1 \le i \le N_\pi} \widetilde{\Theta}_{\pi, i} (\tau) \Omega_{\pi, i} (\tau)
	+ \Omega(\tau),
	\]
	in \cref{thm:false_theta}.
	It suffices to consider an asymptotic expansion as $ \tau \to \tau - 1/\rho $ for the right hand side.
	By \cref{item:thm:false_theta_asymp:1}, we have an asymptotic expansion
	\[
	\widetilde{\Theta}_{\pi, i} \left( \tau - \frac{1}{\rho} \right)
	\sim
	\sum_{j \in \frac{1}{2} \Z}
	c_{\pi, i, j} (\rho) \tau^j
	\quad \text{ as } \tau \to 0
	\]
	for some complex number
	\[
	c_{\pi, i, j} (\rho) \in 
	(2\pi\iu)^{-j} \cdot
	\sprod{ 
		\bm{e} \left( -\frac{1}{2\rho} \transpose{n} S_\pi^* n\right) 
		\relmiddle| 
		n \in U_\pi^{-1} (\Z^r) \cap \calP_{I_s^\complement}^{}
	}_\Q.
	\]
	On the other hand, $ \Omega_{\pi, i} (\tau) $ and $ \Omega(\tau) $ have a representation
	\[
	\Omega'(\tau)
	=
	\int_{C_+^{r'}} q^{\transpose{\xi'} S' \xi'/2} \varphi(\xi') d\xi',
	\]
	where $ r' $ is a positive integer, $ \xi' $ is a tuple of $ r' $ variables, $ S' \in \Sym_{r'}^+ (\Q) $ is a positive definite symmetric matrix,
	and $ \varphi(\xi) $ is a meromorphic function whose all poles lie on the real axis.
	In particular, this function $ \Omega'(\tau) $ is holomorphic at $ \tau = -1/\rho $ and its value is
	\[
	\Omega' \left( -\frac{1}{\rho} \right)
	=
	\int_{C_+^{r'}} \bm{e} \left( -\frac{1}{2\rho} \transpose{\xi'} S' \xi' \right) \varphi(\xi') d\xi'.
	\]
	By deforming the integral path $ C_+ $ into $ C_\varepsilon $ define as in \cref{fig:contour_lem_analy_cont},
	we can apply \cref{prop:asymp_stationary_phase}.
	Then, we obtain an asymptotic expansion
	\[
	\Omega' \left( -\frac{1}{\rho} \right)
	\sim
	\varphi \odot f \left( \frac{1}{\sqrt{2\pi\iu \rho}} \right) \cdot \rho^{r'}
	\quad \text{ as } k \to \infty,
	\]
	where $ \rho = h/k $ for fixed $ h $ and $ f(x) \coloneqq \exp ( -\pi \transpose{x'} (S')^{-1} x') $.
	Thus, we obtain the claim.
\end{proof}


\section{Quantum modularity of indefinite theta functions} \label{sec:indefinite_false}


In a 1920 letter to Hardy, Ramanujan listed seventeen functions he termed ``mock theta functions.''
A complete description of their modular transformation formulas was established by Zwegers in his 2002 thesis~\cite{Zwegers:thesis}.
His approach is based on expressions of the mock theta functions as indefinite theta functions.
He constructed nonholomorphic \emph{modular completions} by inserting error functions and established modular transformation formulas for them.

In this section, we first provide an alternative proof of the modular transformation formula for indefinite theta functions obtained by Zwegers by using modular series instead of modular completion.
We also establish a modular transformation formula for further cases of indefinite theta functions.


\subsection{Zwegers' results} \label{subsec: indefinite_theta_Zwegers}


To begin with, we review results by Zwegers~\cite{Zwegers:thesis} for indefinite theta functions.
Throughout this section, we fix a positive integer $ r \ge 2 $ and an indefinite symmetric matrix $ S \in \Sym_r (\Z) $ of type $ (r-1,1) $.
We define an indefinite quadratic form $ Q \colon \R^r \to \R $ as $ Q(x) \coloneqq \transpose{x}Sx/2 $ 
and a bilinear form $ B \colon \R^r \times \R^r \to \R $ as $ B(x, y) \coloneqq \transpose{x}Sy $.
We also fix vectors $  \alpha, \beta, c, c' \in \R^r $ such that $ c $ and $ c' $ are linearly independent and satisfy 
$ Q(c), Q(c'), B(c, c') < 0 $.


\begin{dfn}[{Zwegers~\cite[Definition 2.1]{Zwegers:thesis}}]
	We define an \emph{indefinite theta function} as
	\[
	\Theta_{\alpha, \beta}^+ (\tau)
	\coloneqq
	\sum_{m \in \alpha + \Z^r}
	\left(
	\sgn \left( B(c, m) \right) - \sgn \left( B(c', m) \right)
	\right)
	\bm{e} \left( B(\beta, m) \right) q^{Q(m)}.
	\]
\end{dfn}

Zwegers essentially proved the following modular transformation formula.

\begin{thm} \label{thm:indef_theta_mod_trans}
	If $ B(c'', m) \notin \Z $ for each $ c'' \in \{ c, c' \} $ and any $ m \in (\alpha + \Z^r) \cup (\beta + S^{-1} (\Z^r)) $, 
	then we have
	\begin{align}
		\Theta_{\alpha, \beta}^+ \left( -\frac{1}{\tau} \right)
		&=
		\frac{\iu}{\sqrt{-\det S}} \sqrt{\frac{\tau}{\iu}}^r \bm{e} \left( B(\alpha, \beta) \right)
		\sum_{\beta' \in (\beta + S^{-1}(\Z^r)) / \Z^r} \Theta_{\beta', -\alpha}^+ (\tau)
		\\
		&\phant
		- \iu \sum_{c'' \in \{ c, c' \}} \varepsilon_{c''} 
		\sum_{\mu \in (\alpha + \Z^r) \cap \Delta_{c''}} 
		\calR_{\mu, \beta, c''} \left( -\frac{1}{\tau} \right),
	\end{align}
	with the following notation:
	\begin{itemize}
		\item Denote $ \varepsilon_{c} \coloneqq 1 $ and $ \varepsilon_{c'} \coloneqq -1 $.
		\item Denote $ \sprod{c''}_{\Z}^{\perp} \coloneqq \{ n \in \Z^r \mid B(c'', n) = 0 \} $.
		
		\item Let $ \Delta_{c''} $ be a fundamental domain of the action of $ \sprod{c''}_{\Z}^{\perp} $ on 
		$ \{ x \in \R^r \mid B(c'', x) / 2Q(c'') \in \left[ 0, 1 \right) \} $.
		
		\item For $ x \in \R^r $, let
		$ x^{\perp} \coloneqq x - \frac{B(c'', x)}{2Q(c'')} c'' \in \sprod{c''}_{\Z}^{\perp} \otimes_{\Z} \R $.
		
		\item For vectors $ \alpha', \beta' \in \sprod{c''}_{\Z}^{\perp} \otimes_\Z \R $, define a theta function for the positive definite quadratic form on $ \sprod{c''}_{\Z}^{\perp} $ defined by $ Q $ as
		\[
		\theta_{\eval{Q}_{\sprod{c''}_{\Z}^{\perp}}, \alpha', \beta'} (\tau) 
		\coloneqq
		\sum_{\lambda \in \alpha' + \sprod{c''}_{\Z}^{\perp}} \bm{e} (B(\beta', \lambda)) q^{Q(\lambda)},
		\]
		
		\item For $ a, b \in \R $, define a unary theta function as
		\[
		g_{a, b} (\tau) \coloneqq
		\sum_{n \in a + \Z} \bm{e} (bn) n q^{n^2/2}
		\]
		and its Eichler integral as
		\[
		g_{a, b}^* (\tau) \coloneqq
		\int_{0}^{\iu \infty} g_{\alpha, \beta} (z) \sqrt{\frac{\iu}{z + \tau}} dz.
		\]
		
		\item For $ \mu \in \R^r $, we define
		\[
		\calR_{\mu, \beta, c''} (\tau)
		\coloneqq
		\theta_{\eval{Q}_{\sprod{c''}_{\Z}^{\perp}}, \mu^{\perp}, \beta^{\perp}} (\tau)
		g^*_{\frac{B(c'', \mu)}{2Q(c'')}, B(c'', \beta)} (-2Q(c'') \tau).
		\]
	\end{itemize}
\end{thm}


Although Zwegers did not state this theorem explicitly, it follows from his modular transformation formula for the modular completion \cite[Corollary 2.9 (5) and Subsection 4.2 and 4.3]{Zwegers:thesis} together with the following proposition.

\begin{prop}[{Zwegers~\cite[Proposition 4.3]{Zwegers:thesis}}] \label{prop:Zwegers_error}
	Let $ E \colon \bbH \times \R \to \bbC $ be a map satisfying the following two conditions:
	\begin{itemize}
		\item For any $ r \in \R_{>0} $, it holds that $ E(r^2 \tau; u) = E(\tau; ru) $.
		\item There exists $ \delta>0 $ such that 
		$ E(\tau; u) = O\left( \abs{q}^{(1+\delta) u^2/2} \right) $ as $ \abs{u} \to \infty $ uniformly in $ \tau $.
	\end{itemize}
	Then, we have
	\[
	\sum_{\nu \in \alpha + \Z^r} E \left( \tau; \frac{B(c, \nu)}{\sqrt{-2Q(c)}} \right) \bm{e} (B(\beta, \nu)) q^{Q(\nu)}
	=
	\sum_{\mu \in (\alpha + \Z^r) \cap \Delta_{c''}} \theta_{\restrict{Q}{\sprod{c}_{\Z}^{\perp}}, \mu^{\perp}, \beta^{\perp}} (\tau)
	g^{*, E}_{-\frac{B(c, \mu)}{2Q(c)}, -B(c, \beta)} (-2Q(c) \tau),
	\]
	where for $ a, b \in \R $ we define
	\[
	g^{*, E}_{a, b} (\tau)
	\coloneqq
	\sum_{n \in a + \Z} 
	E \left( \tau; n \right) \bm{e} \left( bn \right) n q^{-n^2 /2}.
	\]
\end{prop}

\begin{proof}
	For the convenience of the reader, we give a proof.
	By the decomposition
	\[
	\alpha + \Z^r 
	=
	\left\{ \mu \in \alpha + \Z^r \relmiddle| \frac{B(c, \mu)}{\transpose{c} Sc} \in \left[ 0, 1 \right) \right\} 
	\oplus \Z c
	=
	\coprod_{\mu \in (\alpha + \Z^r) \cap \Delta_c} \left( \left( \mu^{\perp} + \sprod{c}_{\Z}^{\perp} \right) \oplus \left( \frac{B(c, \mu)}{2Q(c)} + \Z \right) c \right),
	\]
	the left hand side equals
	\begin{align}
		&\phant
		\sum_{\mu \in (\alpha + \Z^r) \cap \Delta_c} \sum_{\lambda \in \mu^{\perp} + \sprod{c}_{\Z}^{\perp}} 
		\sum_{n \in \frac{B(c, \mu)}{2Q(c)} + \Z}
		E \left( \tau; \frac{B(c, c)}{\sqrt{-2Q(c)}} n \right) \bm{e} \left( B \left( \beta, \lambda + nc \right) \right) q^{Q(\lambda + nc)}
		\\
		&=
		\sum_{\mu \in (\alpha + \Z^r) \cap \Delta_c} 
		\left(
		\sum_{\lambda \in \mu^{\perp} + \sprod{c}_{\Z}^{\perp}} 
		\bm{e} \left( B \left( \beta^{\perp}, \lambda \right) \right) q^{Q(\lambda)}
		\right)
		\left(
		\sum_{n \in \frac{B(c, \mu)}{2Q(c)} + \Z} 
		E \left( -2Q(c)\tau; -n \right) \bm{e} \left( B \left( \beta, c \right) n \right) q^{Q(c) n^2}
		\right),
	\end{align}
	which equals the right hand side.
\end{proof}

\begin{rem}
	Zwegers~\cite[Proposition 4.3]{Zwegers:thesis} stated the above claim for the case where
	$ E(\tau; u) = \sgn(u) - \erf(u \sqrt{2\pi \Im(\tau)}) $.
	In this case, since we can write
	\[
	E(\tau; u) =
	-\iu u q^{u^2/2} \int_{-\overline{\tau}}^{\iu \infty} \bm{e} \left( \frac{u^2}{2} z \right) \sqrt{\frac{\iu}{z + \tau}} dz,
	\]
	we have an expression as a nonholomorphic Eichler integral
	\[
	g^{*, E}_{\alpha, \beta} (\tau)  =
	-\iu \int_{-\overline{\tau}}^{\iu \infty} g_{\alpha, -\beta} (z) \sqrt{\frac{\iu}{z + \tau}} dz.
	\]
\end{rem}

We note here that the error term $ \calR_{\mu, \beta, c} (\tau) $ in \cref{thm:indef_theta_mod_trans} satisfies the following modular transformation formula as follows from Zwegers' argument~\cite[Proof of Proposition 4.5]{Zwegers:thesis}.

\begin{lem} \label{lem:indef_theta_error_term_modularity}
	For any $ \mu, \beta \in \R^r $, we have
	\[
	\calR_{\mu, \beta, c} \left( -\frac{1}{\tau} \right)
	=
	\frac{\iu}{\sqrt{-\det S}} \sqrt{\frac{\tau}{\iu}}^r \bm{e} \left( B(\mu, \beta) \right)
	\sum_{\beta' \in (\beta + S^{-1}\Z^r) / \Z^r}
	\calR_{\beta', -\mu, c} (\tau).
	\]
\end{lem}

\begin{proof}
	The claim follows from the following facts:
	\begin{itemize}
		\item The modular transformation formulas for positive definite theta functions
		\[
		\theta_{\eval{Q}_{\sprod{c}_{\Z}^{\perp}}, \mu^{\perp}, \beta^{\perp}} \left( -\frac{1}{\tau} \right)
		=
		\frac{1}{\sqrt{\det \eval{S}_{ \sprod{c}_{\Z}^{\perp} } }}
		\sqrt{\frac{\tau}{\iu}}^{\, r-1} \bm{e} \left( B(\mu^{\perp}, \beta^{\perp}) \right)
		\sum_{\beta' \in (\beta^{\perp} + S^{-1} (\sprod{c}_{\Z}^{\perp})) / \sprod{c}_{\Z}^{\perp}}
		\theta_{\eval{Q}_{\sprod{c}_{\Z}^{\perp}}, -\beta', \mu^{\perp}} (\tau).
		\]
		
		\item We have
		\begin{equation} \label{eq:Eichler_int_modularity}
			g_{a, b}^* \left( \frac{1}{\tau} \right)
			= \iu \bm{e} (ab) \sqrt{\frac{\tau}{\iu}} g_{b, -a}^* (\tau)
		\end{equation}
		by the fact
		\[
		g_{a, b} (\tau) 
		= \iu \bm{e} (ab) \sqrt{\frac{\tau}{\iu}} g_{b, -a} (\tau),
		\]
		which is stated in \cite[Proposition 1.15 (3) and (5)]{Zwegers:thesis}.
		
		\item For any $ N \in \Z_{>0} $ and $ a, b \in \R $, we have
		\[
		g^*_{a, b/N} \left( \frac{\tau}{N} \right)
		=
		\sum_{a' \in (a + \Z) / N\Z} g^*_{a'/N, b} (N\tau),
		\]
		which follows from
		\[
		g_{a, b/N} \left( \frac{\tau}{N} \right)
		=
		N \sum_{a' \in (a + \Z) / N\Z} g_{a'/N, b} (N\tau).
		\]
		
		\item We have the isomorphism of
		\[
		\begin{array}{ccc}
			(\beta + S(\Z^r)) / \Z^r & \xlongrightarrow{\sim} &
			\left( \beta^{\perp} + S^{-1} (\sprod{c}_{\Z}^{\perp}) \right) / \sprod{c}_{\Z}^{\perp} 
			\oplus (B(c,\beta) + \Z) / 2Q(c) \Z \\
			\beta' & \longmapsto & (\beta^{\prime \perp}, B(c, \beta')).
		\end{array}
		\]
		
	\end{itemize}
\end{proof}

To conclude this subsection, we consider quantum modularity of indefinite theta functions, which follows from \cref{thm:indef_theta_mod_trans}.

\begin{thm} \label{thm:indefinite_theta_quantum_modular}
	If $ B(c'', m) \notin \Z $ for each $ c'' \in \{ c, c' \} $ and any $ m \in (\alpha + \Z^r) \cup (\beta + S^{-1} (\Z^r)) $, 
	then $ \Theta_{\alpha, \beta}^+ (\tau) $ is a component of a vector-valued holomorphic quantum modular form, that is, there exist a vector-valued holomorphic function $ F \colon \bbH \to \bbC^N $ which has $ \Theta_{\alpha, \beta}^+ \left( \tau \right) $ as a component and a matrix-valued holomorphic function $ M \colon \bbH \to \GL_N^{} (\bbC \smallsetminus \R_{\le 0}) $ such that
	\[
	F \left( -\frac{1}{\tau} \right) = M(\tau) F(\tau).
	\]
\end{thm}

This theorem follows from the following remarks.

\begin{rem}
	If $ a, b \notin \Z $, then $ g_{a, b}^* (\tau) $ extends holomorphically to $ \bbC \smallsetminus \R_{\le 0} $.
	Indeed, let $ a_0 \coloneqq \min \{ \abs{n} \mid n \in a + \Z \} > 0 $, 
	then since $ g(z) = O(e^{-\pi \alpha_0 \Im(z)}) $ as $ \Im(z) \to \infty $,
	for any $ \varepsilon > 0 $, an integral
	\[
	\int_{\iu \varepsilon}^{\iu \infty} g_{\alpha, \beta} (z) \sqrt{\frac{\iu}{z + \tau}} dz
	\]
	converges for any $ \tau \in \R_{>0} $.
	In addition, by \cref{eq:Eichler_int_modularity}, we have
	\[
	\int_0^{\iu \varepsilon} g_{\alpha, \beta} (z) \sqrt{\frac{\iu}{z + \tau}} dz
	=
	\iu \bm{e} (\alpha \beta) \sqrt{\frac{\iu}{\tau}} 
	\int_{\iu \varepsilon^{-1}}^{\iu \infty} g_{\beta, -\alpha} (z) \sqrt{\frac{\iu}{z - 1/\tau}} dz
	\]
	by replacing $ z $ by $ -1/z $ and the right hand side converges for any $ \tau \in \R_{>0} $ by the same reason.
\end{rem}


\subsection{Treating indefinite theta functions as modular series} \label{subsec: indefinite_theta_Zwegers_modular_series}


In this subsection, we give an alternative proof of \cref{thm:indef_theta_mod_trans} by using a framework of modular series.
We define $ \gamma(\tau; x) \coloneqq \left( \sgn(B(x, c)) - \sgn(B(x, c')) \right) q^{Q(x)} $.
This function is holomorphic in $ \tau \in \bbH $ and of exponential decay in $ y \in \R^r $ uniformly in $ \tau $ since
$ \gamma(\tau; x) = 0 $ for any $ x \in \R^r $ with $ Q(x) < 0 $.
Although $ \gamma(\tau; x) $ is not continuous in $ x $, it is measurable.
Thus, we can apply the Poisson summation formula.

The Fourier transform $ \gamma(\tau; x) $ is calculated as follows.

\begin{lem} \label{lem:indefinite_Gaussian_Fourier}
	We have
	\[
	\widehat{\gamma}(\tau; \xi)
	=
	\frac{-\iu}{\sqrt{-\det S}} \sqrt{\frac{\iu}{\tau}}^{\, r}
	\left(
	\erf \left( \sqrt{\frac{\pi\iu}{\tau}} \frac{\transpose{c}\xi}{\sqrt{-2Q(c)}} \right)
	- \erf \left( \sqrt{\frac{\pi\iu}{\tau}} \frac{\transpose{c'}\xi}{\sqrt{-2Q(c')}} \right)
	\right)
	\widetilde{q}^{\transpose{\xi}S^{-1} \xi/2},
	\]
	where $ \widetilde{q} \coloneqq \bm{e} (-1/\tau) $ and $ \erf(z) $ is the error functions defined as
	\[
	\erf(z) \coloneqq
	\frac{2}{\sqrt{\pi}} \int_{0}^{z} e^{-t^2} dt.
	\]
	In particular, we have
	\[
	\widehat{\gamma} \left( \tau; S\xi \right)
	=
	\frac{-\iu}{\sqrt{-\det S}} \sqrt{\frac{\iu}{\tau}}^{\, r}
	\left(
	\gamma \left( -\frac{1}{\tau}; \xi \right)
	- \left( E^* \left( -\frac{1}{\tau}; \frac{B(c, \xi)}{\sqrt{-2Q(c)}} \right) - E^* \left( -\frac{1}{\tau}; \frac{B(c', \xi)}{\sqrt{-2Q(c')}} \right) \right)
	\widetilde{q}^{\, Q(\xi)}
	\right),
	\]
	where $ E^* (\tau; u) \coloneqq \sgn(u) - \erf(u \sqrt{-\pi\iu \tau}) $.
\end{lem}

\begin{proof}
	\textbf{Step 1}. We reduce to a simple case.
	Let $ \sprod{c, c'}_\R^\perp \coloneqq \{ x \in \R^r \mid B(c, x) = B(c', x) = 0 \} $.
	Since $ \R^r = \sprod{c, c'}_\R \oplus \sprod{c, c'}_\R^\perp $ is an orthogonal decomposition with respect to the bilinear form $ B $,
	the subspace $ \sprod{c, c'}_\R^\perp $ has dimension $ r-2 $ and $ S $ is positive definite on it.
	Thus, we can write
	\[
	\transpose{C} SC
	=
	\pmat{S_0 & \\ & S_0^\perp},
	\]
	where $ C^\perp \in \Mat_{r, r-2} (\R) $ is the matrix whose columns form a basis of $ \sprod{c, c'}_\R^\perp $,
	$ C \coloneqq (c, c', C^\perp) \in \GL_r(\R) $,
	$ S_0 \coloneqq \transpose{(c, c')} S (c, c') \in \Sym_2(\R) $ is the symmetric matrix of type $ (1, 1) $,
	and $ S_0^\perp \coloneqq \transpose{c^{\perp}} S c^{\perp} \in \Sym_{r-2}^+ (\R) $ is the positive definite symmetric matrix.
	
	Let
	\[
	\pmat{a_0 & b_0 \\ b_0 & c_0} \coloneqq S_0, \quad
	\lambda \coloneqq \frac{b_0 + \sqrt{b_0^2 - a_0 c_0}}{a_0}, \quad
	\lambda' \coloneqq \frac{b_0 - \sqrt{b_0^2 - a_0 c_0}}{a_0}, \quad
	P_0 \coloneqq \sqrt{\frac{-a_0}{2}} \pmat{1 & \lambda \\ -1 & -\lambda'}.
	\]
	Then, we have the following:
	\begin{itemize}
		\item We have $ \lambda, \lambda' \in \R_{>0} $ and $ P_0 \in \GL_2(\R) $ since $ a_0, b_0, c_0, a_0 c_0 - b_0^2 < 0 $.
		
		\item For $ u, v, x, y \in \R $, if
		\[
		\pmat{u \\ v} = P_0 \pmat{x \\ y} = \sqrt{\frac{-a_0}{2}} \pmat{x + \lambda' y \\ -(x + \lambda y)},
		\]
		then we have
		\[
		\transpose{P_0} S P_0 = \pmat{0 & 1 \\ 1 & 0}.
		\]
		
		\item We have
		\[
		S_0 P_0
		=
		\transpose{P_0}^{-1} \pmat{0 & 1 \\ 1 & 0}
		=
		\sqrt{\frac{-2a_0}{b_0^2 - a_0 c_0}} \pmat{1 & -\lambda \\ 1 & -\lambda'}.
		\]
	\end{itemize}
	Therefore, if we set
	\[
	P \coloneqq C \pmat{P_0 & \\ & P_0^\perp} \in \GL_r(\R), \quad
	S' \coloneqq 
	\pmat{0 & 1 & & & \\
		1 & 0 & & & \\
		& & 1 & & \\
		& & & \ddots & \\
		& & & & 1},
	\]
	then we obtain
	\begin{align}
		\transpose{P} SP &= S', \\
		S^{-1} &= P S' \transpose{P}, \\
		\transpose{(c, c')} SP
		&=
		\sqrt{\frac{-2a_0}{b_0^2 - a_0 c_0}} 
		\pmat{1 & -\lambda & 0 & \cdots & 0 \\
			1 & -\lambda' & 0 & \cdots & 0}, 
		\\
		\transpose{(c, c')} S (c, c')
		&=
		\frac{-2a_0}{b_0^2 - a_0 c_0}
		\pmat{-2\lambda & -\lambda - \lambda' \\ -\lambda - \lambda' & -2\lambda'}.
		\label{eq:lem:indefinite_Gaussian_Fourier:1}
	\end{align}
	Thus, we have $ \gamma(\tau; x) = \gamma_0(\tau; P^{-1} x) $ where
	\[
	\gamma_0 (\tau; x)
	\coloneqq
	\left( \sgn(x_1 - \lambda x_2) - \sgn(x_1 - \lambda' x_2) \right)
	q^{x_1 x_2 + (x_3^2 + \cdots + x_r^2)/2}.
	\]
	Here, $ \gamma_0 $ is the specialization of $ \gamma $ in the statement with $ S = S_0, c = \transpose{(1, -\lambda, 0, \dots, 0)} $, and $ c' = \transpose{(1, -\lambda', 0, \dots, 0)} $.
	The claim for $ \gamma_0 $ is stated as
	\[
	\widehat{\gamma}_0 (\tau; \xi)
	=
	\left(
	\erf \left( \sqrt{\frac{\pi\iu}{\tau}} \frac{\xi_2 - \lambda \xi_1}{2\lambda} \right)
	- \erf \left( \sqrt{\frac{\pi\iu}{\tau}} \frac{\xi_2 - \lambda' \xi_1}{2\lambda'} \right)
	\right)
	\widetilde{q}^{\xi_1 \xi_2 + (\xi_3^2 + \cdots + \xi_r^2)/2}.
	\]
	If this holds, then by $ \gamma(\tau; x) = \gamma_0(\tau;  P^{-1} x) $, we have
	\begin{align}
		\widehat{\gamma}(\tau; \xi)
		&=
		\frac{1}{\abs{\det P}} \widehat{\gamma}_0(\tau; \transpose{P}\xi)
		\\
		&=
		\frac{1}{\sqrt{\abs{\det S}}}
		\left(
		\erf \left( \sqrt{\frac{\pi\iu}{-2Q(c) \tau}} \transpose{c}\xi \right)
		- \erf \left( \sqrt{\frac{\pi\iu}{-2Q(c') \tau}} \transpose{c'}\xi \right)
		\right)
		\widetilde{q}^{\transpose{\xi}S^{-1} \xi/2},
	\end{align}
	which implies the claim for $ \gamma $.
	Thus, the claim is reduced to the case for $ \gamma_0 $.
	Moreover, we can assume $ r=2 $.
	
	\textbf{Step 2}.
	We prove the claim for $ \gamma_0 $ with $ r=2 $.
	We can assume $ \lambda > \lambda' > 0 $.
	
	If $ 0, x_1, \lambda x_2 $, and $ \lambda' x_2 $ are pairwise distinct, then we have
	\begin{align}
		\sgn(x_1 - \lambda x_2) - \sgn(x_1 - \lambda' x_2)
		&=
		\begin{cases}
			-2 & \text{ if } 0 < \lambda'  x_2 < x_1 < \lambda x_2, \\
			2 & \text{ if } 0 > \lambda'  x_2 > x_1 > \lambda x_2, \\
			0 & \text{ otherwise}.
		\end{cases}
	\end{align}
	Thus, we have
	\begin{align}
		\widehat{\gamma}_0 (\tau; \xi)
		&=
		-2 \int_{0 < \lambda'  x_2 < x_1 < \lambda x_2} q^{x_1 x_2}
		\left( \bm{e} (-\transpose{\xi}x) - \bm{e} (\transpose{\xi}x) \right) dx
		\\
		&=
		-2 \int_0^\infty dx_2 \cdot
		\left[
		\frac{q^{x_1 x_2}}{2\pi\iu}
		\left(
		\frac{\bm{e} (-\transpose{\xi}x)}{\tau x_2 - x_1}
		- \frac{\bm{e} (\transpose{\xi}x)}{\tau x_2 + x_1}
		\right)
		\right]_{x_1 = \lambda'  x_2}^{\lambda x_2}
		\\
		&=
		-2 \int_0^\infty dx_2 \cdot
		\frac{1}{2\pi\iu}
		\left(
		\frac{\bm{e}(- \xi_2 x_2)}{\tau x_2 - \xi_1} 
		\left( q^{\lambda x_2^2} \bm{e} (-\lambda \xi_1 x_2) - q^{\lambda'  x_2^2} \bm{e} (-\lambda' \xi_1 x_2) \right)
		\right.
		\\
		&\phant
		\hphantom{
			-2 \int_0^\infty dx_2 \cdot
			\frac{1}{2\pi\iu}
			\left( \right.
		}
		\left.
		- \frac{\bm{e}(\xi_2 x_2)}{\tau x_2 + \xi_1} 
		\left( q^{\lambda x_2^2} \bm{e} (\lambda \xi_1 x_2) - q^{\lambda' x_2^2} \bm{e} (\lambda' \xi_1 x_2) \right)
		\right).
	\end{align}
	By replacing $ x_2 $ by $ -x_2 $ in the second term, we obtain
	\[
	\widehat{\gamma}_0 (\tau; x)
	=
	-2 \int_{\R} dx_2 \cdot
	\frac{1}{2\pi\iu\tau}
	\frac{\bm{e}(-\xi_2 x_2)}{\xi_2 - x_1 / \tau} 
	\left( q^{\lambda x_2^2} \bm{e} (-\lambda \xi_1 x_2) - q^{\lambda' x_2^2} \bm{e} (-\lambda' \xi_1 x_2) \right).
	\]
	Since $ \gamma $ is an odd function, we can assume $ \xi_1 < 0 $.
	In this case, $ \bm{e} (-(x_2 - \xi_1 / \tau) \xi'_2) $ is of exponential decay in $ \xi'_2 \to \infty $ since $ \Im(-\xi_1 / \tau) < 0 $.
	Thus, we have
	\begin{align}
		\widehat{\gamma}_0 (\tau; \xi)
		&=
		-2 \widetilde{q}^{\, \xi_1 \xi_2} \int_{\R} dx_2 \cdot
		\left[
		-\frac{1}{2\pi\iu\tau}
		\frac{\bm{e}(-(x_2 - \xi_1 / \tau) \xi'_2)}{x_2 - \xi_1 / \tau} 
		\right]_{\xi'_2 = \xi_2}^\infty
		\left( q^{\lambda x_2^2} \bm{e} (-\lambda \xi_1 x_2) - q^{\lambda' x_2^2} \bm{e} (-\lambda' \xi_1 x_2) \right)
		\\
		&=
		-\frac{2}{\tau} \widetilde{q}^{\, \xi_1 \xi_2} \int_{\R} dx_2 \cdot
		\int_{\xi_2}^{\infty}
		\bm{e}(-(x_2 - \xi_1 / \tau) \xi'_2) d\xi'_2 \cdot
		\left( q^{\lambda x_2^2} \bm{e} (-\lambda \xi_1 x_2) - q^{\lambda' x_2^2} \bm{e} (-\lambda' \xi_1 x_2) \right)
		\\
		&=
		-\frac{2}{\tau} \widetilde{q}^{\, \xi_1 \xi_2} 
		\int_{\xi_2}^{\infty} d\xi'_2 \cdot \widetilde{q}^{\, -\xi_1 \xi'_2} 
		\int_{\R}
		\left( q^{\lambda x_2^2} \bm{e} (-(\lambda \xi_1 + \xi'_2) x_2) - q^{\lambda' x_2^2} \bm{e} (-(\lambda' \xi_1 + \xi'_2) x_2) \right) dx_2
		\\
		&=
		-\frac{2}{\tau} \widetilde{q}^{\, \xi_1 \xi_2} 
		\int_{x_2}^{\infty} d\xi'_2 \cdot \widetilde{q}^{\, -\xi_1 \xi'_2}
		\left( 
		\sqrt{\frac{\iu}{2\lambda \tau}} \widetilde{q}^{\, (\lambda \xi_1 + \xi'_2)^2 / 4\lambda} 
		- \sqrt{\frac{\iu}{2\lambda' \tau}} \widetilde{q}^{\, (\lambda' \xi_1 + \xi'_2)^2 / 4\lambda'} 
		\right)
		\\
		&=
		-\frac{2}{\tau} \widetilde{q}^{\, \xi_1 \xi_2} 
		\int_{\xi_2}^{\infty} d\xi'_2 \cdot
		\left( 
		\sqrt{\frac{\iu}{2\lambda \tau}} \widetilde{q}^{\, (\lambda \xi_1 - \xi'_2)^2 / 4\lambda} 
		- \sqrt{\frac{\iu}{2\lambda' \tau}} \widetilde{q}^{\, (\lambda' \xi_1 - \xi'_2)^2 / 4\lambda'} 
		\right)
		\\
		&=
		-\frac{2}{\sqrt{\pi} \tau} \widetilde{q}^{\, \xi_1 \xi_2} 
		\left( 
		\int_{(\xi_2 - \lambda \xi_1) \sqrt{\pi \iu / 2\lambda \tau} }^{\infty}
		- \int_{(\xi_2 - \lambda' \xi_1) \sqrt{\pi \iu / 2\lambda' \tau} }^{\infty}
		\right)
		e^{-t^2} dt
		\\
		&=
		-\iu \frac{\iu}{\tau}
		\left(
		\erf \left( \sqrt{\frac{\pi\iu}{2\lambda \tau}} (\xi_2 - \lambda \xi_1) \right)
		- \erf \left( \sqrt{\frac{\pi\iu}{2\lambda' \tau}} (\xi_2 - \lambda' \xi_1) \right)
		\right)
		\widetilde{q}^{\, \xi_1 \xi_2}.
	\end{align}
\end{proof}

In his proof of \cref{thm:indef_theta_mod_trans}, Zwegers constructs the modular completion by introducing error function factors, whereas in our approach \cref{lem:indefinite_Gaussian_Fourier} shows that these error functions arise naturally.

\begin{proof}[An alternative proof of \cref{thm:indef_theta_mod_trans}]
	We can write 
	\[
	\Theta_{\alpha, \beta}^+ (\tau)
	=
	\sum_{m \in \alpha + \Z^r}
	\bm{e} \left( B(\beta, m) \right) \gamma(\tau; m).
	\]
	By the Poisson summation formula, we have
	\[
	\sum_{n \in \alpha + \Z^r}
	\bm{e} \left( B(\beta, n) \right) \widehat{\gamma} \left( \tau; Sn \right)
	=
	\frac{1}{\abs{\det S}} \sum_{m \in \Z^r} \bm{e} \left( \transpose{\alpha} m \right) \gamma \left( \tau; S^{-1} m - \beta \right).
	\]
	Thus, we obtain
	\[
	\Theta_{\alpha, \beta}^+ \left( -\frac{1}{\tau} \right)
	- \calE_{\alpha, \beta, c} (\tau) + \calE_{\alpha, \beta, c'} (\tau)
	=
	\frac{\iu}{\sqrt{-\det S}} \sqrt{\frac{\tau}{\iu}}^{\, r}
	\bm{e} \left( B(\alpha, \beta) \right)
	\sum_{\beta' \in (\beta + S^{-1} (\Z^r)) / \Z^r}
	\Theta_{-\beta, \alpha}^+ \left( -\frac{1}{\tau} \right),
	\]
	where
	\[
	\calE_{\alpha, \beta, c} (\tau)
	\coloneqq
	\sum_{n \in \alpha + \Z^r}
	E^* \left(\tau; \frac{B(c, n)}{\sqrt{-2Q(c)}} \right) 
	\bm{e} \left( B(\beta, n) \right) q^{Q(n)}.
	\]
	We have
	\[
	E^*(\tau; u)
	=
	\sgn(u) \frac{2}{\pi} \int_{\abs{u} \sqrt{-\pi\iu \tau}}^{\infty} e^{-t^2} dt
	=
	\iu u \int_{\tau}^{\iu \infty} \bm{e} \left( \frac{u^2}{2} z \right) \sqrt{\frac{\iu}{z}} dz
	=
	\iu u q^{u^2 /2} \int_{0}^{\iu \infty} \bm{e} \left( \frac{u^2}{2} z \right) \sqrt{\frac{\iu}{z + \tau}} dz.
	\]
	In particular, $ E^*(\tau; u) $ satisfies two condition in \cref{prop:Zwegers_error}.
	Thus, by \cref{prop:Zwegers_error}, we have
	\[
	\calE_{\alpha, \beta, c} (\tau)
	=
	\sum_{\mu \in (\alpha + \Z^r) \cap \Delta_{c}} \theta_{\restrict{Q}{\sprod{c}_{\Z}^{\perp}}, \mu^{\perp}, \beta^{\perp}} (\tau)
	g^{*, E}_{-\frac{B(c, \mu)}{2Q(c)}, -B(c, \beta)} (-2Q(c) \tau).
	\]
	For any $ a, b \in \R $, we have
	\[
	g^{*, E}_{a, b} (\tau)
	=
	\sum_{n \in a + \Z} 
	\bm{e} \left( bn \right) n q^{-n^2 /2}
	\left(
	\iu n q^{n^2 /2} \int_{0}^{\iu \infty} \bm{e} \left( \frac{n^2}{2} z \right) \sqrt{\frac{\iu}{z + \tau}} dz
	\right)
	=
	\iu g^*_{a, b} (\tau)
	=
	-\iu g^*_{-a, -b} (\tau).
	\]
	Therefore, we obtain
	\[
	\calE_{\alpha, \beta, c} (\tau)
	=
	-\iu \sum_{\mu \in (\alpha + \Z^r) \cap \Delta_{c}} \calR_{\alpha, \beta, c} (\tau).
	\]
	which proves the claim.
\end{proof}



\subsection{Further cases of indefinite theta functions} \label{subsec: indefinite_theta_Zwegers_outside}


We continue to fix an indefinite symmetric matrix $ S \in \Sym_r (\Z) $ of type $ (r-1,1) $.
In the previous subsection, we assumed that $ c $ and $ c' $ satisfy $ Q(c), Q(c'), B(c, c') < 0 $.
In this subsection, we drop these assumptions.
Our targets are indefinite theta functions defined for any $ c, c' \in \R^r $ as
\[
\Theta_{\alpha, \beta} (\tau)
\coloneqq
\sum_{m \in \alpha + \Z^r}
\left(
\sgn \left( B(c, m) \right) - \sgn \left( B(c', m) \right)
\right)
\bm{e} \left( B(\beta, m) \right) q^{\abs{Q(m)}}.
\]

Our aim is to prove the following formula.

\begin{thm} \label{thm:indef_theta_mod_trans_further_case}
	If $ B(c'', m) \notin \Z $ for each $ c'' \in \{ c, c' \} $ and any $ m \in (\alpha + \Z^r) \cup (\beta + S^{-1} (\Z^r)) $, 
	then we have
	\begin{align}
		\Theta_{\alpha, \beta}^+ \left( -\frac{1}{\tau} \right)
		&=
		\frac{\iu}{\sqrt{-\det S}} \sqrt{\frac{\tau}{\iu}}^r \bm{e} \left( B(\alpha, \beta) \right)
		\sum_{\beta' \in (\beta + S^{-1}\Z^r) / \Z^r} \Theta_{\beta', -\alpha}^+ (\tau)
		\\
		&\phant
		- \iu \sum_{c'' \in \{ c, c' \}} \varepsilon_{c''} 
		\sum_{\mu \in (\alpha + \Z^r) \cap \Delta_{c''}} 
		\calR_{\mu, \beta, c''} \left( -\frac{1}{\tau} \right).
	\end{align}
	Here, in the case when $ Q(c'') > 0 $, we define a set $ \Delta_{c''} $ and a function $ \calR_{\mu, \beta, c''} (\tau) $ as in \cref{thm:indef_theta_mod_trans}, with $ S $ replaced by $ -S $.
	
	In particular, if in addition $ B(c'', \alpha) / 2Q(c'') \notin \Z $ for each $ c'' \in \{ c, c' \} $, 
	then $ \Theta_{\alpha, \beta}^+ (\tau) $ is a component of a vector-valued holomorphic quantum modular form.
\end{thm}

To prove this theorem, it suffices to consider the case when $ Q(c), B(c, c') < 0 $ and $ Q(c') > 0 $ by \cref{thm:indef_theta_mod_trans}.
In this case, the above theorem follows from the following lemma and the same argument in the proof of \cref{thm:indef_theta_mod_trans} in \cref{subsec: indefinite_theta_Zwegers_modular_series}.

\begin{lem} \label{lem:indefinite_Gaussian_Fourier2}
	Let $ \gamma(\tau; x) \coloneqq 
	\left(
	\sgn \left( B(c, x) \right) - \sgn \left( B(c', x) \right)
	\right)
	q^{\abs{Q(x)}} $.
	Then, its Fourier transform is 
	\begin{align}
		\widehat{\gamma}(\tau; \xi)
		=
		\frac{-\iu}{\sqrt{-\det S}} \frac{\iu}{\tau}
		&\left(
		\left(
		\sgn(\transpose{c} \xi) - \sgn(\transpose{c'} \xi)
		\right) 
		\widetilde{q}^{\, \abs{\transpose{\xi}S^{-1} \xi/2}}
		\vphantom{ E^* \left( -\frac{1}{\tau}; \frac{\transpose{c} \xi}{\sqrt{-2Q(c)}} \right) }
		\right.
		\\
		&\phantom{\left( \right.}
		\left.
		-
		E^* \left( -\frac{1}{\tau}; \frac{\transpose{c} \xi}{\sqrt{-2Q(c)}} \right)
		\widetilde{q}^{\, \transpose{\xi}S^{-1} \xi/2}
		+
		E^* \left( -\frac{1}{\tau}; \frac{\transpose{c'} \xi}{\sqrt{2Q(c)}} \right)
		\widetilde{q}^{\, -\transpose{\xi}S^{-1}\xi/2}
		\right).
	\end{align}
	In particular, we have
	\[
	\widehat{\gamma} \left( -\frac{1}{\tau}; S\xi \right)
	=
	\frac{-\iu}{\sqrt{-\det S}} \frac{\tau}{\iu}
	\left(
	\gamma (\tau; \xi)
	- E^* \left( \tau; \frac{B(c, \xi)}{\sqrt{-2Q(c)}} \right) q^{Q(\xi)}
	+ E^* \left( \tau; \frac{-B(c', \xi)}{\sqrt{2Q(c)}} \right) q^{-Q(\xi)}
	\right).
	\]
\end{lem}

\begin{proof}
	By the same argument in the proof of \cref{thm:indef_theta_mod_trans} in \cref{subsec: indefinite_theta_Zwegers_modular_series},
	we may assume
	\[
	S = \pmat{0 & 1 \\ 1 & 0},
	\quad
	c = \pmat{-1 \\ \lambda},
	\quad
	c' = \pmat{1 \\ \lambda'},
	\quad
	\lambda \ge \lambda'.
	\]
	In this case, we can write
	\[
	\gamma (\tau; x) 
	=
	\left(
	\sgn \left( \lambda x_1 - x_2 \right) - \sgn \left( \lambda' x_1 + x_2 \right)
	\right)
	q^{\abs{x_1 x_2}}. 
	\]
	Since
	\begin{align}
		\sgn \left( \lambda x_1 - x_2 \right) - \sgn \left( \lambda' x_1 + x_2 \right)
		&=
		\begin{cases}
			-2 & \text{ if } x_2 > \lambda x_1 > 0, \\
			2 & \text{ if } x_2 < \lambda x_1 < 0 \\
			-2 & \text{ if } x_2 > -\lambda' x_1 > 0, \\
			2 & \text{ if } x_2 < -\lambda' x_1 < 0, \\
			0 & \text{ otherwise}
		\end{cases}
		\\
		&=
		-2 \sgn(x_2)  
		\left( \bm{1}_{(\lambda, \infty)} \left( \frac{x_2}{x_1} \right) + \bm{1}_{(\lambda', \infty)} \left( -\frac{x_2}{x_1} \right) \right),
	\end{align}
	we can write $ \gamma(\tau; x) = \gamma_\lambda^{} (\tau; x) + \gamma_{\lambda'}^{} (\tau; -x_1, x_2) $, where 
	\[
	\gamma_\lambda^{} (\tau; x) 
	\coloneqq
	-2\sgn(x_2) \bm{1}_{(\lambda, \infty)} \left( \frac{x_2}{x_1} \right) q^{x_1 x_2}.
	\]
	We have
	\begin{align}
		\widehat{\gamma}_\lambda^{} (\tau; \xi)
		&=
		\sum_{e_1, e_2 \in \{ \pm 1 \}} \int_0^\infty \int_0^\infty
		\gamma_\lambda^{} (\tau; e_1 x_1, e_2 x_2) \bm{e} \left( -e_1 \xi_1 x_1 - e_2 \xi_2 x_2 \right) dx
		\\
		&=
		-2 \int_0^\infty \int_0^\infty q^{x_1 x_2} 
		\left(
		\bm{1}_{(\lambda, \infty)} \left( \frac{x_2}{x_1} \right)
		\left( \bm{e} (-\transpose{\xi} x) - \bm{e} (\transpose{\xi} x) \right)
		\right)
		dx
		\\
		&=
		-2 \int_0^\infty dx_1 \cdot
		\int_{\lambda x_1}^{\infty} q^{x_1 x_2} 
		\left( \bm{e} (-\transpose{\xi} x) - \bm{e} (\transpose{\xi} x) \right)
		dx_2
		\\
		&=
		-2 \int_0^\infty \frac{q^{\lambda x_1^2}}{2\pi\iu} 
		\left( 
		- \frac{\bm{e} (-(\xi_1 + \lambda \xi_2) x_1)}{\tau x_1 - \xi_2}
		+ \frac{\bm{e} ((\xi_1 + \lambda \xi_2) x_1)}{\tau x_1 + \xi_2}
		\right)
		dx_1.
	\end{align}
	By replacing $ x_1 $ by $ -x_1 $ in the second term, we obtain
	\[
	\widehat{\gamma}_\lambda^{} (\tau; \xi)
	=
	\frac{1}{\pi\iu\tau} \int_\R q^{\lambda x_1^2}
	\frac{\bm{e} (-(\xi_1 + \lambda \xi_2) x_1)}{x_1 - \xi_2 / \tau}
	dx_1.
	\]
	By replacing $ x_1 $ by $ x_1 + \xi_2 / \tau $, we have
	\[
	\widehat{\gamma}_\lambda^{} (\tau; \xi)
	=
	\frac{\widetilde{q}^{\xi_1 \xi_2}}{\pi\iu\tau} 
	\int_{\R + \iu \xi_2} q^{\lambda x_1^2}
	\bm{e} ((-\xi_1 + \lambda \xi_2) x_1)
	\frac{dx_1}{x_1}.
	\]
	By substituting $ w = \sqrt{-2\iu \tau \lambda} x_1 $, we have
	\[
	\widehat{\gamma}_\lambda^{} (\tau; \xi)
	=
	\frac{\widetilde{q}^{\xi_1 \xi_2}}{\pi\iu\tau} 
	\int_{\R + \iu \xi_2} e^{-\pi w^2}
	\bm{e} \left( \sqrt{\frac{\iu}{\tau}} \frac{-\xi_1 + \lambda \xi_2}{\sqrt{2\lambda}} w \right)
	\frac{dw}{w}.
	\]
	Here, we have
	\[
	\frac{\iu}{\pi} \int_{\R \pm \iu} e^{-\pi w^2} \bm{e} (-uw) \frac{dw}{w}
	=
	\pm 1 + \erf (\sqrt{\pi} u)
	\]
	by the residue theorem and the formula
	\[
	\erf (\sqrt{\pi} u)
	=
	\frac{\iu}{\pi} \PV \int_{\R} e^{-\pi w^2} \bm{e} (-uw) \frac{dw}{w},
	\]
	which is obtained by applying the Fourier transform.
	Thus, we obtain 
	\begin{align}
		\widehat{\gamma}_\lambda^{} (\tau; \xi) 
		&=
		\frac{1}{\tau} \widetilde{q}^{\, \xi_1 \xi_2}
		\left(
		-\sgn(\xi_2) + \erf \left( \sqrt{\frac{\pi\iu}{\tau}} \frac{-\xi_1 + \lambda \xi_2}{\sqrt{2\lambda}} \right)
		\right)
		\\
		&=
		\frac{1}{\tau} \widetilde{q}^{\, \xi_1 \xi_2}
		\left(
		-\sgn(\xi_2) + \sgn(-\xi_1 + \lambda \xi_2)
		- E^* \left( -\frac{1}{\tau}; \frac{-\xi_1 + \lambda \xi_2}{\sqrt{2\lambda}} \right)
		\right).
	\end{align}
	Since $ -\sgn(\xi_2) + \sgn(-\xi_1 + \lambda \xi_2) = 0 $ if $ \xi_1 \xi_2 < 0 $, we have
	\[
	\widehat{\gamma}_\lambda^{} (\tau; \xi)
	=
	\frac{1}{\tau}
	\left(
	-\sgn(\xi_2) + \sgn(-\xi_1 + \lambda \xi_2)
	\right) 
	\widetilde{q}^{\, \abs{\xi_1 \xi_2}}
	-
	\frac{1}{\tau}
	E^* \left( -\frac{1}{\tau}; \frac{-\xi_1 + \lambda \xi_2}{\sqrt{2\lambda}} \right)
	\widetilde{q}^{\, \xi_1 \xi_2}.
	\]
	Therefore, we obtain
	\begin{align}
		\widehat{\gamma} (\tau; \xi) 
		&= \widehat{\gamma}_\lambda^{} (\tau; \xi) + \widehat{\gamma}_{\lambda'}^{} (\tau; -\xi_1, \xi_2)
		\\
		&=
		\frac{1}{\tau}
		\left(
		\sgn(-\xi_1 + \lambda \xi_2) - \sgn(\xi_1 + \lambda' \xi_2)
		\right) 
		\widetilde{q}^{\, \abs{\xi_1 \xi_2}}
		\\
		&\phant
		-
		\frac{1}{\tau}
		E^* \left( -\frac{1}{\tau}; \frac{-\xi_1 + \lambda \xi_2}{\sqrt{2\lambda}} \right)
		\widetilde{q}^{\, \xi_1 \xi_2}
		+
		\frac{1}{\tau}
		E^* \left( -\frac{1}{\tau}; \frac{-\xi_1 - \lambda' \xi_2}{\sqrt{2\lambda}} \right)
		\widetilde{q}^{\, -\xi_1 \xi_2}.
	\end{align}
\end{proof}


\section{Quantum modularity of Eisenstein series} \label{sec:Eisenstein}



\subsection{Eisenstein series of even weight} \label{subsec:Eisenstein_even}


Usually, the modular transformation formula for Eisenstein series of even weight is proved separately for the cases $ k=2 $ and $ k \ge 4 $.
In this subsection, we provide a unified proof of both cases based on the framework of modular series.

Throughout this subsection, we fix an even integer $ k \ge 2 $.

\begin{dfn}
	We define \emph{Eisenstein series of weight $ k $} as
	\begin{align}
		G_k (\tau) &\coloneqq
		\begin{dcases}
			\sum_{(m, n) \in \Z^2 \smallsetminus \{ (0, 0) \}} \frac{1}{(m\tau + n)^k} & \text{ if } k \ge 4, \\
			\left( \sum_{m \in \Z} \sum_{n \in \Z \smallsetminus \{ 0 \}} + \sum_{m = 0} \sum_{n \in \Z} \right) \frac{1}{(m\tau + n)^2} 
			& \text{ if } k = 2,
		\end{dcases}
		\\
		E_k (\tau) &\coloneqq
		\frac{G_k (\tau)}{2 \zeta(k)},
	\end{align}
	where
	\[
	\zeta (s) \coloneqq
	\sum_{n=1}^\infty \frac{1}{n^s},
	\quad
	s \in \bbC, \, \Re(s) > 1
	\]
	is the \emph{Riemann zeta function}. 
\end{dfn}

\begin{rem} \label{rem:Eisenstein_even}
	We recall well-known facts.
	By using the partial fraction decomposition
	\begin{equation} \label{eq:cot_decomp}
		\pi \cot (\pi z) =
		\frac{1}{z} + \sum_{n=1}^\infty \left( \frac{1}{z-n} + \frac{1}{z+n} \right),
	\end{equation}
	we can prove
	\begin{equation} \label{eq:Eisenstein_even_expression}
		E_k (\tau) 
		=
		1 + \frac{(2\pi\iu)^k}{(k-1)! \zeta(k)} \sum_{d=1}^\infty \frac{d^{k-1} q^d}{1 - q^d}
		=
		1 + \frac{(2\pi\iu)^k}{(k-1)! \zeta(k)} \sum_{n=1}^\infty \sigma_{k-1} (n) q^n,
	\end{equation}
	where
	\[
	\sigma_{k-1} (n)
	\coloneqq
	\sum_{0 < d \mid n} d^{k-1}
	\]
	is the divisor function. 
	We also have
	\begin{equation} \label{eq:zeta_even_expression}
		(2\pi\iu)^{-k} (k-1)! \zeta(k)
		=
		\frac{\zeta(1-k)}{2}
		=
		-\frac{B_k}{2k},
	\end{equation}
	where $ B_k $ is the $ k $-th Bernoulli number defined as
	\[
	\sum_{m=0}^\infty B_m \frac{t^m}{m!}
	\coloneqq
	\frac{t}{e^t - 1}.
	\]
\end{rem}

It is well-known that Eisenstein series satisfy the following modular transformation formula. 

\begin{prop} \label{prop:Eisenstein_even}
	We have
	\[
	E_k (\tau)
	=
	\tau^{-k} E_k \left( -\frac{1}{\tau} \right) 
	- \frac{12}{2\pi\iu\tau} \delta_{2, k},
	\]
	where
	\[
	\delta_{m,n}
	\coloneqq
	\begin{cases}
		1 & \text{ if } m=n, \\
		0 & \text{ if } m \neq n.
	\end{cases}
	\]
	is the Kronecker delta function.
\end{prop}

Usually, this formula is proved by exchanging the double sum in the definition of $ G_k(\tau) $ (for example, see \cite[Exercise 1.2.8]{Diamond-Shurman}).
Here, we give an alternative proof using the method of modular series.

To begin with, we express Eisenstein series as modular series.
Inspired by the calculation of Eisenstein series of level $ 4 $ in \cref{ex:ker_func} \cref{item:ex:ker_func:Eisenstein_level_4}, we start from a \emph{Lambert series}
\[
\Phi_k (\tau)
\coloneqq
\sum_{d=1}^\infty d^{k-1} \frac{q^d}{1 - q^d}.
\]
One might consider taking
\[
\gamma' (\tau; x)
\coloneqq
\frac{q^x}{1 - q^x}
\]
as a kernel function. 
However, this does not work since $ \gamma' (\tau; x) $ is not continuous at $ x=0 $ and does not decay as $ y \to -\infty $ but converges to $ -1 $.
Thus, we take a function 
\[
\gamma (\tau; x)
\coloneqq
x \left( \frac{q^x}{1 - q^x} \bm{1}_{\R \smallsetminus \{ 0 \}} (x) + \bm{1}_{(-\infty, 0)} (x) \right),
\]
which corrects these obstacles, as the kernel function.
This choice is the main idea in this subsection.

This function $ \gamma (\tau; x) $ is of exponential decay.
We remark that we can also write
\begin{equation} \label{eq:kernel_func_Eisen_even}
	\gamma (\tau; x)
	=
	\abs{x} \frac{q^{\abs{x}}}{1 - q^{\abs{x}}}.
\end{equation}
Since $ x^{k-2} \in \overline{C}_1 $ is a quasi-polynomial, we have
\[
\Phi [x^{k-2}, \gamma] (\tau)
=
2\Phi_k (\tau) - \frac{\delta_{2,k}}{2\pi\iu \tau}
=
\frac{B_k}{k} \left( 1 - E_k (\tau) \right) - \frac{\delta_{2,k}}{2\pi\iu \tau}
\]
since we have $ \gamma(\tau; 0) = -1/2\pi\iu \tau $ and $ k \ge 2 $ is even.

Using this expression, we give an alternative proof of \cref{prop:Eisenstein_even}. 
The Fourier transform of $ \gamma (\tau; x) $ can be calculated as follows.

\begin{lem} \label{lem:Eisen_kernel_func_Fourier}
	We have
	\[
	\widehat{\gamma} (\tau; \xi)
	=
	(2\pi\iu \tau)^{-2} 
	\sum_{l \in \Z \smallsetminus \{ 0 \}} \frac{1}{(-\xi/\tau + l)^2}.
	\]
\end{lem}

\begin{proof}
	We can calculate as follows:
	\begin{align}
		\widehat{\gamma} (\tau; \xi)
		&=
		\int_0^\infty \gamma (\tau; x) \left( \bm{e} (-\xi x) + \bm{e} (\xi x) \right) dx
		\\
		&=
		-\frac{1}{2\pi\iu} \frac{\partial}{\partial \xi}
		\int_0^\infty \frac{q^x}{1 - q^x} \left( \bm{e} (-\xi x) - \bm{e} (\xi x) \right) dx
		\\
		&=
		-\frac{1}{2\pi\iu} \frac{\partial}{\partial \xi}
		\sum_{l=1}^\infty
		\int_{0}^{\infty} q^{lx} \left( \bm{e} (-\xi x) - \bm{e} (\xi x) \right) dx
		\\
		&=
		\frac{1}{2\pi\iu} \frac{\partial}{\partial \xi}
		\sum_{l=1}^\infty \frac{1}{2\pi\iu} \left( \frac{1}{l\tau - \xi} - \frac{1}{l\tau + \xi} \right) dx
		\\
		&=
		(2\pi\iu \tau)^{-2} 
		\sum_{l \in \Z \smallsetminus \{ 0 \}} \frac{1}{(-\xi/\tau + l)^2}.
	\end{align}
\end{proof}

\begin{rem}
	We present a version of the above formula before multiplying by $ x $.
	Let
	\[
	f(\tau; x) \coloneqq
	\sgn(x) \frac{q^{\abs{x}}}{1 - q^{\abs{x}}} + \frac{1}{2\pi\iu \tau x}.
	\]
	Then, we can calculate its Fourier transform in the sense of Cauchy principal value as
	\[
	\widehat{f} (\tau; \xi) = \frac{1}{\tau} f \left( -\frac{1}{\tau}; \xi \right).
	\]
	Here we remark that $ f(\tau; x) $ does not decay exponentially; indeed, it is not even summable on $ \Z_{>0} $.
\end{rem}

Based on the above preparations, we present an alternative proof of \cref{prop:Eisenstein_even}.

\begin{proof}[Proof of $ \cref{prop:Eisenstein_even} $]
	By \cref{prop:modular_series_modular} or the Poisson summation formula, we obtain
	\[
	\frac{B_k}{k} \left( 1 - E_k (\tau) \right) - \frac{\delta_{2,k}}{2\pi\iu \tau}
	=
	\Phi [x^{k-2}, \gamma] (\tau)
	=
	\Phi \left[1, \left( -\frac{1}{2\pi\iu} \frac{d}{d\xi} \right)^{k-2} \widehat{\gamma} \right] (\tau).
	\]
	By \cref{lem:Eisen_kernel_func_Fourier}, the right hand side equals to
	\begin{align}
		&\phant
		(2\pi\iu \tau)^{-2} \sum_{n \in \Z} 
		\eval{
			\left( -\frac{1}{2\pi\iu} \frac{d}{d\xi} \right)^{k-2}
			\sum_{l \in \Z \smallsetminus \{ 0 \}} \frac{1}{(-\xi/\tau + l)^2}
		}_{\xi = n}
		\\
		&=
		(2\pi\iu \tau)^{-k} (k-1)!
		\sum_{n \in \Z} \sum_{l \in \Z \smallsetminus \{ 0 \}} \frac{1}{(-n/\tau + l)^k}
		\\
		&=
		(2\pi\iu \tau)^{-k} (k-1)!
		\left( G_k \left( -\frac{1}{\tau} \right) - 2 \zeta(k) \tau^k \right)
		\\
		&=
		(2\pi\iu \tau)^{-k} (k-1)! \cdot 2\zeta(k)
		\left( E_k \left( -\frac{1}{\tau} \right) - \tau^k \right).
	\end{align}
	By \cref{eq:zeta_even_expression}, this is equal to
	\[
	-\frac{B_k}{k} \left( \tau^{-k} E_k \left( -\frac{1}{\tau} \right) - 1 \right).
	\]
	Thus, we obtain
	\[
	-1 + E_k (\tau) + \frac{k}{B_k} \frac{\delta_{2,k}}{2\pi\iu \tau}
	=
	\tau^{-k} E_k \left( -\frac{1}{\tau} \right) - 1,
	\]
	which implies the claim since $ B_2 = 1/6 $.
\end{proof}

\begin{rem}
	For $ k \ge 4 $, the above proof is essentially the same as a proof of the partial fraction decomposition of $ \cot(z) $ \cref{eq:cot_decomp} using the Poisson summation formula.
	We present this proof below.
	Let 
	\[
	f_\tau (y) \coloneqq \frac{1}{\tau - y} + \frac{1}{\tau + y}.
	\]
	Its Fourier transform is calculated as
	\[
	\widehat{f}_\tau (x) = -2\pi\iu q^{\abs{x}}.
	\]
	Since both of $ f_\tau (y) $ and $ \widehat{f}_\tau (x) $ are continuous $ L^2 $-function, by the Poisson summation formula, we obtain
	\[
	\sum_{m \in \Z} f_\tau (m)
	=
	\sum_{n \in \Z} \widehat{f}_\tau (n),
	\]
	which implies the partial fraction decomposition of $ \cot(z) $.
\end{rem}


\subsection{Eisenstein series of odd weight} \label{subsec:Eisenstein_odd}


Inspired by the above proof, by using modular series, we prove quantum modularity for Eisenstein series of odd weight defined as follows.
Hereafter in this subsection, we fix a positive odd integer $ k $.

\begin{dfn}
	We define an \emph{Eisenstein series of weight $ k $} as
	\[
	\calE_k (\tau) 
	\coloneqq 
	-\frac{B_k}{2k}
	+ \sum_{d=1}^\infty \frac{d^{k-1}q^d}{1 - q^d}
	=
	-\frac{1}{4} \delta_{1, k}
	+ \sum_{n=1}^\infty \sigma_{k-1} (n) q^n.
	\]
\end{dfn}

This function has quantum modularity as follows.

\begin{thm}[{Lewis--Zagier~\cite[Chapter III \S 1, Example 2 and Chapter IV \S 1]{Lewis-Zagier:period_func}, 
		Shimomura~\cite[Theorem 2.1]{Shimomura},
		Zagier~\cite{Zagier:QMF_lecture}, 
		Bettin--Conrey~\cite[Theorem 1]{Bettin-Conrey:period} and 
		Bringmann--Ono--Wagner~\cite[Theorem 1.6 (1)]{Bringmann-Ono-Wagner}}] 
	\label{thm:Eisenstein_odd}
	For any positive odd integer $ k $, the Eisenstein series of weight $ k $ is a quantum modular form in sense of \cite{Zagier:quantum}, that is, 
	the modular gap $ \calE_k (\tau) - \tau^{-k} \calE_k \left( -1/\tau \right) $ 
	extends to an holomorphic function on $ \bbC \smallsetminus \R_{\le 0} $.
\end{thm}

\begin{rem}
	We discuss previous works on \cref{thm:Eisenstein_odd}.
	Lewis--Zagier~\cite[Chapter III \S 1, Example 2 and Chapter IV \S 1]{Lewis-Zagier:period_func} and Bettin--Conrey~\cite[Theorem 1]{Bettin-Conrey:period} proved the same statement for the case when the weight $ k $ is a general complex number.
	
	In~\cite{Zagier:QMF_lecture}, Zagier states that there are four proofs of \cref{thm:Eisenstein_odd}, namely:
	\begin{enumerate}
		\item \label{item:Eisenstein_odd_proof:1}
		Sums over cones in lattices.
		
		\item \label{item:Eisenstein_odd_proof:2}
		Inverse Mellin transforms.
		
		\item \label{item:Eisenstein_odd_proof:3}
		Green forms for Maass--Eisenstein series.
		
		\item \label{item:Eisenstein_odd_proof:4}
		Taylor coefficients of the Faddeev's quantum dilogarithm.
	\end{enumerate}
	He also demonstrates the proof of \cref{item:Eisenstein_odd_proof:1} for $ k \ge 3 $, which is based on the following expression:
	\[
	\calE_k (\tau)
	- \tau^{-k} \calE_k \left( -\frac{1}{\tau} \right)
	=
	\left(
	\frac{1}{2} \sum_{m=0} \sum_{n=0}
	+ \sum_{m=0} \sum_{n>0} + \sum_{m>0} \sum_{n=0}
	+ 2 \sum_{(m, n) \in \Z_{> 0}^2}
	\right)
	\frac{1}{(m\tau + n)^k}.
	\]
	Shimomura~\cite[Theorem 2.1]{Shimomura} also found this expression.
	Bettin--Conrey~\cite[Theorem 1]{Bettin-Conrey:period} and Bringmann--Ono--Wagner~\cite[Theorem 1.6 (1)]{Bringmann-Ono-Wagner} gave a representation of the modular gap $ \calE_k (\tau) - \tau^{-k} \calE_k \left( -1/\tau \right) $ as an inverse Mellin transform.
	Their proof aligns with method \cref{item:Eisenstein_odd_proof:2}.
\end{rem}

As an application of the Poisson summation formula with signature (\cref{thm:PSF_sgn}), we present an alternative proof of \cref{thm:Eisenstein_odd}.
Our proof is new and is different from any of the four proofs \cref{item:Eisenstein_odd_proof:1,item:Eisenstein_odd_proof:2,item:Eisenstein_odd_proof:3,item:Eisenstein_odd_proof:4}.

Our proof splits into two cases: $ k \ge 3 $ and $ k=1 $. 
We begin with the former.

\begin{proof}[\cref{thm:Eisenstein_odd} for $ k \ge 3 $]
	As in the previous subsection, we choose a kernel function as
	\[
	\gamma (\tau; x)
	\coloneqq
	\abs{x} \frac{q^{\abs{x}}}{1 - q^{\abs{x}}}.
	\]
	In this case, we have $ \Phi[\sgn(x) x^{k-2}, \gamma] (\tau) = 2\calE_k (\tau) $.
	By the Poisson summation formula with signature (\cref{thm:PSF_sgn}) or the modular transformation formula for false modular series (\cref{thm:false_modular_series_modular}), we have
	\[
	2 \calE_k (\tau)
	=
	\sum_{n \in \Z} \sgn(n)
	\left( -\frac{1}{2\pi\iu} \frac{d}{d\xi} \right)^{k-2} \widehat{\gamma} \left( \tau; n \right)
	+ 2 \int_{C_-} 
	\left( -\frac{1}{2\pi\iu} \frac{d}{d\xi} \right)^{k-2} \widehat{\gamma} \left( \tau; \xi \right) 
	\frac{d\xi}{1 - \bm{e}(\xi)}.
	\]
	
	To calculate each term, we need to calculate derivatives of $ \widehat{\gamma} \left( \tau; \xi \right) $.
	By \cref{lem:Eisen_kernel_func_Fourier}, we have
	\begin{align}
		\widehat{\gamma} \left( \tau; \xi \right) 
		&=
		(2\pi\iu \tau)^{-2} 
		\eval{
			\frac{d}{dz} 
			\left( \frac{1}{z} - \pi \cot (\pi z) \right)
		}_{z = -\xi/\tau}
		=
		(2\pi\iu \tau)^{-2} 
		\eval{
			\frac{d}{dz} 
			\left( \frac{1}{z} + \pi\iu + 2\pi\iu \frac{\bm{e}(z)}{1 - \bm{e}(z)} \right)
		}_{z = -\xi/\tau}
		\\
		&=
		(2\pi\iu \tau)^{-2} 
		\eval{
			\frac{d}{dz} 
			\left( \frac{1}{z} + 2\pi\iu \sum_{d=1}^\infty \bm{e}(dz) \right)
		}_{z = -\xi/\tau}.
	\end{align}
	Thus, we have
	\begin{align}
		\left( -\frac{1}{2\pi\iu} \frac{d}{d\xi} \right)^{k-2} \widehat{\gamma} \left( \tau; \xi \right) 
		&=
		(2\pi\iu \tau)^{-k} 
		\eval{
			\frac{d^{k-1}}{dz^{k-1}} 
			\left( \frac{1}{z} + 2\pi\iu \sum_{d=1}^\infty \bm{e}(dz) \right)
		}_{z = -\xi/\tau}
		\\
		&=
		(2\pi\iu)^{-k} (k-1)! \frac{1}{\xi^k} + \tau^{-k} \sum_{d=1}^\infty d^{k-1} \widetilde{q}^{\, d\xi},
		\label{eq:Eisenstein_odd_proof}
	\end{align}
	where $ \widetilde{q} \coloneqq \bm{e}(-1/\tau) $.
	Here we remark that
	\[
	\left( -\frac{1}{2\pi\iu} \frac{d}{d\xi} \right)^{k-2} \widehat{\gamma} \left( \tau; \xi \right) 
	=
	(2\pi\iu \tau)^{-2} (k-1)!
	\sum_{l \in \Z \smallsetminus \{ 0 \}} \frac{1}{(-\xi/\tau + l)^k}.
	\]
	also holds, although not used in the proof.
	
	Since $ \gamma(\tau; x) $ is an even function in $ x $, its Fourier transform $ \widehat{\gamma}(\tau; \xi) $ is even in $ \xi $.
	Thus, we have
	\[
	\sum_{n \in \Z} \sgn(n)
	\left( -\frac{1}{2\pi\iu} \frac{d}{d\xi} \right)^{k-2} \widehat{\gamma} \left( \tau; n \right)
	=
	2 \sum_{n=1}^\infty
	\left( -\frac{1}{2\pi\iu} \frac{d}{d\xi} \right)^{k-2} \widehat{\gamma} \left( \tau; n \right).
	\]
	By \cref{eq:Eisenstein_odd_proof}, we have
	\begin{align}
		\sum_{n \in \Z} \sgn(n)
		\left( -\frac{1}{2\pi\iu} \frac{d}{d\xi} \right)^{k-2} \widehat{\gamma} \left( \tau; n \right)
		&=
		2 \sum_{n=1}^\infty
		\left( 
		(2\pi\iu)^{-k} (k-1)! \frac{1}{n^k} 
		+ \tau^{-k} \sum_{d=1}^\infty d^{k-1} \widetilde{q}^{\, dn}
		\right)
		\\
		&=
		2(2\pi\iu)^{-k} (k-1)! \zeta(k)
		+ 2 \tau^{-k} \calE_k \left( -\frac{1}{\tau} \right).
	\end{align}
	
	Let $ C'_{\varepsilon} $ be an integration path shown in \cref{fig:int_path_EIsenstein_odd}.
	Then, the contour $ C_- $ can be deformed into 
	$ C'_{\varepsilon} $, $ \R_{\ge 0} - \iu \varepsilon $, and $ \R_{\le 0} + \iu \varepsilon $.
	
	\begin{figure}[htbp]
		\centering
		\begin{tikzpicture}
			\tikzset{midarrow/.style={postaction={decorate},
					decoration={markings,
						mark=at position 0.5 with {\arrow{stealth}},
					}
				}
			}
			\draw[->] (-3, 0) -- (3, 0) node[right]{$ \Re (x_i) $};
			\draw[->] (0, -2) -- (0, 2) node[above]{$ \Im (x_i) $};
			
			\draw (0,0) node[below left]{0};
			
			
			\draw[very thick, midarrow] (0,-0.5)--(0,-1.5) node[left]{$ -\iu \varepsilon $};
			
			\draw[very thick, midarrow] (0,0.5) arc (90:270:0.5);
			
			\draw[very thick, midarrow] (0,1.5) node[left]{$ \iu \varepsilon $}--(0,0.5);
			
			\node at (0.5,-0.8) {$ C'_{\varepsilon} $};
			
		\end{tikzpicture}
		\caption{The integration path $ C'_{\varepsilon} $}
		\label{fig:int_path_EIsenstein_odd}
	\end{figure}
	
	Since $ \widehat{\gamma}(\tau; \xi) $ is even in $ \xi $, we have
	\begin{align}
		&\phant \int_{C_-} 
		\left( -\frac{1}{2\pi\iu} \frac{d}{d\xi} \right)^{k-2} \widehat{\gamma} \left( \tau; \xi \right) 
		\frac{d\xi}{1 - \bm{e}(\xi)}
		\\
		&=\int_{C'_\varepsilon} 
		\left( -\frac{1}{2\pi\iu} \frac{d}{d\xi} \right)^{k-2} \widehat{\gamma} \left( \tau; \xi \right) 
		\frac{d\xi}{1 - \bm{e}(\xi)}
		+ \int_{\R_{\ge 0} - \iu \varepsilon} 
		\left( -\frac{1}{2\pi\iu} \frac{d}{d\xi} \right)^{k-2} \widehat{\gamma} \left( \tau; \xi \right) 
		\frac{1 + \bm{e}(\xi)}{1 - \bm{e}(\xi)} d\xi.
	\end{align}
	The first integral extends to a holomorphic function on $ \bbC \smallsetminus \R_{\le 0} $.
	By \cref{eq:Eisenstein_odd_proof}, the second integral can be transformed as 
	\[
	(2\pi\iu)^{-k} (k-1)! 
	\int_{\R_{\ge 0} - \iu \varepsilon} \frac{1}{\xi^k}  \frac{d\xi}{1 - \bm{e}(\xi)}
	+ \tau^{-k} \int_{\R_{\ge 0} - \iu \varepsilon} 
	\left(
	\sum_{d=1}^\infty d^{k-1} \widetilde{q}^{\, d\xi}
	\right)
	\frac{d\xi}{1 - \bm{e}(\xi)}.
	\]
	In this equation, the first integral converges and independent of $ \tau $.
	The second integral converges when $ \Im(-\xi/\tau) > 0 $, in particular $ \tau \in \R_{>0} $.
	By deforming the contour to satisfy $ \Im(-\xi/\tau) > 0 $, the second integral extends to a holomorphic function on $ \bbC \smallsetminus \R_{\le 0} $.
	Thus, we obtain the claim.
\end{proof}


\subsection{Eisenstein series of weight one} \label{subsec:Eisenstein_wt1}


Next, we prove \cref{thm:Eisenstein_odd} for $ k=1 $.
Let
\[
\gamma(\tau; x) \coloneqq \sgn(x) \frac{q^{\abs{x}}}{1 - q^{\abs{x}}}.
\]
Then, we have
\[
\calE_1 (\tau) = -\frac{1}{4} + \frac{1}{2} \sum_{m \in \Z \smallsetminus \{ 0 \}} \sgn(m) \gamma(\tau; m).
\]
However, this $ \gamma(\tau; x) $ has a singularity at $ x=0 $, namely,
\[
\gamma(\tau; x) = -\frac{1}{2\pi\iu \tau x} - \frac{1}{2} \sgn(x) + O(x) 
\quad \text{ as } x \to 0.
\]
To remove this singularity and make it continuous, we use the following correction terms:
\[
\gamma_* (\tau; x) \coloneqq \frac{q^{\abs{x}}}{2\pi\iu \tau x}, \quad
\gamma_{**} (\tau; x) \coloneqq - \frac{1}{2} \sgn(x) q^{\abs{x}}.
\]
We define $ \gamma_{\mathrm{reg}} (\tau; x) \coloneqq \gamma (\tau; x) + \gamma_* (\tau; x) + \gamma_{**} (\tau; x) $.
This function is continuous and of exponential decay.
We prove \cref{thm:Eisenstein_odd} for $ k=1 $ by using this function as a kernel function.

In our proof, we need the following lemma.

\begin{lem} \label{lem:Fourier_Eisenstein_base}
	We have
	\begin{alignat}{2}
		\calF \left[ q^{\abs{x}} \right]
		&=
		-\frac{1}{2\pi\iu} \left( \frac{1}{\tau + \xi} + \frac{1}{\tau - \xi} \right), & &
		\\
		\calF \left[ \sgn(x) q^{\abs{x}} \right]
		&=
		\frac{1}{2\pi\iu} \left( \frac{1}{\tau + \xi} - \frac{1}{\tau - \xi} \right), & &
		\\
		\calF \left[ \frac{q^{\abs{x}}}{x} \right]
		&=
		\log \frac{\tau + \xi}{\tau - \xi}
		& &=
		\sgn(\xi) \log \frac{\abs{\xi} + \tau}{\abs{\xi} - \tau} - \pi\iu \sgn(\xi),
		\\
		\calF \left[ \sgn(x) \frac{q^{\abs{x}}}{1 - q^{\abs{x}}} \right]
		&=
		-\frac{1}{2\iu \tau} \cot \left( -\frac{\xi}{\tau} \right) - \frac{1}{2\pi\iu \xi}
		& &=
		\frac{1}{\tau} \sgn(\xi) \frac{\widetilde{q}^{\abs{\xi}}}{1 - \widetilde{q}^{\abs{\xi}}} - \frac{1}{2 \pi \iu \xi} + \frac{1}{2} \sgn(\xi).
	\end{alignat}
	Here, we take Fourier transforms in the sense of Cauchy principal value and take $ \log $ to be the principal branch.
\end{lem}

We omit a proof of this lemma.

To calculate a sum of $ \widehat{\gamma}_{\mathrm{reg}} (\tau; \xi) $, we need the following lemma.

\begin{lem} \label{lem:sum_gamma_digamma}
	For $ z \in \bbC \smallsetminus \Z_{\le -1} $, we have
	\begin{align}
		\sum_{n=1}^\infty \left( \log \left( 1 + \frac{z}{n} \right) - \frac{z}{n} \right)
		&=
		-\log \Gamma(1+z) - \gamma z,
		\\
		\sum_{n=1}^\infty \left( \frac{1}{z+n} - \frac{1}{n} \right)
		&=
		-\psi(1+z) - \gamma,
	\end{align}
	where $ \psi(s) \coloneqq \Gamma'(s) / \Gamma(s) $ is the digamma function and
	$ \gamma = 0.577\cdots $  is the Euler--Mascheroni constant.
\end{lem}

\begin{proof}
	The first equality follows from the Weierstrass product representation
	\[
	\frac{1}{\Gamma(z)}
	=
	e^{\gamma z} \prod_{n=1}^\infty \left( 1 + \frac{z}{n} \right) e^{-z/n}.
	\]
	The second equality follows from the first by termwise differentiation.
\end{proof}

We are now ready to prove \cref{thm:Eisenstein_odd} for $ k=1 $.

\begin{proof}[Proof of \cref{thm:Eisenstein_odd} for $ k=1 $]
	By the Poisson summation formula with signature (\cref{thm:PSF_sgn}), we have
	\[
	\sum_{m \in \Z} \sgn(m) \gamma_{\mathrm{reg}} (\tau; m)
	=
	\sum_{n \in \Z} \sgn(n) \widehat{\gamma}_{\mathrm{reg}} (\tau; n)
	+ \int_{C_-} \widehat{\gamma}_{\mathrm{reg}} \left( \tau; \xi \right) \frac{d\xi}{1 - \bm{e}(\xi)}.
	\]
	We have
	\[
	\sum_{m \in \Z} \sgn(m) \gamma_{\mathrm{reg}} (\tau; m)
	=
	\frac{1}{2} + 2 \calE_1 (\tau) + \frac{1}{\pi\iu \tau} \log (1-q) - \frac{q}{1-q}.
	\]
	On the other hand, by \cref{lem:Fourier_Eisenstein_base}, we have
	\[
	\widehat{\gamma}_{\mathrm{reg}} (\tau; \xi)
	=
	\frac{1}{\tau} \gamma \left( -\frac{1}{\tau}; \xi \right)
	+ \frac{\sgn(\xi)}{2\pi\iu} \left( 
		- \frac{1}{\abs{\xi}} 
		+ \frac{1}{\tau} \log \frac{\abs{\xi} + \tau}{\abs{\xi} - \tau}  
		- \frac{1}{2} \left( \frac{1}{\tau + \abs{\xi}} - \frac{1}{\tau - \abs{\xi}} \right)
	\right).
	\]
	In particular, we have $ \widehat{\gamma}_{\mathrm{reg}} (\tau; \xi) = -\sgn(\xi) / 2\tau $.
	Thus, by combining with \cref{lem:sum_gamma_digamma}, we have
	\begin{align}
		&\phant
		\sum_{n \in \Z} \sgn(n) \widehat{\gamma}_{\mathrm{reg}} (\tau; n)
		\\
		&=
		-\frac{1}{2\tau} + \frac{1}{\tau} \left( \frac{1}{2} + 2 \calE_1 \left( -\frac{1}{\tau} \right) \right)
		- \frac{1}{\pi\iu \tau} \log \frac{\Gamma(1 + \tau)}{\Gamma(1 - \tau)}
		+ \frac{1}{2\pi\iu} \left( \psi(1 + \tau) + \psi(1 - \tau) + 2\gamma \right).
	\end{align}
	By the reflection formula for the gamma functions, we have
	\begin{alignat}{2}
		\log \frac{\Gamma(1 + \tau)}{\Gamma(1 - \tau)}
		&=
		\log \left( \tau \Gamma(\tau)^2 \frac{\sin \pi \tau}{\pi} \right)
		& &=
		2 \log \Gamma(\tau) + \log \tau + \log (1-q) - \pi\iu \tau - \log 2\pi + \frac{\pi\iu}{2},
		\\
		\psi(1 + \tau) + \psi(1 - \tau)
		&=
		2\psi(\tau) + \frac{1}{\tau} + \pi \cot \pi\tau
		& &=
		2\psi(\tau) + \frac{1}{\tau} - \pi\iu - 2\pi\iu \frac{q}{1 - q}.
	\end{alignat}
	Thus, we have
	\begin{align}
		&\phant
		\sum_{n \in \Z} \sgn(n) \widehat{\gamma}_{\mathrm{reg}} (\tau; n)
		\\
		&=
		\frac{2}{\tau} \calE_1 \left( -\frac{1}{\tau} \right)
		- \frac{1}{\pi\iu \tau} 
		\left( 
			-2 \log \Gamma (\tau) + \tau \psi(\tau) 
			- \log (1-q) - \log \tau 
			+ \log 2\pi + 2
			+ \gamma \tau
		\right)
		+ \frac{1}{2} - \frac{1}{2\tau} - \frac{q}{1-q}.
	\end{align}
	Therefore, we obtain
	\begin{align}
		2 \calE_1 (\tau)
		&=
		\frac{2}{\tau} \calE_1 \left( -\frac{1}{\tau} \right)
		+ \frac{1}{\pi\iu \tau} 
		\left( 
		-2 \log \Gamma (\tau) + \tau \psi(\tau) 
		- 2\log (1-q) - \log \tau 
		+ \log 2\pi + 2
		+ \gamma \tau
		\right)
		- \frac{1}{2\tau}
		\\
		&\phant
		+ \int_{C_-} \widehat{\gamma}_{\mathrm{reg}} \left( \tau; \xi \right) \frac{d\xi}{1 - \bm{e}(\xi)}.
	\end{align}
	By the same argument in the proof for $ k \ge 3 $, the integral term extends to a holomorphic function on $ \bbC \smallsetminus \R_{\le 0} $.
	Thus, we obtain the claim.
\end{proof}


\part{Application to Witten's asymptotic expansion conjecture} \label{part:WRT}


In this part, we prove our asymptotic formula (\cref{thm:main_WRT_asymptotic}) for WRT invariants as an application of results in \cref{sec:false_theta}.


\section{Preliminaries} \label{sec:preliminaries}


In this section, we provide some notation and basic facts, which we use throughout this paper.


\subsection{Plumbed manifolds} \label{subsec:plumbed_manifold}



A \emph{plumbing graph} $ \Gamma $ is a tree with integer weights on the vertices.
For a plumbing graph $ \Gamma $, we obtain a surgery diagram $ \calL(\Gamma) $ whose all components are trivial knots as in \cref{fig:surgery_diagram}.
Let $ M(\Gamma) $ be the plumed manifold obtained from $ S^3 $ through the surgery along the diagram $ \calL(\Gamma) $.
Such $ 3 $-manifolds are called \emph{plumbed manifolds}.
According to Neumann's theorem (\cite[Proposition 2.2]{Neumann_Lecture}, \cite[Theorem 3.1]{Neumann_work}), two plumbing graphs define homeomorphic $ 3 $-manifolds if and only if they are related by Neumann moves shown in \cref{fig:Neumann}.
Lens spaces and Seifert homology spheres are examples of plumbed manifolds.

For a plumbing graph $ \Gamma $, let $ M \coloneqq M(\Gamma) $ and let $ W $ be the linking matrix of $ \Gamma $.
Then, we have an isomorphism $ H_1(M, \Z) \cong \Z^{V} / W(\Z^{V}) $ which is compatible with the linking form $ H_1(M, \Z) \times H_1(M, \Z) \to \Z $ and $ \Z^{V} / W(\Z^{V}) \times \Z^{V} / W(\Z^{V}) \to \Z $; $ (m, n) \mapsto \transpose{m} W^{-1} n $. 
We also have $ H^1(M, \Z) \cong \Hom_\Z (H_1(M, \Z), \Z) \cong W^{-1}(\Z^V) / \Z^V $ by the universal coefficient theorem.

For a vertex $ v \in V $, define its degree as $ \deg(v) \coloneqq \# \left\{ v' \in V \relmiddle{|} \text{$ v' $ is adjacent to $ v $} \right\} $ 
Let $ \delta \coloneqq (\deg(v))_{v \in V} \in \Z^V $.
Then, the affine space $ \Spin^c(M) $ over $ H_1(M, \Z) $ of $ \Spin^c $ structures are isomorphic to $ \left( \delta + 2\Z^{V} \right) / 2W(\Z^{V}) $ as stated in \cite[Equation (34)]{Gukov-Manolescu}.
Moreover, through this bijection, the action by conjugation of $ \Spin^c $ structures commutes with the action by multiplication by 
$ -1 $ on $ \left( \delta + 2\Z^{V} \right) / 2W(\Z^{V}) $.

\begin{figure}[htb]
	\begin{minipage}[b]{0.45\linewidth}
		\centering
		\begin{tikzpicture}
			\node[shape=circle,fill=black, scale = 0.4] (1) at (0,0) { };
			\node[shape=circle,fill=black, scale = 0.4] (2) at (1.5,0) { };
			\node[shape=circle,fill=black, scale = 0.4] (3) at (-1,-1) { };
			\node[shape=circle,fill=black, scale = 0.4] (4) at (-1,1) { };
			\node[shape=circle,fill=black, scale = 0.4] (5) at (2.5,1) { };
			\node[shape=circle,fill=black, scale = 0.4] (6) at (2.5,-1) { };
			
			\node[draw=none] (B1) at (0,0.4) {$ w_1 $};
			\node[draw=none] (B2) at (1.5, 0.4) {$ w_2 $};
			\node[draw=none] (B3) at (-0.6,1) {$ w_3 $};
			\node[draw=none] (B4) at (-0.6,-1) {$ w_4 $};
			\node[draw=none] (B5) at (2.1,1) {$ w_5 $};		
			\node[draw=none] (B6) at (2.1,-1) {$ w_6 $};	
			
			\path [-](1) edge node[left] {} (2);
			\path [-](1) edge node[left] {} (3);
			\path [-](1) edge node[left] {} (4);
			\path [-](2) edge node[left] {} (5);
			\path [-](2) edge node[left] {} (6);
		\end{tikzpicture}
		\caption{An example of a plumbing graph $ \Gamma $} \label{fig:H-graph}		
	\end{minipage}
	\begin{minipage}[b]{0.45\linewidth}
		\centering
		\begin{tikzpicture}[scale=0.5]
			\begin{knot}[
				clip width=5, 
				flip crossing=2, 
				flip crossing=4, 
				flip crossing=6, 
				flip crossing=7, 
				flip crossing=9
				]
				\strand[thick] (0, 0) circle [x radius=3cm, y radius=1.5cm];
				\strand[thick] (0, 0) +(4cm, 0pt) circle [x radius=3cm, y radius=1.5cm];
				\strand[thick] (0, 0) +(-2.5cm, 2cm) circle [radius=1.5cm];
				\strand[thick] (0, 0) +(-2.5cm, -2cm) circle [radius=1.5cm];
				\strand[thick] (0, 0) +(6.5cm, 2cm) circle [radius=1.5cm];
				\strand[thick] (0, 0) +(6.5cm, -2cm) circle [radius=1.5cm];
				
				\node (1) at (0, 2) {$w_{1}$};
				\node (2) at (4, 2) {$w_{2}$};
				\node (3) at (-2.5, 4) {$w_{3}$};
				\node (4) at (-2.5, -4) {$w_{4}$};
				\node (5) at (6.5, 4) {$w_{5}$};
				\node (6) at (6.5, -4) {$w_{6}$};
			\end{knot}
		\end{tikzpicture}
		\caption{The surgery diagram $ \calL(\Gamma) $ corresponding to $\Gamma$ in \cref{fig:H-graph}} \label{fig:surgery_diagram}
	\end{minipage}
\end{figure}

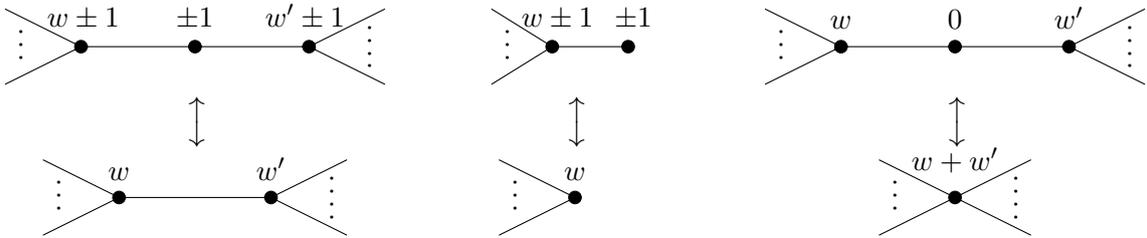
\begin{figure}[htp]
	\centering
	\begin{tikzpicture}
		\draw[fill]
		(-1.5,0) node[above=0.1cm, xshift=0.03cm]{$w \pm 1$} circle(0.5ex)--
		(0,0) node[above=0.1cm]{$\pm 1$} circle(0.5ex)--
		(1.5,0) node[above=0.1cm, xshift=-0.05cm]{$w' \pm 1$} circle(0.5ex)
		(-2.5,0.5) node[above]{}--(-1.5,0) node[above]{}
		(-2.3,0) node[rotate=270]{$\cdots$}
		(-2.5,-0.5) node[above]{}--(-1.5,-0) node[above]{}
		(1.5,0) node[above]{}--(2.5,0.5) node[above]{}
		(2.3,0) node[rotate=270]{$\ldotp\ldotp\ldotp\ldotp$}
		(1.5,0) node[above]{}--(2.5,-0.5) node[above]{}
		(0,-1) node[rotate=270]{$\longleftrightarrow$}
		(-1,-2) node[above=0.1cm]{$ w $} circle(0.5ex)--
		(1,-2) node[above=0.1cm]{$ w' $} circle(0.5ex)
		(-2,-1.5) node[above]{}--(-1,-2) node[above]{}
		(-1.8,-2) node[rotate=270]{$\cdots$}
		(-2,-2.5) node[above]{}--(-1,-2) node[above]{}
		(1,-2) node[above]{}--(2,-1.5) node[above]{}
		(1.8,-2) node[rotate=270]{$\ldotp\ldotp\ldotp\ldotp$}
		(1,-2) node[above]{}--(2,-2.5) node[above]{};
		
		\draw[fill]
		(4.7,0) node[above=0.1cm, xshift=0.07cm]{$w \pm 1$} circle(0.5ex)--
		(5.7,0) node[above=0.1cm, xshift=0.07cm]{$\pm 1$} circle(0.5ex)
		(3.9,0.5) node[above]{}--(4.7,0) node[above]{}
		(4.1,0) node[rotate=270]{$\cdots$}
		(3.9,-0.5) node[above]{}--(4.7,0) node[above]{}
		(5,-1) node[rotate=270]{$\longleftrightarrow$}
		(5,-2) node[above=0.1cm]{$w$} circle(0.5ex)
		(4,-1.5) node[above]{}--(5,-2) node[above]{}
		(4.2,-2) node[rotate=270]{$\cdots$}
		(4,-2.5) node[above]{}--(5,-2) node[above]{};
		
		\draw[fill]
		(8.5,0) node[above=0.1cm]{$w$} circle(0.5ex)--
		(10,0) node[above=0.1cm]{$0$} circle(0.5ex)--
		(11.5,0) node[above=0.1cm]{$w'$} circle(0.5ex)
		(7.5,0.5) node[above]{}--(8.5,0) node[above]{}
		(7.7,0) node[rotate=270]{$\cdots$}
		(7.5,-0.5) node[above]{}--(8.5,0) node[above]{}
		(11.5,0) node[above]{}--(12.5,0.5) node[above]{}
		(12.3,0) node[rotate=270]{$\ldotp\ldotp\ldotp\ldotp$}
		(11.5,0) node[above]{}--(12.5,-0.5) node[above]{}
		(10,-1) node[rotate=270]{$\longleftrightarrow$}
		(10,-2) node[above=0.2cm]{$w + w'$} circle(0.5ex)
		(9,-1.5) node[above]{}--(10,-2) node[above]{}
		(9.2,-2) node[rotate=270]{$\cdots$}
		(9,-2.5) node[above]{}--(10,-2) node[above]{}
		(10,-2) node[above]{}--(11,-1.5) node[above]{}
		(10.8,-2) node[rotate=270]{$\ldotp\ldotp\ldotp\ldotp$}
		(10,-2) node[above]{}--(11,-2.5) node[above]{};
	\end{tikzpicture}
	\caption{Neumann moves}
	\label{fig:Neumann}
\end{figure}


\subsection{The Witten--Reshetikhin--Turaev invariants} \label{subsec:WRT}


For a closed and oriented $ 3 $-manifold $ M $ and a positive integer $ k $, let $ Z_k(M) \in \bbC $ denote the $ \SU(2) $ Witten--Reshetikhin--Turaev (WRT) invariant for $ M $ at level $ k $.
We adopt the normalization $ Z_k (S^3) = 1 $.
For a plumbed manifold $ M $, its WRT invariant $ Z_k(M) $ is expressed as follows.

\begin{prop}[{Gukov--Pei--Putrov--Vafa~\cite[Equation A.12]{GPPV}}]
	\label{prop:WRT_rep}
	For plumbed manifolds $ M = M(\Gamma) $ defined by a plumbing graph $ \Gamma $, we have
	\[
	Z_k (M)
	=
	\frac{\zeta_8^{-\sgn(W)} \zeta_{4k}^{-\sum_{v \in V} (w_{v} + 3)}}
	{2 \sqrt{2k}^{\abs{V}} \left( \zeta_{2k} - \zeta_{2k}^{-1} \right)}
	\sum_{\mu \in (\Z \smallsetminus k\Z)^V/2k\Z^V}
	\bm{e} \left( \frac{1}{4k} {}^t\!\mu W \mu \right)
	\prod_{v \in V} F_v \left( \zeta_{2k}^{\mu_v} \right),
	\]
	where 
	\begin{align}
		\sgn(W) &\coloneqq \# \{ \text{positive eigenvalues of $ W $} \} - \# \{ \text{negative eigenvalues of $ W $} \},
		\\
		F_v (z_v) &\coloneqq \left( z_v - z_v^{-1} \right)^{2 - \deg(v)}.
	\end{align}
\end{prop}


\subsection{The GPPV invariants} \label{subsec:HB_false_main}


A plumbed manifold is \emph{negative definite} if it is defined by a plumbing graph whose linking matrix is negative definite.
We remark that a negative definite plumbed manifold is a rational homology sphere.
For a negative definite plumbed manifold, one can define its \emph{Gukov--Pei--Putrov--Vafa} (\emph{GPPV}) \emph{invariant} as follows.

\begin{dfn}[{\cite[Equation (A.29)]{GPPV}}] \label{dfn:GPPV_inv}
	The \emph{Gukov--Pei--Putrov--Vafa invariant} for a negative definite plumbed manifold $ M = M(\Gamma) $ and its $ \Spin^c $ structure $ b \in \Spin^c (M) \cong \left( \delta + 2\Z^{V} \right) / 2W(\Z^{V}) $ is defined as
	\[
	\widehat{Z}_{b} (q; M)
	\coloneqq 
	q^{\Delta} 
	\sum_{l \in 2W(\Z^V) + b} F_l q^{-\transpose{l} W^{-1} l/4},
	\]
	where 
	\[
	\Delta \coloneqq -\frac{3 \abs{V} + \tr W}{4},
	\quad
	F_l \coloneqq \prod_{v \in V} F_{v, l_v}, 
	\quad
	F_{v, l_v} \coloneqq \PV \int_{\abs{z_v} = 1} F_v (z_v) \frac{z_v^{l_v} dz_v}{2\pi\iu z_v},
	\]
	and the Cauchy principal value
	\[
	\PV
	\coloneqq 
	\frac{1}{2}
	\lim_{\veps \to +0} \left( \int_{\abs{z} = 1+\veps} + \int_{\abs{z} = 1-\veps} \right).
	\]
\end{dfn}

The coefficient $ F_l $ of GPPV invariants satisfies the following properties.

\begin{lem} \label{lem:coeff_of_F}
	Let $ l = (l_v)_{v \in V} \in \Z^V $ and $ v \in V $.
	\begin{enumerate}
		\item \label{item:lem:coeff_of_F:expression}
		\textup{(\cite[p. 743]{Andersen-Mistegard})} We have
		\begin{equation}
			F_{v, l_v} =
			\begin{dcases}
				-l_v, & \text{ if } \deg (v) = 1, \, l_v \in \{ \pm 1 \}, \\
				1, & \text{ if } \deg (v) = 2, \, l_v =0, \\
				\frac{\sgn(l_v)^{\deg (v)}}{2} \binom{m + \deg (v) - 3}{\deg (v) - 3}, & \text{ if } \deg (v) \ge 3, \, \pm l_v = \deg (v) -2 + 2m \text{ for some } m \in \Z_{\ge 0}, \\
				0, & \text{ otherwise}.
			\end{dcases}
		\end{equation}
		
		\item \label{item:lem:coeff_of_F:symmetry}
		\textup{(\cite[Lemma 2.3 (ii)]{M:GPPV})} 
		We have $ F_{v, l_v} = (-1)^{\deg (v)} F_{v, l_v} $ and $ F_{-l} = F_l $.
		
		\item \label{item:lem:coeff_of_F:expansion}
		\textup{(\cite[Lemma 2.3 (iii)]{M:GPPV})} 
		We have
		\[
		\sum_{l_v \in \deg (v) + 2 \Z_{\ge -1}} F_{v, l_v} z_v^{l_v}
		=
		(-1)^{\deg (v)} F_v(z_v) \cdot
		\begin{dcases}
			1, \text{ if } \deg (v) \le 2, \\
			\frac{1}{2}, \text{ if } \deg (v) \ge 3.
		\end{dcases}
		\]
	\end{enumerate}
\end{lem}

We have $ F_V^{} (z_V^{}) \coloneqq \prod_{v \in V} F_v (z_v) \in \frakR_V^{} $ and $ F_l = F_{-l} = \widetilde{\coe}[F_v^{}] (l) $ by definition.
Thus, the GPPV invariant can be regarded as a false modular series.

\begin{rem} \label{rem:coeff_of_F_modular_series}
	We can prove \cref{lem:coeff_of_F} \cref{item:lem:coeff_of_F:expansion} by using a framework of cyclotomic rational functions and quasi-polynomials.
	Indeed, by \cref{lem:R_C_corresp} \cref{item:lem:R_C_corresp:vp,item:lem:R_C_corresp:involution_bilateral}, we have
	\[
	F_{v, l_v} = \frac{1}{2} \sum_{e_v \in \{ \pm 1 \}} \coe[F_v] (e_v l_v)
	= \frac{1}{2} \sum_{e_v \in \{ \pm 1 \}} \coe[F_v (z_v^{e_v})] (l_v).
	\]
	Since $ F_v (z_v^{-1}) = (-1)^{\deg (v)} F_v (z_v) $, we have
	\[
	F_l = \frac{1}{2} \sum_{e_v \in \{ \pm 1 \}} e_v^{\deg (v)} \coe[F_v] (l_v).
	\]
	On the other hand, by \cref{rem:coe_Laurent}, we have
	\[
	F_v (z_v)
	=
	\sum_{l_v \in \Z} \coe[F_v](l_v) z_v^{l_v}
	=
	\sum_{l_v \in \deg (v) + 2\Z_{\ge -1}} \coe[F_v](l_v) z_v^{l_v}
	\]
	for $ \abs{z_v} < 1 $.
	Thus, we obtain \cref{lem:coeff_of_F} \cref{item:lem:coeff_of_F:expansion}.
\end{rem}


\subsection{Radial limits of GPPV invariants} \label{subsec:GPPV_radial_limits}


The \emph{radial limit conjecture} asks whether the WRT invariants can be expressed as radial limits of certain $q$-series.
Here, a radial limit refers to the limit along $ q \to \zeta_k $ for a positive integer $ k $, where we set $ \zeta_k \coloneqq \bm{e}(1/k) $.
This conjecture plays an important role in resolving the asymptotic expansion conjecture for WRT invariants and in their categorification.
It can be formulated as follows.

\begin{conj}[The radial limit conjecture, {Hikami~\cite[Equation (1.4)]{Hikami:radial_limit}, Gukov--Pei--Putrov--\linebreak[0]Vafa~\cite[Conjecture 2.1]{GPPV}, Gukov--Manolescu~\cite[Conjecture 3.1]{Gukov-Manolescu}}]
	\label{conj:radial_limit_conj}
	Let $ M $ be a rational homology sphere.
	Then, for each $ \Spin^c $ structure $ b \in \Spin^c(M) $, there exist invariants
	\[
	\Delta_b \in \Q, \quad
	c \in \Z_{\ge 0}, \quad
	\widehat{Z}_b (q; M) \in 2^{-c} q^{\Delta_b} \Z[[q]]
	\]
	that are conjugation-invariant, 
	$ \widehat{Z}_b (q; M) $ conveges on $ \abs{q} < 1 $, 
	and for infinitely many $ k \in \Z_{>0} $, the radial limits $ \lim_{q \to \zeta_k} \widehat{Z}_b (q; M) $ converge and we have
	\[
	Z_k (M) 
	=
	\lim_{q \to \zeta_k}
	\frac{1}{ \left( \zeta_{2k} - \zeta_{2k}^{-1} \right) \sqrt{\abs{H_1(M, \Z)}} }
	\sum_{a, b \in \Spin^c(M)} 
	\bm{e}(k \lk(a, a) - \lk(a, b)) \widehat{Z}_b (q; M),
	\]      
	where $ \lk $ is the linking form $ H_1(M, \Z) \times H_1(M, \Z) \to \Z $.
\end{conj}
In the case of negative definite plumbed manifolds, this conjecture was refined by Gukov--Pei--Putrov--\linebreak[0]Vafa~\cite{GPPV} as follows.

\begin{conj}[{\cite[Equation (A.28)]{GPPV}}] \label{conj:GPPV}
	For a negative definite plumbed manifold $ M $ and its GPPV invariants $ \widehat{Z}_b(q; M) $,  we have
	\begin{equation} \label{eq:GPPV_conj}
		Z_k(M)
		=
		\frac{1}{2(\zeta_{2k} - \zeta_{2k}^{-1}) \sqrt{\abs{\det W}}}
		\sum_{\substack{
				a \in \Z^V/W(\Z^V), \\
				b \in (\delta + 2\Z^V)/2W(\Z^V)
		}}
		\bm{e} \left( -k \transpose{a} W^{-1} a - \transpose{a} W^{-1} b \right)
		\lim_{q \to \zeta_k} \widehat{Z}_{b} (q; M).
	\end{equation}
\end{conj}

This conjecture has been proved by Murakami~\cite{M:GPPV}, except for the convergence of $ \lim_{q \to \zeta_k} \widehat{Z}_{b} (q; M) $ for each $ b \in \Spin^c(M) $.

\begin{thm}[{\cite[Theorem 1.2]{M:GPPV}}] \label{thm:GPPV_conj}
	For a negative definite plumbed manifold $ M $, we have
	\[
	Z_k(M)
	=
	\lim_{q \to \zeta_k} \frac{1}{2(\zeta_{2k} - \zeta_{2k}^{-1}) \sqrt{\abs{\det W}}}
	\sum_{\substack{
			a \in \Z^V/W(\Z^V), \\
			b \in (\delta + 2\Z^V)/2W(\Z^V)
	}}
	\bm{e} \left( -k \transpose{a} W^{-1} a - \transpose{a} W^{-1} b \right)
	\widehat{Z}_{b} (q; M).
	\]
	Here, the limit is taken over $ \theta_0 \le \arg \left( \tau - 1/k \right) \le \theta_1 $ for any $ 0 < \theta_0 < \theta_1 < \pi $. 
\end{thm}

\begin{rem}
	\cite{M:GPPV} only consider the case $ \tau = 1/k + \iu t $ for $ t > 0 $.
	However, the argument in \cite{M:GPPV} also works for $ t = e^{\iu \theta} $ with any $ 0 < \theta < \pi $ by applying the asympotic formula of Euler--Maclaurin type for the function $ \exp (e^{\iu \theta} \cdot \transpose{x} W^{-1} x) $.
\end{rem}


\subsection{Twisted sums of GPPV invariants} \label{subsec:GPPV_twist}


For each $ \alpha \in H^1(M, \Z) \cong W^{-1} (\Z^V) / \Z^V $, we define
\[
\widehat{Z}^{(\alpha)} (q; M)
=
\sum_{b \in (\delta + 2\Z^V)/2W(\Z^V)}
\bm{e}\left( \transpose{\alpha} b \right)
\widehat{Z}_{b} (q; M)
=
q^{\Delta} \sum_{l \in \delta + 2\Z^V} \bm{e}\left( \transpose{\alpha} l \right) F_l q^{-\transpose{l} W^{-1} l/4}.
\]
Then, \cref{thm:GPPV_conj} can be restated as
\begin{equation} \label{eq:GPPV_conj_modify}
	Z_k(M)
	=
	\frac{1}{2(\zeta_{2k} - \zeta_{2k}^{-1}) \sqrt{\abs{\det W}}}
	\lim_{q \to \zeta_k} 
	\sum_{\alpha \in W^{-1} (\Z^V) / \Z^V}
	\bm{e} \left( -k \transpose{\alpha} W \alpha \right)
	\widehat{Z}^{(\alpha)} (q; M).
\end{equation}
Thus, to establish asymptotic expansions of WRT invariants, it suffices to consider a modular transformation of $ \widehat{Z}^{(\alpha)} (q) $ for each $ \alpha \in W^{-1} (\Z^V) / \Z^V $. 


\section{Linking matrices of plumbing graphs} \label{sec:linking_matrix}


In this section, we consider properties of linking matrices.
In particular, we give a formula for its inverse matrices in \cref{lem:linking_matrix_inverse}, which we need to provide a modular transformation of $ \widehat{Z}^{(\alpha)} (q) $ .

Throughout this section, we fix a plumbing graph $ \Gamma $.
Let $ W $ be its linking matrix.


\subsection{Branches and trunks} \label{subsec:branch_trunk}


We use the following notation:
\begin{itemize}
	\item Let $ V $ be the set of vertices of $ \Gamma $.
	
	\item For each vertex $ v \in V $, we define its \emph{degree} as
	$ \deg(v) \coloneqq \left\{ v' \in V \relmiddle{|} \text{$ v' $ is adjacent to $ v $} \right\} $.
	
	\item For a positive integer $ d $, we define $ V_d \coloneqq \{ v \in V \mid \deg (v) = d \} $.
	\item We define $ V^{\veemidvert} \coloneqq \{ v \in V \mid \deg (v) \ge 3 \} $.%
	\footnote{The symbol $ \veemidvert $ represents a vertex where three edges meet.  
		Its Unicode code point is U+2A5B, and it can be produced by using the \texttt{stix} or \texttt{unicode-math} package with the command \texttt{\textbackslash veemidvert}.}
\end{itemize}

To calculate the inverse matrix $ W^{-1} $, we introduce the following terminology.

\begin{dfn}
	\begin{enumerate}
		\item A \emph{branch of $ \Gamma $ emanating from a vertex $ v \in V^{\veemidvert} $} is a subgraph consisting of $ v $, a vertex of degree $1$, and several intermediate vertices of degree $2$ connecting them, as illustrated in \cref{fig:branch}.
		
		\item A \emph{trunk of $ \Gamma $ emanating from a vertex $ v \in V^{\veemidvert} $ toward a vertex $ v' \in V^{\veemidvert} $} is an oriented subgraph consisting of $ v $, $ v' $, and several intermediate vertices of degree $2$ connecting them, oriented from $v$ to $v'$, as illustrated in \cref{fig:trunk}.
		
	\end{enumerate}
\end{dfn}

\begin{figure}[htp]
	\begin{minipage}[b]{0.45\linewidth}
		\centering
		\begin{tikzpicture}
			\node[shape=circle,fill=black, scale = 0.4] (0) at (0,0) { };
			\node[shape=circle,fill=black, scale = 0.4] (1) at (1,0) { };
			\coordinate (1+) at (1.5,0) { };
			\coordinate (2-) at (2.5,0) { };
			\node[shape=circle,fill=black, scale = 0.4] (2) at (3,0) { };
			\node[shape=circle,fill=black, scale = 0.4] (3) at (4,0) { };
			
			\node[above=0.1cm of 0]{$ w_v $};
			\node[above=0.1cm of 1]{$ w_{i_{1}} $};
			\node[above=0.1cm of 2]{$ w_{i_{s-1}} $};
			\node[above=0.1cm of 3]{$ w_{i_{s}} $};
			
			\path [-](0) edge node[left] {} (1);
			\path [-](1) edge node[left] {} (1+);
			\path [dashed](1+) edge node[left] {} (2-);
			\path [-](2-) edge node[left] {} (2);
			\path [-](2) edge node[left] {} (3);
			
			\coordinate (0u) at (-0.7,-0.5) { };
			\coordinate (0d) at (-0.7,0.5) { };
			\path [-](0u) edge node[left] {} (0);
			\node[left=0.5cm of 0,rotate=90,anchor=center]{$\cdots$};
			\path [-](0d) edge node[left] {} (0);
		\end{tikzpicture}
		\caption{A branch of plumbing graph} \label{fig:branch}
	\end{minipage}
	\begin{minipage}[b]{0.45\linewidth}
		\centering
		\begin{tikzpicture}
			\node[shape=circle,fill=black, scale = 0.4] (0) at (0,0) { };
			\node[shape=circle,fill=black, scale = 0.4] (1) at (1,0) { };
			\coordinate (1+) at (1.5,0) { };
			\coordinate (2-) at (2.5,0) { };
			\node[shape=circle,fill=black, scale = 0.4] (2) at (3,0) { };
			\node[shape=circle,fill=black, scale = 0.4] (3) at (4,0) { };
			
			\node[above=0.1cm of 0]{$ w_v $};
			\node[above=0.1cm of 1]{$ w_{i_{1}} $};
			\node[above=0.1cm of 2]{$ w_{i_{s}} $};
			\node[above=0.1cm of 3]{$ w_{v'} $};
			
			\path [-](0) edge node[left] {} (1);
			\path [-](1) edge node[left] {} (1+);
			\path [dashed](1+) edge node[left] {} (2-);
			\path [-](2-) edge node[left] {} (2);
			\path [-](2) edge node[left] {} (3);
			
			\coordinate (0u) at (-0.7,-0.5) { };
			\coordinate (0d) at (-0.7,0.5) { };
			\path [-](0u) edge node[left] {} (0);
			\node[left=0.5cm of 0,rotate=90,anchor=center]{$\cdots$};
			\path [-](0d) edge node[left] {} (0);
			
			\coordinate (3u) at (4.7,-0.5) { };
			\coordinate (3d) at (4.7,0.5) { };
			\path [-](3u) edge node[left] {} (3);
			\node[right=0.5cm of 3,rotate=90,anchor=center]{$\cdots$};
			\path [-](3d) edge node[left] {} (3);
		\end{tikzpicture}
		\caption{A trunk of plumbing graph} \label{fig:trunk}
	\end{minipage}
\end{figure}

For a branch or trunk $ \beta $ emanating from a vertex $ v \in V^{\veemidvert} $ weighted as in \cref{fig:branch} or \cref{fig:trunk}, we define the following notation:
\[
i_\beta \coloneqq i_s \in V_1, \quad
l_\beta \coloneqq s, \quad
\overline{v} \coloneqq \{ \text{branches emanating from $ v $} \},
\]
\[
W_\beta \coloneqq
\pmat{w_{i_1} & 1 & & \\ 1 & \ddots & \ddots & \\ & \ddots & \ddots & 1 \\ & & 1 & w_{i_s}}, \quad
p_\beta 
\coloneqq 
\det W_\beta, \quad
q_\beta 
\coloneqq 
\vmat{w_{i_2} & 1 & & \\ 1 & \ddots & \ddots & \\ & \ddots & \ddots & 1 \\ & & 1 & w_{i_s}}, \quad
\acute{q}_\beta 
\coloneqq 
\vmat{w_{i_{s-1}} & 1 & & \\ 1 & \ddots & \ddots & \\ & \ddots & \ddots & 1 \\ & & 1 & w_{i_1}}.
\]
The following properties hold for these notations.

\begin{lem} \label{lem:p_b}
	For a branch or trunk $ \beta $ emanating from a vertex $ v \in V^{\veemidvert} $ weighted as in \cref{fig:branch} or \cref{fig:trunk}, the following properties hold.
	\begin{enumerate}
		\item \label{item:lem:p_b:comp}
		If $ p_\beta \neq 0 $, then $ q_\beta / p_\beta, \acute{q}_\beta / p_\beta $, and $ (-1)^{l_\beta}/p_\beta $ are $ (1,1), (s, s) $, and $ (1,s) $ components of the matrix $ W_\beta^{-1} $ respectively.
		
		\item \label{item:lem:p_b:2_2_matrix} We have
		\[
		\pmat{w_{i_1} & -1 \\ 1 & 0} \cdots \pmat{w_{i_s} & -1 \\ 1 & 0}
		= \pmat{p_\beta & * \\ q_\beta & *}.
		\]
		\item \label{item:lem:p_b:conti_frac}
		If $ p_\beta \neq 0 $, then we have
		\[
		\frac{q_\beta}{p_\beta}
		= \cfrac{1}{w_{i_1} - \cfrac{1}{\ddots - \cfrac{1}{w_{i_s}}}}, \quad
		\frac{\acute{q}_\beta}{p_\beta}
		= \cfrac{1}{w_{i_s} - \cfrac{1}{\ddots - \cfrac{1}{w_{i_1}}}}.
		\]
	\end{enumerate}
\end{lem}

\begin{proof}
	\cref{item:lem:p_b:comp} follows from the cofactor expansion of the inverse matrix.
	
	We now prove \cref{item:lem:p_b:2_2_matrix}. 
	For each $ i \le j \le s $, let
	\[
	a_j \coloneqq
	\pmat{w_{i_j} & 1 & & \\ 1 & \ddots & \ddots & \\ & \ddots & \ddots & 1 \\ & & 1 & w_{i_s}}, \quad
	\pmat{p_j & * \\ q_j & *} \coloneqq
	\pmat{w_{i_j} & -1 \\ 1 & 0} \cdots \pmat{w_{i_s} & -1 \\ 1 & 0}.
	\]
	Then, since we have
	\[
	a_s = k_s = p_s, \quad
	q_s = 1, \quad
	a_j = k_j a_{j+1} - a_{j+2}, \quad
	p_j = k_j p_{j+1} - p_{j+2}, \quad
	q_j = p_{j-1},
	\]
	we obtain $ p_j = a_j, q_j = a_{j-1} $. 
	Thus, we have $ p_1 = p_\beta, q_1 = q_\beta $ because $ a_1 = p_\beta, a_2 = q_\beta $.
	
	\cref{item:lem:p_b:conti_frac} follows from \cref{item:lem:p_b:2_2_matrix} and the action of $ \SL_2(\Z) $ on $ \bbP^1(\Q) $ via M\"{o}bius transformations.
\end{proof}

Next, we give a sufficient condition for $ p_\beta \neq 0 $.

\begin{dfn} \label{dfn:strong_nondeg}
	A plumbing graph $ \Gamma $ is \emph{strongly nondegenerate} if its linking matrix $ W $ is invertible
	and the $ V^{\veemidvert} \times V^{\veemidvert} $-submatrix of $ W^{-1} $ is also invertible.
\end{dfn}

\begin{rem} \label{rem:neg_def_strong_nondeg}
	As defined in \cite[Definition 4.3]{Gukov-Manolescu}, a plumbing graph $ \Gamma $ is said to be \emph{weakly negative definite}
	if $ W $ is invertible and the $ V^{\veemidvert} \times V^{\veemidvert} $-submatrix of $ W^{-1} $ is negative definite.
	In this case, $ \Gamma $ is strongly nondegenerate.
\end{rem}

We have the following sufficient condition for $ p_\beta \neq 0 $.

\begin{lem} \label{lem:strong_nondeg_p_non-zero}
	For a strongly nondegenerate plumbing graph $ \Gamma $, its any branch or trunk $ \beta $ satisfies $ p_\beta \neq 0 $.
\end{lem}

\begin{proof}
	We can assume $ V^{\veemidvert} \neq \emptyset $ since $ 0 \times 0 $ matrix has determinant $ 1 $.
	Let $ A $ be the $ (V_1 \cup V_2) \times (V_1 \cup V_2) $-submatrix of $ W $ and
	$ C^* $ be the $ V^{\veemidvert} \times V^{\veemidvert} $-submatrix of $ W^{-1} $. 
	By applying \cref{lem:block_matrix} for $ W $, we have $ (\det C^*)(\det W) = \det A $.
	Since $ \Gamma $ is strongly nondegenerate, $ C $ is invertible.
	Thus, $ A $ is also invertible. 
	Here, we can write $ A = \diag(W_\beta)_\beta $, where $ \beta $ runs over all branches or trunks.
	Thus, $ W_\beta $ is invertible, that is, $ p_\beta \neq 0 $. 
\end{proof}


\subsection{Computing inverses of linking matrices} \label{subsec:linking_matrix_inverse}


For a strongly nondegenerate plumbing graph $ \Gamma $, we define the following notation. 
\begin{itemize}
	\item For vertices $ v \neq v' \in V^{\veemidvert} $, define
	\begin{align}
		w_{v, v}^{\veemidvert} = w_{v}^{\veemidvert}
		&\coloneqq
		w_v - \sum_{\beta \in \overline{v}} \frac{\acute{q}_\beta}{p_\beta} 
		- \sum_{\text{trunk $ \tau $ emanating from $ v $}} \frac{\acute{q}_\beta}{p_\beta}, 
		\\
		w_{v, v'}^{\veemidvert}
		&\coloneqq
		\begin{cases}
			1 & \text{$ v $ is adjacent to $ v' $}, \\
			-1/p_\tau & \text{there exists a trunk $ \tau $ emanating from $ v $ toward $ v' $}, \\
			0 & \text{otherwise}.
		\end{cases}
	\end{align}
	\item Let $ W^{\veemidvert} \coloneqq (w_{v, v}^{\veemidvert})_{v, v' \in V^{\veemidvert}} \in \Sym(\Z^{V^{\veemidvert}}) $. 
\end{itemize}

Here we remark that $ p_\beta \neq 0 $ for each branch or trunk $ \beta $ by \cref{lem:strong_nondeg_p_non-zero} since $ \Gamma $ is strongly nondegenerate.

We now arrive at the formula for the inverse of a linking matrix, which was the main objective of this section.

\begin{lem} \label{lem:linking_matrix_inverse}
	For a strongly nondegenerate plumbing graph $ \Gamma $, the matrix $ W^{\veemidvert} $ is the inverse of 
	the $ V^{\veemidvert} \times V^{\veemidvert} $-submatrix of $ W^{-1} $ and satisfies
	\[
	\det W = (\det W^{\veemidvert}) \prod_{\text{branches or trunks $ \beta $}} p_\beta.
	\]
\end{lem}

\begin{proof}
	The claim follows by applying \cref{lem:block_matrix} as in \cref{lem:strong_nondeg_p_non-zero} and computing the $ V^{\veemidvert} \times V^{\veemidvert} $-submatrix with the aid of \cref{lem:p_b} \cref{item:lem:p_b:comp}. 
\end{proof}


\subsection{Properties of strongly nondegenerate plumbing graph} \label{subsec:strong_nondeg_condition}


We recall that two plumbing graphs define homeomorphic 3-manifolds if and only if they are related by Neumann moves shown in \cref{fig:Neumann}.
For an equivalence class under Neumann moves, we have the following property.

\begin{lem} \label{lem:V_minimal}
	If a strongly non-degenerate plumbing graph $ \Gamma $ attains the minimal number of vertices $ \abs{V} $ in its equivalence class under Neumann moves, then for any branch or trunk $ \beta $, we have $ \abs{p_\beta} \ge 2 $.
\end{lem}

\begin{rem}
	A plumbing graph which attains the minimal number of vertices $ \abs{V} $ is not unique within its equivalence class under Neumann moves.
\end{rem}

\cref{lem:V_minimal} immediately follows from the following lemma. 

\begin{lem} \label{lem:triple_diag_det}
	For integers $ w_1, \dots, w_s \in \Z \smallsetminus \{ 0, \pm 1 \} $, we have
	\[
	\vmat{w_{1} & 1 & & \\ 1 & \ddots & \ddots & \\ & \ddots & \ddots & 1 \\ & & 1 & w_{s}}
	\in \Z \smallsetminus \{ 0, \pm 1 \}.
	\]
\end{lem}

\begin{proof}
	We prove
	\[
	\abs{ \det \pmat{w_{1} & 1 & & \\ 1 & \ddots & \ddots & \\ & \ddots & \ddots & 1 \\ & & 1 & w_{s}} }
	> \abs{ \det \pmat{w_{2} & 1 & & \\ 1 & \ddots & \ddots & \\ & \ddots & \ddots & 1 \\ & & 1 & w_{s}} }
	> 1
	\]
	by induction on $ s $. 
	In the case when $ s=2 $, this follows from a direct computation.
	If this holds for $ s-1 $, then this also holds for $ s $ by considering the cofactor expansion. 
\end{proof}


\section{Asymptotic expansions of GPPV invariants and WRT invariants} \label{sec:GPPV_modularity}


Finally, we establish asymptotic formulas for GPPV invariants and WRT invariants.


\subsection{Expression of GPPV invariants as a sum grouped by vertices of degree $ \ge 3 $} \label{subsec:general_false_theta_deg_3}


In our modular transformation formula for false modular series (\cref{thm:false_modular_series_modular}) or false theta functions (\cref{thm:false_theta}), we only consider false quasi-polynomials corresponding to quasi-polynomials in $ \overline{\frakQ}^0 $.
For this reason, we need to express GPPV invariants as a sum grouped by vertices of degree $ \ge 3 $.
Our expression generalizes the expression by Andersen--Misterg\aa{}rd~\cite[Theorem 3]{Andersen-Mistegard} for Seifert homology spheres.

In what follows, we employ the terminology and notation set up in \cref{sec:linking_matrix}.

Throughout the rest of this section, we fix a plumbed manifold $ M = M(\Gamma) $ defined by a plumbing graph $ \Gamma $ and a 1-cocycle $ \alpha \in H^1(M, \Z) \cong W^{-1} (\Z^V) / \Z^V $.
We assume that the linking matrix $ W $ of $ \Gamma $ is negative definite and $ V^{\veemidvert} \coloneqq \{ v \in V \mid \deg(v) \ge 3 \} \neq \emptyset $.
In particular, $ M $ is not a lens space.
In this case, each branch or trunk $ \beta $ of $ \Gamma $ satisfies $ p_\beta \neq 0 $ by \cref{lem:strong_nondeg_p_non-zero}, since $ \Gamma $ is strongly nondegenerate (\cref{dfn:strong_nondeg}) by \cref{rem:neg_def_strong_nondeg}. 
Using this fact, we set up the following notation for the operation of grouping vertices of degree $ \ge 3 $.

\begin{notn*}
	\begin{itemize}
		\item Let
		\[
		\Delta^{\veemidvert} \coloneqq 
		\Delta - \sum_{v \in V^{\veemidvert}} \sum_{\beta \in \overline{v}} \frac{\acute{q}_\beta}{4p_\beta}.
		\]
		\item We define a quadratic form $ Q^{\veemidvert} \colon \Z^{V^{\veemidvert}} \to (\det W)^{-1} \Z $ as
		$ Q^{\veemidvert}(x) \coloneqq -\transpose{x} (W^{\veemidvert})^{-1} x $.
		\item For each $ v \in V^{\veemidvert} $, we denote $ P_v $ the least common multiple of $ \{ p_\beta \mid \beta \in \overline{v} \} $ and define a cyclotomic rational function $ F_v^{\veemidvert, (\alpha)} (z_v) $ of one variable and its Laurent coefficients as
		\[
		F_v^{\veemidvert, (\alpha)} (z_v)
		\coloneqq 
		F_v ( \bm{e}(\alpha_v) z_v^{P_v})
		\prod_{\beta \in \overline{v}} F_{i_\beta} ( \bm{e}(\alpha_{i_\beta}) z_v^{P_v / p_\beta})
		\eqqcolon \sum_{m_v = m_{v, 0}}^{\infty} F_{v, m_v}^{\veemidvert, (\alpha)} z_v^{m_v}.
		\]
		\item Let $ n_0 \coloneqq (n_{v, 0})_{v \in V^{\veemidvert}} \in \Z^{V^{\veemidvert}} $.
		\item For each $ m \in \Z_{\ge m_0}^{V^{\veemidvert}} $, define
		\[
		F_m^{\veemidvert, (\alpha)}
		\coloneqq
		\prod_{v \in V^{\veemidvert}} F_{v, m_v}^{\veemidvert, (\alpha)}.
		\]
		\item Let $ z_{V^{\veemidvert}}^{} = (z_v)_{v \in V^{\veemidvert}} $ be a variable.
		\item Let
		\[
		F^{\veemidvert, (\alpha)} \left( \xi_{V^{\veemidvert}}^{} \right)
		\coloneqq
		\prod_{v \in V^{\veemidvert}} F_v^{\veemidvert, (\alpha)} (z_v).
		\]
	\end{itemize}
\end{notn*}

\begin{rem} \label{rem:inverse_mat_quad_form}
	For any $ l \in \{ \pm 1 \}^{V_1} \times \{ 0 \}^{V_2} \times \Z^{V^{\veemidvert}} $, we have
	\[
	-\transpose{l} W^{-1} l
	= Q^{\veemidvert} \left( \left( l_v - \sum_{\beta \in \overline{v}} \frac{l_{i_\beta}}{p_\beta} \right)_{v \in V^{\veemidvert}} \right)
	- \sum_{v \in V^{\veemidvert}} \sum_{\beta \in \overline{v}} \frac{\acute{q}_\beta}{p_\beta}
	\]
	by \cref{lem:block_matrix,lem:linking_matrix_inverse}. 
	
	Note that the last term is independent of $l$ since $l_i^2 = 1$ for any $i \in V_1$.
	By this reason, applying the operation $ W \mapsto W^{\veemidvert} $ once again to $ W^{\veemidvert} $ does not yield a valid representation.
\end{rem}

\begin{rem} \label{rem:F_v_symmetry_deg_3}
	For each $ v \in V^{\veemidvert} $, we have $ F_v^{\veemidvert, (\alpha)} (z_v) = (-1)^{\deg (v) + \abs{\overline{v}}} F_v^{\veemidvert, (-\alpha)} (z_v^{-1}) $
	by \cref{lem:coeff_of_F} \cref{item:lem:coeff_of_F:symmetry}.
\end{rem}

Under the above preparations, we obtain the following false theta function representation.

\begin{lem} \label{lem:GPPV_false_theta_rep}
	We have
	\begin{align}
		\widehat{Z}^{(\alpha)} (q; M) 
		&=
		2^{-\abs{V^{\veemidvert}}} q^{\Delta^{\veemidvert}}
		\sum_{e \in \{ \pm 1 \}^{V^{\veemidvert}}} \left( \prod_{v \in V^{\veemidvert}} e_v^{\deg (v) - \abs{\overline{v}}} \right)
		\sum_{n \in \Z_{\ge n_0}^{V^{\veemidvert}}} F_n^{\veemidvert, (e\alpha)} q^{Q^{\veemidvert}((e_v n_v / 2P_v)_{v \in V^{\veemidvert}})}
		\\
		&=
		q^{\Delta^{\veemidvert}}
		\sum_{n \in \Z^{V^{\veemidvert}}} \widetilde{\coe} \left[ F^{\veemidvert, (\alpha)} \right] (n) q^{Q^{\veemidvert}((n_v / 2P_v)_{v \in V^{\veemidvert}})},
	\end{align}
	where we define $ e\alpha \coloneqq (e_v \alpha_v)_{v \in V^{\veemidvert}} \in \Z^{V^{\veemidvert}} $.
\end{lem}

\begin{proof}
	The last equality follows from \cref{lem:R_C_corresp} \cref{item:lem:R_C_corresp:vp,item:lem:R_C_corresp:involution_bilateral,rem:F_v_symmetry_deg_3}.
	We prove the first equality.
	By \cref{lem:coeff_of_F} \cref{item:lem:coeff_of_F:expression,rem:inverse_mat_quad_form}, we have
	\begin{align}
		q^{-\Delta} \widehat{Z}^{(\alpha)} (q; M)
		=
		\sum_{\varepsilon \in \{ \pm 1 \}^{V_1}} \left( \bm{e} (\alpha_i \varepsilon_i) (-\varepsilon_i) \right)
		\sum_{l \in (\deg (v))_{v \in V^{\veemidvert}} + 2 \Z^{V^{\veemidvert}}} 
		&q^{Q^{\veemidvert} ( ( l_v - \sum_{\beta \in \overline{v}} \varepsilon_{i_\beta}/ p_\beta )_{v \in V^{\veemidvert}} )/4} \\
		&\prod_{v \in V^{\veemidvert}} \bm{e} (\alpha_v l_v) \sgn(l_v)^{\deg (v)} \cdot \frac{1}{2} F_{v, l_v}.
	\end{align}
	On the right-hand side, for each $i \in V_1$, choose $v \in V^{\veemidvert}$ and $\beta \in \overline{v}$ such that $i = i_\beta$ (here we use the notation introduced in \cref{subsec:branch_trunk}).  
	We replace $\varepsilon_i$ by $-\sgn(l_v)\varepsilon_\beta$.  
	Moreover, for each $v \in V^{\veemidvert}$, we set $e_v \coloneqq \sgn(l_v)$ and
	$ n_v / P_v \coloneqq l_v + \sum_{\beta \in \overline{v}} \varepsilon_\beta / p_\beta $.
	Then, we have
	\begin{align}
		q^{-\Delta} \widehat{Z}^{(\alpha)} (q; M)
		=
		2^{-\abs{V^{\veemidvert}}}
		\sum_{e \in \{ \pm 1 \}^{V^{\veemidvert}}} 
		&\left( \prod_{v \in V^{\veemidvert}} e_v^{\deg (v) - \abs{\overline{v}}} \right)
		\sum_{n \in \Z^{V^{\veemidvert}}} 
		q^{Q^{\veemidvert}((e_v n_v / 2P_v)_{v \in V^{\veemidvert}})} \\
		&\prod_{v \in V^{\veemidvert}}
		\sum_{\substack{
				\varepsilon \in \{ \pm 1 \}^{\overline{v}}, \\
				l_v \in \deg (v) + 2 \Z_{\ge -1}, \\
				l_v + \sum_{\beta \in \overline{v}} \varepsilon_\beta / p_\beta = n_v / P_v
		}}
		\bm{e} (\alpha_v l_v) F_{v, l_v}
		\prod_{\beta \in \overline{v}} \bm{e} (\alpha_{i_\beta} \varepsilon_{\beta}) \varepsilon_{\beta}
	\end{align}
	To calculate the coefficients of $ q^{Q^{\veemidvert}((e_v n_v / 2P_v)_{v \in V^{\veemidvert}})} $, we consider its generating function.
	For each $ v \in V^{\veemidvert} $, we can calculate as
	\begin{align}
		&\phant
		\sum_{n \in \Z} z_v^{n_v} 
		\sum_{\substack{
				\varepsilon \in \{ \pm 1 \}^{\overline{v}}, \\
				l_v \in \deg (v) + 2 \Z_{\ge -1}, \\
				l_v + \sum_{\beta \in \overline{v}} \varepsilon_\beta / p_\beta = n_v / P_v
		}}
		\bm{e} (\alpha_v l_v) F_{v, l_v}
		\prod_{\beta \in \overline{v}} \bm{e} (\alpha_{i_\beta} \varepsilon_{\beta}) \varepsilon_{\beta} \\
		&=
		\sum_{\substack{
				\varepsilon \in \{ \pm 1 \}^{\overline{v}}, \\
				l_v \in \deg (v) + 2 \Z_{\ge -1}
		}}
		z_v^{P_v(l_v + \sum_{\beta \in \overline{v}} \varepsilon_\beta / p_\beta)}
		\bm{e} (\alpha_v l_v) F_{v, l_v}
		\prod_{\beta \in \overline{v}} \bm{e} (\alpha_{i_\beta} \varepsilon_{\beta}) \varepsilon_{\beta} \\
		&=
		\left(
		\sum_{l_v \in \deg (v) + 2 \Z_{\ge -1}}
		F_{v, l_v} \left( \bm{e} (\alpha_v) z_v^{P_v} \right)^{l_v}
		\right)
		\prod_{\beta \in \overline{v}} 
		\sum_{\varepsilon_\beta \in \{ \pm 1 \}}
		\left( \bm{e} (\alpha_{i_\beta}) z_v^{P_v / p_\beta} \right)^{\varepsilon_\beta} \\
		&=
		F_v \left( \bm{e} (\alpha_v) z_v^{P_v} \right)
		\prod_{\beta \in \overline{v}} 
		F_{i_\beta} \left( \bm{e} (\alpha_{i_\beta}) z_v^{P_v / p_\beta} \right) \\
		&=
		F_v^{\veemidvert, (\alpha)} (z_v) \\
		&=
		\sum_{n_v = n_{v, 0}}^{\infty} F_{v, n_v}^{\veemidvert, (\alpha)} z_v^{n_v}.
	\end{align}
	Thus, we obtain the claim. 
\end{proof}


\subsection{Quantum modularity of GPPV invariants} \label{subsec:GPPV_modularity}


Since GPPV invariants $ \widehat{Z}_b (q; M) $ and their twisted sum $ \widehat{Z}^{(\alpha)} (q; M) $ are false theta functions, they admit modular transformation formulas as in \cref{thm:main_false_theta} as follows.

\begin{cor} \label{cor:GPPV_modular}
	GPPV invariants $ \widehat{Z}_b (q; M) $ and their twisted sum $ \widehat{Z}^{(\alpha)} (q; M) $ have the modular transformation formulas
	\begin{alignat}{3}
		&\frac{1}{\sqrt{\abs{H_1(M, \Z)}}} {\sqrt{\frac{\iu}{\tau}}}^{\, \abs{V^{\veemidvert}}} &
		&\cdot \widehat{Z}_b \left( -\frac{1}{\tau}; M \right) &
		&=
		\sum_{\pi \in \Pi(V^{\veemidvert})} \sum_{1 \le i \le N_\pi} \widehat{Z}_{\pi, i} (\tau) \Omega_{\pi, i} (\tau)
		+ \Omega(\tau),
		\\
		&\frac{1}{\sqrt{\abs{H_1(M, \Z)}}} {\sqrt{\frac{\iu}{\tau}}}^{\, \abs{V^{\veemidvert}}} &
		&\cdot \widehat{Z}^{(\alpha)} \left( -\frac{1}{\tau}; M \right) &
		&=
		\sum_{\pi \in \Pi(V^{\veemidvert})} \sum_{1 \le i \le N_\pi^{(\alpha)}} 
		\widehat{Z}_{\pi, i}^{(\alpha)} (\tau) \Omega_{\pi, i}^{(\alpha)} (\tau)
		+ \Omega^{(\alpha)} (\tau),
	\end{alignat}
	with the following notation:
	\begin{itemize}
		\item We denote $ \Pi(V^{\veemidvert}) \coloneqq \left\{ (V'_1, \dots, V'_s) \mid V'_1 \sqcup \cdots \sqcup V'_s = V^{\veemidvert} \right\} $ 
		the set of ordered partitions of $ V^{\veemidvert} $.
		
		\item For each $ \pi \in \Pi(V^{\veemidvert}) $, $ N_\pi $ and $ N_\pi^{(\alpha)} $ are a positive integer.
		
		\item For each $ \pi= (V'_1, \dots, V'_s) \in \Pi(V^{\veemidvert}) $, and $ 1 \le i \le N_\pi $ or $ 1 \le j \le N_\pi^{(\alpha)} $, 
		$ \Omega_{\pi, i} (\tau) $ and $ \Omega_{\pi, j}^{(\alpha)} (\tau) $ are holomorphic functions on $ \widetilde{\bbC} $
		and $ \widehat{Z}_{\pi, i} (\tau) $ and $ \widehat{Z}_{\pi, j}^{(\alpha)} (\tau) $ are false theta functions associated to a symmetric matrix
		$ 2P_{V'_1 \sqcup \cdots \sqcup V'_{s-1}} W^{\veemidvert} P_{V'_1 \sqcup \cdots \sqcup V'_{s-1}} $, where
		$ P_{V'}^{} \coloneqq \diag(P_v)_{v \in V'} $ and
		\[
		W^*_\pi \coloneqq
		\pmat{ ((W^{\veemidvert})_{V'_1}^*)^{-1} & & & \\ & (W^{\veemidvert}_{{V'_1}^\complement})^*_{V'_2} & & \\ & & \ddots & & \\ & & & (W^{\veemidvert}_{(V'_1 \sqcup \cdots \sqcup V'_{s-2})^\complement})^*_{V'_{s-1}}}
		\in \Sym_{V'_1 \sqcup \cdots \sqcup V'_{s-1}}^{} (\Q).
		\]
		
		\item The functions $ \Omega (\tau) $ and $ \Omega^{(\alpha)} (\tau) $ are some holomorophic functions on $ \widetilde{\bbC} $.
		
	\end{itemize}
\end{cor}

\begin{rem} \label{rem:GPPV_int_rep}
	As a preliminary step to the modular transformation formulas above, we obtain the following elegant expressions by \cref{ex:false_theta_PV_rep}:
	\begin{align}
		\widehat{Z} (q; M)
		&=
		\sqrt{\frac{\iu}{\tau}}^{\abs{V}} \frac{q^{\Delta}}{\sqrt{\abs{H_1 (M, \Z)}}} 
		\PV \int_{\R^V} \widetilde{q}^{\, \transpose{\xi_V^{}} W \xi_V^{}/2} F_V^{} (\bm{e} (\xi_V^{})) d\xi_V^{},
		\\
		\widehat{Z}^{(\alpha)} (q; M)
		&=
		\sqrt{\frac{\iu}{\tau}}^{\abs{V^{\veemidvert}}} \frac{q^{\Delta^{\veemidvert}}}{\sqrt{\abs{H_1 (M, \Z)}}} 
		\PV \int_{\R^V} \widetilde{q}^{\, \transpose{\xi_{V^{\veemidvert}}^{}} P W^{\veemidvert} P \xi_{V^{\veemidvert}}^{}/2}
		F^{\veemidvert, (\alpha)} (\bm{e} (\xi_{V^{\veemidvert}}^{})) d\xi_{V^{\veemidvert}}^{},
	\end{align}
	where $ P \coloneqq \diag(P_v)_{v \in V^{\veemidvert}} $.
\end{rem}


\subsection{Asympototic expansions of WRT invariants} \label{subsec:WRT_asymp}


By \cref{thm:false_theta_asymp}, we obtain the following asymptotic formulas.

\begin{cor} \label{cor:GPPV_asymp}
	The function $ \widehat{Z}^{(\alpha)} (q; M) $ satisfies the following asymptotic formulas.
	\begin{enumerate}
		\item \label{item:cor:GPPV_asymp:1}
		For any rational number $ \rho \in \Q $, we have
		\[
		\widehat{Z}^{(\alpha)} (\tau + \rho; M)
		\sim
		\sum_{j \in \frac{1}{2} \Z} Z_j^{(\alpha)} (\rho) \tau^j
		\quad \text{ as } \tau \to 0
		\]
		for some complex number
		\[
		Z_j^{(\alpha)} (\rho) \in 
		\sprod{ \bm{e} \left( -\frac{\rho}{4} \transpose{m} (W^{\veemidvert})^{-1} m 
			\relmiddle| m \in \prod_{v \in V^{\veemidvert}} \frac{1}{P_v} \Z, \, F^{\veemidvert, (\alpha)}_m \neq 0 \right) }_\Q.
		\]
		
		\item \label{item:cor:GPPV_asymp:2}
		In the case where $ \rho = h/k \neq 0 $ for fixed $ h \in \Z \smallsetminus \{ 0 \} $, for any $ j $ we have
		\[
		Z_j^{(\alpha)} (\rho)
		\sim
		\sum_{\theta \in \calS^{(\alpha)}} \bm{e} \left( -\frac{\theta}{\rho} \right) \widetilde{\varphi}_{\theta, h, j}^{(\alpha)} \left( \frac{1}{\sqrt{k}} \right) 
		\quad \text{ as } k \to \infty
		\]
		for some finite set $ \calS^{(\alpha)} \subset \Q / \Z $ independent of $ \rho $ and $ j $ and 
		some formal power series 
		$ \widetilde{\varphi}_{\theta, h, j}^{(\alpha)} (u) \in\bbC((u)) $.
		
		Moreover, we can write $ \calS^{(\alpha)} $ explicitly as
		\[
		\calS^{(\alpha)} = \{ 0 \} \cup \bigcup_{\pi \in \Pi(V^{\veemidvert})} \calS_\pi^{(\alpha)},
		\]
		where
		\begin{align}
			\calS_\pi^{} &\coloneqq
			\left\{
			-\frac{1}{4} \transpose{n} W_\pi^* n
			\bmod \Z
			\relmiddle|
			n \in \prod_{v \in V'_1 \sqcup \cdots \sqcup V'_{s-1}} \calP_v^{(\alpha)}
			\right\},
			\\
			\calP_{v}^{(\alpha)}
			&\coloneqq
			\begin{dcases}
				\left\{
				n_{v} \in \Z
				\relmiddle|
				\# \left\{ \beta \in \overline{v} \relmiddle| 2\alpha_{i_\beta} + \frac{n_v}{p_\beta} \in \Z \right\} \le \deg (v) - 3
				\right\}
				& \text{ if } \alpha_v \in \Z, \\
				\emptyset & \text{ if } \alpha_v \notin \Z,
			\end{dcases}
		\end{align}
		for $ \pi = (V'_1, \dots, V'_s) $ and $ W_\pi^* $ are as in \cref{cor:GPPV_modular}.
	\end{enumerate}
\end{cor}

Combining this asymptotic formula and the radial limit theorem (\cref{thm:GPPV_conj}), we obtain asymptotic expansions of WRT invariants as follows.

\begin{cor} \label{cor:WRT_asymp}
	We have an asymptotic expansion
	\[
	Z_k (M)
	\sim
	\sum_{\theta \in \calS} \bm{e} (k \theta) Z_\theta (k) 
	\quad \text{ as } k \to \infty
	\]
	for a finite set
	\[
	\calS \coloneqq
	\bigcup_{\alpha \in H^1(M, \Z)}  \left( -\transpose{\alpha} W \alpha + \calS^{(\alpha)} \right)
	\subset \Q/\Z
	\]
	and Puiseux series $ Z_\theta (k) \in \bbC((k^{-1/2})) $ for each $ \theta \in \calS $.
\end{cor}

\begin{rem}
	In particular, the above finite set $ \calS $ contains
	\[
	\bigcup_{\alpha \in H^1(M, \Z)} \calS_{(V^{\veemidvert})}^{(\alpha)}
	=
	\left\{
	\transpose{\alpha} W \alpha -\frac{1}{2} n W^{\veemidvert} n
	\relmiddle|
	n \in \prod_{v \in V^{\veemidvert}} \calP_v^{(\alpha)}
	\right\}.
	\]
\end{rem}

By the Witten's asymptotic expansion conjecture (\cref{conj:asymptotic}) and \cref{cor:WRT_asymp}, we obtain the following conjecture.

\begin{conj} \label{conj:CS_inv_expression}
	The finite set $ \calS $ in \cref{cor:WRT_asymp} coincides with the set $ \CS_\bbC (M) \cup \{ 0 \} $, where $ \CS_\bbC (M) $ is the set of Chern--Simons invariants of $ M $.
\end{conj}


\subsection{Examples} \label{subsec:CS_ex}


To conclude this section, we write $ \calS $ in some cases.


\subsubsection{Case of Seifert homology spheres} \label{subsubsec:CS_examples_Seifert}


Seifert homology spheres are negative definite plumbed manifolds determined by plumbing graphs shown in \cref{fig:Seifert} (star-shaped graph) with $ N \ge 3 $ and its linking matrix $ W $ is negative definite and satisfies $ \det W = \pm 1 $.

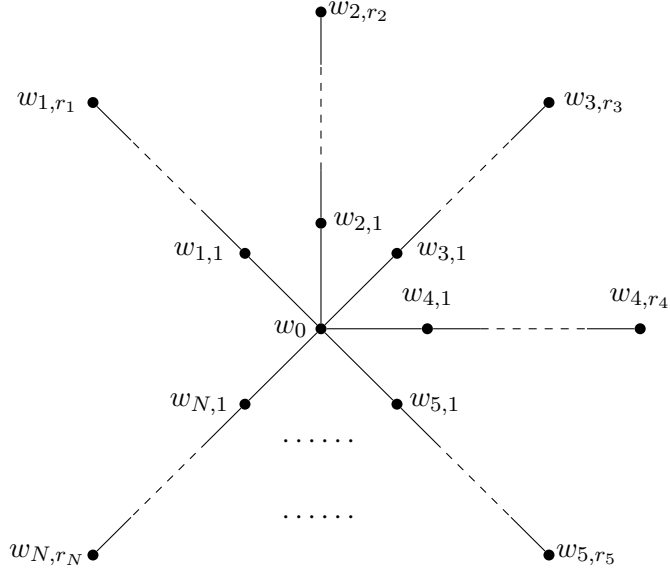
\begin{figure}[htbp]
	\centering
	\begin{tikzpicture}
		\node[shape=circle,fill=black, scale = 0.4] (1) at (0,0) { };
		\node[shape=circle,fill=black, scale = 0.4] (21) at (-1,1) { };
		\coordinate (22) at (-1.5,1.5) { };
		\coordinate (23) at (-2.5,2.5) { };
		\node[shape=circle,fill=black, scale = 0.4] (24) at (-3,3) { };
		\node[shape=circle,fill=black, scale = 0.4] (31) at (0,1.4) { };
		\coordinate (32) at (0,2.1) { };
		\coordinate (33) at (0,3.5) { };
		\node[shape=circle,fill=black, scale = 0.4] (34) at (0,4.2) { };
		\node[shape=circle,fill=black, scale = 0.4] (41) at (1,1) { };
		\coordinate (42) at (1.5,1.5) { };
		\coordinate (43) at (2.5,2.5) { };
		\node[shape=circle,fill=black, scale = 0.4] (44) at (3,3) { };
		\node[shape=circle,fill=black, scale = 0.4] (51) at (1.4,0) { };
		\coordinate (52) at (2.1,0) { };
		\coordinate (53) at (3.5,0) { };
		\node[shape=circle,fill=black, scale = 0.4] (54) at (4.2,0) { };
		\node[shape=circle,fill=black, scale = 0.4] (61) at (1,-1) { };
		\coordinate (62) at (1.5,-1.5) { };
		\coordinate (63) at (2.5,-2.5) { };
		\node[shape=circle,fill=black, scale = 0.4] (64) at (3,-3) { };
		\node[shape=circle,fill=black, scale = 0.4] (81) at (-1,-1) { };
		\coordinate (82) at (-1.5,-1.5) { };
		\coordinate (83) at (-2.5,-2.5) { };
		\node[shape=circle,fill=black, scale = 0.4] (84) at (-3,-3) { };

		\node[draw=none] (B0) at (-0.4,0) {$ w_0 $};
		\node[draw=none] (B01) at (0,-1.5) {$\cdots\cdots$};
		\node[draw=none] (B02) at (0,-2.5) {$\cdots\cdots$};
		\node[draw=none] (B1) at (-1.6,1) {$ w_{1, 1} $};
		\node[draw=none] (B12) at (-3.6,3) {$ w_{1, r_1} $};
		\node[draw=none] (B21) at (0.5,1.4) {$ w_{2, 1} $};
		\node[draw=none] (B22) at (0.5,4.2) {$ w_{2, r_2} $};
		\node[draw=none] (B31) at (1.6,1) {$ w_{3, 1} $};
		\node[draw=none] (B32) at (3.6,3) {$ w_{3, r_3} $};
		\node[draw=none] (B41) at (1.4,0.4) {$ w_{4, 1} $};
		\node[draw=none] (B42) at (4.2,0.4) {$ w_{4, r_4} $};
		\node[draw=none] (B51) at (1.5,-1) {$ w_{5, 1} $};
		\node[draw=none] (B52) at (3.5,-3) {$ w_{5, r_5} $};
		\node[draw=none] (B61) at (-1.6,-1) {$ w_{N, 1} $};
		\node[draw=none] (B62) at (-3.6,-3) {$ w_{N, r_N} $};

		\path [-](1) edge node[left] {} (21);
		\path [-](21) edge node[left] {} (22);
		\path [dashed](22) edge node[left] {} (23);
		\path [-](23) edge node[left] {} (24);
		
		\path [-](1) edge node[left] {} (31);
		\path [-](31) edge node[left] {} (32);
		\path [dashed](32) edge node[left] {} (33);
		\path [-](33) edge node[left] {} (34);
		
		\path [-](1) edge node[left] {} (41);
		\path [-](41) edge node[left] {} (42);
		\path [dashed](42) edge node[left] {} (43);
		\path [-](43) edge node[left] {} (44);
		
		\path [-](1) edge node[left] {} (51);
		\path [-](51) edge node[left] {} (52);
		\path [dashed](52) edge node[left] {} (53);
		\path [-](53) edge node[left] {} (54);
		
		\path [-](1) edge node[left] {} (61);
		\path [-](61) edge node[left] {} (62);
		\path [dashed](62) edge node[left] {} (63);
		\path [-](63) edge node[left] {} (64);
		
		
		\path [-](1) edge node[left] {} (81);
		\path [-](81) edge node[left] {} (82);
		\path [dashed](82) edge node[left] {} (83);
		\path [-](83) edge node[left] {} (84);
	\end{tikzpicture}
	\caption{A star-shaped plumbing graph. }
	\label{fig:Seifert}
\end{figure}

In this case, $ \calS $ in \cref{cor:WRT_asymp} is expressed as follows.

\begin{lem}
	We have
	\[
	\calS
	= 
	\{ 0 \} \cup
	\left\{ 
	- \frac{n^2}{4} \left( \prod_{\beta \text{ branches}} \frac{1}{p_\beta} \right) \bmod \Z 
	\relmiddle| 
	\text{$ n \in \Z $ such that $ \# \{ \beta \text{ branches} \mid n / p_\beta \in \Z \} \le \abs{\overline{v_0}} - 3 $}
	\right\}.
	\]
\end{lem}

This expression coincides with the expression of $ \CS_\bbC^{} (M) $ in Andersen--Misteg\aa{}rd~\cite[Corollary 9]{Andersen-Mistegard}.
Thus, \cref{conj:CS_inv_expression} holds for Seifert homology spheres.

\begin{proof}
	Let $ \{ v_0 \} \coloneqq V^{\veemidvert} $. 
	We have
	\[
	\calS =
	\{ 0 \} \cup
	\left\{ \frac{1}{4} \transpose{n} W^{\veemidvert} n \bmod \Z \relmiddle| 
	\text{$ n \in \Z $ such that $ \# \{ \beta \in \overline{v_0} \mid n / p_\beta \in \Z \} \le \abs{\overline{v_0}} - 3 $}
	\right\}.
	\]
	Since $ W^{\veemidvert} = w_{v_0}^{\veemidvert} $, it suffices to show 
	$ -1/w_{v_0}^{\veemidvert} = \prod_{\beta \in \overline{v_0}} \abs{p_\beta} $.
	By \cref{lem:linking_matrix_inverse}, we have
	\[
	w_{v_0}^{\veemidvert}
	=
	\det W
	\prod_{\beta \in \overline{v_0}} \frac{1}{p_\beta}.
	\]
	Since $ W $ is negative definite, the diagonal component $ 1/ w_{v_0}^{\veemidvert} $ of $ W^{-1} $ is negative.
	Thus, we obtain the claim.
\end{proof}


\subsubsection{Case of H-graphs} \label{subsubsec:CS_examples_H-graph}


We assume that $ M $ is a negative definite plumbed manifold determined by plumbing graphs shown in \cref{fig:H-graph_revisited} (H-graph) and its linking matrix $ W $ is negative definite and satisfies $ \det W = 1 $.
By \cref{lem:V_minimal}, we can assume $ w_3, \dots, w_6 \neq -1 $ without loss of generality. 

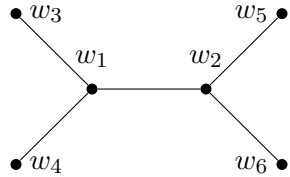
\begin{figure}[htb]
	\centering
	\begin{tikzpicture}
		\node[shape=circle,fill=black, scale = 0.4] (1) at (0,0) { };
		\node[shape=circle,fill=black, scale = 0.4] (2) at (1.5,0) { };
		\node[shape=circle,fill=black, scale = 0.4] (3) at (-1,-1) { };
		\node[shape=circle,fill=black, scale = 0.4] (4) at (-1,1) { };
		\node[shape=circle,fill=black, scale = 0.4] (5) at (2.5,1) { };
		\node[shape=circle,fill=black, scale = 0.4] (6) at (2.5,-1) { };
		
		\node[draw=none] (B1) at (0,0.4) {$ w_1 $};
		\node[draw=none] (B2) at (1.5, 0.4) {$ w_2 $};
		\node[draw=none] (B3) at (-0.6,1) {$ w_3 $};
		\node[draw=none] (B4) at (-0.6,-1) {$ w_4 $};
		\node[draw=none] (B5) at (2.1,1) {$ w_5 $};		
		\node[draw=none] (B6) at (2.1,-1) {$ w_6 $};	
		
		\path [-](1) edge node[left] {} (2);
		\path [-](1) edge node[left] {} (3);
		\path [-](1) edge node[left] {} (4);
		\path [-](2) edge node[left] {} (5);
		\path [-](2) edge node[left] {} (6);
	\end{tikzpicture}
	\caption{An H-graph.} \label{fig:H-graph_revisited}
\end{figure}

In this case, we can express $ \calS $ as follows.

\begin{lem}
	We have
	\begin{align}
		\calS =
		\{ 0 \} &\cup
		\left\{
		\frac{1}{4} \transpose{n} W^{\veemidvert} n \bmod \Z
		\relmiddle|
		n \in \left( \Z \smallsetminus (w_3 \Z \cup w_4 \Z) \right) \times \left( \Z \smallsetminus (w_5 \Z \cup w_6 \Z) \right)
		\right\}
		\\
		&\cup
		\left\{
		\frac{n_2^2}{w_1^{\veemidvert} P_1 P_2} \bmod \Z
		\relmiddle|
		n_2 \in \Z \smallsetminus (w_5 \Z \cup w_6)
		\right\}
		\\
		&\cup
		\left\{
		\frac{n_1^2}{w_2^{\veemidvert} P_1 P_2} \bmod \Z
		\relmiddle|
		n_1 \in \left( \Z \smallsetminus (w_3 \Z \cup w_4 \Z) \right)
		\right\}.
	\end{align}
\end{lem}

\begin{proof}
	We have
	\[
	W^{\veemidvert} =
	\pmat{w_1^{\veemidvert} & 1 \\ 1 & w_2^{\veemidvert}},
	\quad
	w_1^{\veemidvert} = w_1 - \frac{1}{w_3} - \frac{1}{w_4},
	\quad
	w_2^{\veemidvert} = w_2 - \frac{1}{w_5} - \frac{1}{w_6},
	\quad
	P_1 = w_3 w_4,
	\quad
	P_2 = w_5 w_6.
	\]
	Since $ \deg (v_1) = \deg (v_2) = 3 $, for $ j = 1, 2 $ we have
	\[
	\calP_{j}^{(0)}
	=
	\Z \smallsetminus (w_{2j+1} \Z \cup w_{2j+2} \Z).
	\]
	Moreover, since $ \det W^{\veemidvert} = 1/P_1 P_2 $ by \cref{lem:linking_matrix_inverse}, we have
	\[
	(W^{\veemidvert})^{-1} =
	P_1 P_2 \pmat{w_2^{\veemidvert} & -1 \\ -1 & w_1^{\veemidvert}}.
	\]
	Thus, we obtain the claim.
\end{proof}


\appendix
\def\thepart{\Alph{part}}

\part*{Appendix}



\section{Treatment of signed sums} \label{sec:sgn}


The following proposition implies that the sum involving $ \sgn(Am) $ can be reduced to a sum involving only $ \sgn(m) $.

\begin{prop} \label{prop:sign_Am_to_m}
	Let $ r $ and $ s $ be positive integers and $ A \in \Mat_{s, r} (\Z) $ be a matrix of rank $ d $.
	Then, for each $ 1 \le d' \le d $, there exists a rational number $ a_{d'} \in \Q $, a matrix $ B_{d'} \in \Mat_r (\Z) \cap \GL_r (\Q) $, and
	a free $ \Z $-submodule $ \Z^{d'} \subset L_{d'} \subset \Q^{d'} $ such that
	$ B_{d'} $ induces an isomorphism $ L_{d'} \oplus \Z^{r-d'} \xrightarrow{\sim} \Z^r $ and for any $ F \colon \Z^r \to \bbC $, we have
	\[
	\sum_{m \in \Z^r} \sgn(Am) F(m)
	=
	\sum_{1 \le d' \le d} a_{d'}
	\sum_{m =(m', m'') \in L_{d'} \oplus \Z^{r-d'}} \sgn(m') F \left( Bm \right)
	\]
	if the left hand side converges.
\end{prop}

To prove this proposition, we prepare the following three statements.

\begin{lem} \label{lem:sign_Am_to_m}
	Let $ r \ge d \ge 1 $ be positive integers and $ A \in \Mat_{d, r} (\Z) $ be a matrix of rank $ d $.
	Then, there exists a matrix $ B \in \Mat_r (\Z) \cap \GL_r (\Q) $ such that
	\[
	A B = \pmat{ D(A) I_d & O_{d, r-d} },
	\]
	where $ D(A) \coloneqq \abs{\Z^d / A(\Z^r)} $, $ I_d \in \GL_d (\Z) $ be the identity matrix, and $ O_{d, r-d} \in \Mat_{d, r-d} (\Z) $ be the zero matrix.
	
	Moreover, for any $ F \colon \Z^r \to \bbC $, we have
	\[
	\sum_{m \in \Z^r} \sgn(Am) F(m)
	=
	\sum_{\mu \in D(A)^{-1} A(\Z^r) / \Z^d} 
	\sum_{m =(m', m'') \in (\mu + \Z^d) \oplus \Z^{r-d'}} \sgn(m') F \left( Bm \right)
	\]
	if the left hand side converges.
\end{lem}

\begin{proof}
	By the invariant factor theorem, there exist matrices $ P \in \Mat_d (\Z) $ and $ Q \in \GL_r (\Z) $ such that
	\[
	A = \pmat{ P & O_{d, r-d} } Q, \quad
	D(A) = \abs{\det P}.
	\]
	Thus, the first statement holds for
	\[
	B \coloneqq
	Q^{-1} \pmat{ D(A) P^{-1} &  \\  & I_{r-d} }.
	\]
	The last statement follows by replacing $ m \in \Z^r $ by $ B \smat{m' \\ m''} $ with 
	$ m' \in (\det P)^{-1} P (\Z^d) = D(A)^{-1} A(\Z^r) $ and $ m'' \in \Z^{r-d} $.
\end{proof}

The following lemma also follows from the invariant factor theorem.
We omit a proof.

\begin{lem} \label{lem:sign_Am_to_m_rank_less}
	Let $ r \ge s \ge 1 $ be positive integers and $ A \in \Mat_{d, r} (\Z) $ be a matrix of rank $ d < s $.
	Then, there exists a matrix $ A' \in \Mat_{s, d} (\Z) $ of rank $ d $ and a matrix $ B \in \GL_r (\Z) $ such that
	\[
	A = \pmat{ A'& O_{d, r-d} } B.
	\]
	
	Moreover, for any $ F \colon \Z^r \to \bbC $, we have
	\[
	\sum_{m \in \Z^r} \sgn(Am) F(m)
	=
	\sum_{m =(m', m'') \in \Z^{d'} \oplus \Z^{r-d'}} \sgn(A' m') F \left( Bm \right)
	\]
	if the left hand side converges.
\end{lem}

The following proposition implies that whenever the number of sign products appearing in the summand exceeds the rank, it can be reduced to a number not exceeding the rank.

\begin{prop} \label{prop:signs_more_than_rank}
	For any vectors $ a_1, \dots, a_s \in \R^r $ with $ s > r $, we have
	\[
	\sgn(\transpose{a_1} x) \cdots \sgn(\transpose{a_s} x)
	\in
	\sprod{
		\prod_{i \in I} \sgn(\transpose{a_i} x)
		\relmiddle|
		I \subset \{ 1, \dots, s \}, \abs{I} \le r
	}_\Q.
	\]
\end{prop}

\begin{ex}
	In the case when $ r=2 $, for any $ a_1, a_2 > 0 $ and $ x_1, x_2 \in \R $, we have
	\[
	\sgn(x_1) \sgn(x_2) \sgn(a_1 x_1 + a_2 x_2)
	=
	\sgn(x_1) + \sgn(x_2) - \sgn(a_1 x_1 + a_2 x_2).
	\]
\end{ex}

Before proving this proposition, we complete the proof of Theorem 1.

\begin{proof}[Proof of \cref{prop:sign_Am_to_m}]
	For a matrix $ A \in \Mat_{s, r} (\Z) $ be a matrix of rank $ d $, we apply the following operations:
	\begin{enumerate}[label=(\arabic*)]
		\item If $ s>r $, then we apply \cref{prop:signs_more_than_rank} and reduce to the case where $ s \le r $.
		\item[(2a)] If $ d < s \le r $, then we apply \cref{lem:sign_Am_to_m_rank_less} and reduce to the case where $ s \le r = d $.
		\item[(2b)] If $ d=s \le r $, then we apply \cref{lem:sign_Am_to_m} and finish operations.
	\end{enumerate}
	Since these operations reduce the size of the matrix $ A $, the claim follows by induction.
\end{proof}

\begin{proof}[Proof of \cref{prop:signs_more_than_rank}]
	We may assume $ s = r+1 $.
	We introduce the following notation:
	\begin{itemize}
		\item For each $ 1 \le i \le r+1 $, let
		$ H_i \coloneqq \{ x \in \R^r \mid \sgn(\transpose{a_i} x) \} $ be the hyperplane.
		\item Let $ \calC $ be the set of chambers of the arrangement of hyperplanes $ \{ H_1, \dots, H_{r+1} \} $, that is,
		the set of connected components of $ \R^r \smallsetminus H_1 \cup \cdots \cup H_{r+1} $.
		\item Let $ G \coloneqq \{ \pm 1 \}^{r+1} $.
		\item Let $ \sigma \colon \calC \hookrightarrow G $ be the injection defined by 
		$ C \mapsto (\sgn(\sgn(\transpose{a_i} x)))_{1 \le i \le r+1} $.
		\item Let $ \widehat{G} \coloneqq \{ \text{ group homomorphisms } \chi \colon G \to \bbC^\times \} $.
		\item For each $ I \subset \{ 1, \dots, r+1 \} $, define $ \chi_I^{} \in \widehat{G} $ by
		$ e \mapsto \prod_{i \in I} e_i $.
	\end{itemize}
	Then, for each $ I \subset \{ 1, \dots, r+1 \} $, we can regard the function $ \prod_{i \in I} \sgn(\transpose{a_i} x) $ as a map from $ \calC $ to $ \Q $ and it coincides with $ \chi_I^{} \in \widehat{G} $.
	Thus, it suffices to show that
	\begin{equation} \label{eq:proof:signs_more_than_rank}
		\{ f \colon \calC \to \Q \}
		=
		\sprod{ \chi_I^{} \circ \sigma \mid I \subsetneq \{ 1, \dots, r+1 \} }_\Q.
	\end{equation}
	We shall show that $ \abs{\calC} < \abs{G} $ implies this equality. 
	Indeed, take an arbitrary $ f \colon \calC \to \Q $. 
	Then, by the assumption, the space of possible extensions $ F \colon G \to \Q $ of $ f $ is at least one-dimensional.
	Define the Fourier transform $ \widehat{F} \colon \widehat{G} \to \bbC $ of $ F $ by
	$ \chi \mapsto \sum_{e \in G} F(e) \chi(e) $.
	By the Fourier inversion formula for the discrete Fourier transform, for each $ e \in G $, we have
	\[
	F(e) = \frac{1}{\abs{G}} \sum_{\chi \in \widehat{G}} \widehat{F}(\chi) \chi(e).
	\]
	Here, we have $ \widehat{G} = \{ \chi_I^{} \mid I \subset \{ 1, \dots, r+1 \} \} $.
	Taking $ F $ so as to satisfy $ \widehat{F}(\chi_{ \{ 1, \dots, r+1 \} }^{}) = 0 $, we obtain \cref{eq:proof:signs_more_than_rank}.
	Finally, $ \abs{\calC} < \abs{G} = 2^{r+1} $ follows from the following lemma.
\end{proof}

\begin{lem} \label{lem:hypplane_arrangement}
	The number of chambers of an arrangement of $ m $ hyperplanes in general position in $ \R^r $ is
	\[
	2 \sum_{0 \le l \le r-1} \binom{m-1}{l}.
	\]
	In particular, in the case when $ m=r+1 $, this number equals to $ 2^{r+1} - 2 $.
\end{lem}

\begin{proof}
	Let $ N(r, m) $ be the number of chambers.
	It suffices to show that it satisfies the recurrence relations $ N(r, 1) = N(1, m) = 2 $ and
	$ N(r, m+1) = N(r, m) + N(r-1, m) $ since the formula in the statement satisfies the same recurrence relations.
	The first one is clear. We now prove the second.
	Let $ \calA $ be an arrangement of $ m $ hyperplanes in general position in $ \R^r $ and $ H $ be an other hyperplane.
	Let $ \calC $ be the set of chambers of $ \calA $.
	Further dividing each chamber $C \in \mathcal{C}$ by $ H $, we see that:
	\begin{itemize}
		\item If $ C \cap H = \emptyset $, then $ C $ remains unchanged.
		\item If $ C \cap H \neq \emptyset $, then $ C $ is divided into two parts by $ H $.
	\end{itemize}
	Thus, we have
	\[
	N(r, m+1) = N(r, m) + \# \{ C \in \calC \mid C \cap H \neq \emptyset \}.
	\]
	The last term equals to $ N(r-1, m) $ since	a map
	\[
	\begin{array}{ccc}
		\{ C \in \calC \mid C \cap H \neq \emptyset \} & \longrightarrow & \{ \text{ chambers of $ \eval{\calA}_H $ } \} \\
	\end{array}
	\]
	is bijective, where $ \eval{\calA}_H \coloneqq \{ H' \cap H \mid H' \in \calA \} $ is the hyperplane arrangement of $ H $ determined by $ \calA $.
\end{proof}

We can express explicitly the difference $ G \smallsetminus \sigma(\calC) $ in the proof of \cref{prop:signs_more_than_rank}.
To achieve this, we need the following statement known as ``Gordan's theorem.''

\begin{prop}[{Gordan's theorem, \cite[Theorem 2.2.1]{Borwein-Lewis}}] \label{prop:Gordan}
	For any matrix $ A \in \Mat_{s, r} (\R) $, exactly one of the followings holds:
	\begin{enumerate}
		\item There exists $ y \in \R_{\ge 0}^s \smallsetminus \{ 0 \} $ such that $ \transpose{A} y = 0 $.
		\item There exists $ x \in \R^r $ such that $ Ax \in \R_{>0}^s $.
	\end{enumerate}
\end{prop}

\begin{rem}
	In the proof of \cref{prop:signs_more_than_rank}, let $ A \coloneqq (a_1, \dots, a_{r+1}) \in \Mat_{r+1, r} (\R) $ be the matrix
	and we assume that $ A $ has rank $ r $.
	We may set $ \{ \pm e_A \} \coloneqq \{ (\sgn(y_1), \dots, \sgn(y_{r+1})) \mid y \in \Ker A \} $.
	Under these settings, we show that $ G \smallsetminus \sigma(\calC) = \{ \pm e_A \} $.
	
	For $ e \in G = \{ \pm 1 \}^{r+1} $, let
	\[
	C_e \coloneqq
	\left\{ 
	x \in \R^r
	\relmiddle|
	e_i \transpose{a_i} x > 0 \text{ for each } 1 \le i \le r+1
	\right\}
	\in \calC \cup \{ \emptyset \}.
	\]
	Then, the condition $ C_e \neq \emptyset $ holds if and only if
	there exists $ x \in \R^r $ such that $ A_e x \in \R_{>0}^r $, where
	$ A_e \coloneqq \diag{(e_1, \dots, e_r)} A $.
	By Gordan's theorem (\cref{prop:Gordan}), this does not hold if and only if
	there exists $ y \in \R_{\ge 0}^{r+1} \smallsetminus \{ 0 \} $ such that $ \transpose{A_e} y = 0 $.
	Thus, the condition $ C_e = \emptyset $ holds if and only if
	there exists $ y \in \Ker \transpose{A} \smallsetminus \{ 0 \} $ such that $ e_i y_i > 0 $ for each $ 1 \le i \le r+1 $,
	that is, $ e \in \{ \pm e_A \} $.
\end{rem}


\bibliographystyle{alpha}
\bibliography{quantum_inv,quantum_inv_jp,modular}

\end{document}